\documentclass[12pt]{amsart} %article}
\usepackage{fullpage, here}
\usepackage{amsmath}
\usepackage{amssymb}
\usepackage{amsthm}
\usepackage{graphicx}
\usepackage{amscd}
\usepackage{enumitem}
\usepackage{mathtools}
\usepackage{comment}
\usepackage{mathtools}
\usepackage{multirow}
\usepackage[all,knot,poly]{xy}
\usepackage[hyphens]{url}

\usepackage{bm}
\usepackage[nameinlink]{cleveref}

\usepackage{longtable}
\usepackage[format=plain, font=normalsize]{caption} 

\usepackage[toc=false, symbols, nogroupskip, sort=use]{glossaries-extra}
\usepackage{glossary-longbooktabs}
\makenoidxglossaries
\newglossarystyle{formel_altlong4colheader}{%
 \setglossarystyle{altlong4colheader}%
}
\glsxtrRevertMarks 

\glsxtrnewsymbol[description={a fixed prime number}]{p}{\ensuremath{p}}
\glsxtrnewsymbol[description={a fixed positive integer}]{d}{\ensuremath{d}}
\glsxtrnewsymbol[description={the fundamental group of a space $X$}]{pi1}{\ensuremath{\pi_1(X)}}
\glsxtrnewsymbol[description={the pro-$p$ completion of $\pi_1(X)$}]{whpi1}{\ensuremath{\widehat{\pi}_1(X)}}
\glsxtrnewsymbol[description={the $R$-torsion submodule of an $R$-module $A$, often with $R=\Z,\Zp$}]{tor}{\ensuremath{{\rm tor}_R(A)}}
\glsxtrnewsymbol[description={the ideal of $\Lambda_{\mathbb{Z}}$ generated by $\{t_i^n-1\mid 1\leq i\leq d\}$}]{in}{\ensuremath{I_{n}}}

\glsxtrnewsymbol[description={$H_1(X_{\infty})/I_{p^n} H_1(X_{\infty})$}]{sfHpn}{\ensuremath{\mathsf{H}_{p^n}}}
\glsxtrnewsymbol[description={the ring of $p$-adic integers}]{zp}{\ensuremath{\mathbb{Z}_p}}
\glsxtrnewsymbol[description={the free $\mathbb{Z}_p$-rank of a $\Zp$-module $\mathcal{M}$}]{zprk}{\ensuremath{r(\mathcal{M})}}
\glsxtrnewsymbol[description={the $p$-exponent of the size of the $p$-torsion subgroup of an abelian group $G$}]{exponential}{\ensuremath{e(G)}}
\glsxtrnewsymbol[description={a rational homology $3$-sphere}]{qhs3}{\ensuremath{\mathbb{Q}\mbox{HS}^3}}
\glsxtrnewsymbol[description={an integral homology $3$-sphere}]{zhs3}{\ensuremath{\mathbb{Z}\mbox{HS}^3}}
\glsxtrnewsymbol[description={a tubular neighborhood $V_L$ of $L$}]{VL}{\ensuremath{V_L}}
\glsxtrnewsymbol[description={the interior of a tubular neighborhood $V_L$ of $L$}]{IntVL}{\ensuremath{{\rm Int}\, V_L}}
\glsxtrnewsymbol[description={$\varprojlim H_1(X_{p^n};\mathbb{Z}_p)$}]{projectivelimitmodule}{\ensuremath{\mathcal{H}}}
\glsxtrnewsymbol[description={Bachmann--Landau notation}]{landauo}{\ensuremath{O}}
\glsxtrnewsymbol[description={the ideal of $\Lambda$ topologically generated by $\{(1+T_i)^{p^n}-1\mid 1\leq i\leq d\}$}]{ipn}{\ensuremath{\mathcal{I}_{p^n}}}
\glsxtrnewsymbol[description={the ideal of $\Lambda$ topologically generated by ${\{\frac{((1+T_i)^{p^n}-1)}{T_i} \mid 1\leq i\leq d}\}$}]{jpn}{\ensuremath{\mathcal{J}_{p^n}}}
\glsxtrnewsymbol[description={the characteristic element of a $\Lambda$-module $\mathcal{M}$}]{char}{\ensuremath{{\rm char}\,\mathcal{M}}}
\glsxtrnewsymbol[description={the characteristic ideal of a $\Lambda$-module $\mathcal{M}$}]{Char}{\ensuremath{{\rm Char}\,\mathcal{M}}}
\glsxtrnewsymbol[description={Iwasawa $\mu$-invariant}]{iwasawamu}{\ensuremath{\mu}}
\glsxtrnewsymbol[description={Iwasawa $\lambda$-invariant}]{iwasawalambda}{\ensuremath{\lambda}}
\glsxtrnewsymbol[description={equal up to multiplication by a unit}]{equaluptomultiplicationbyunit}{\ensuremath{\doteq}}
\glsxtrnewsymbol[description={free abelianization}]{fab}{\ensuremath{{\rm fab}}}
\glsxtrnewsymbol[description={Alexander polynomial $\Delta_{\rm fab}(t_1,\ldots, t_c)$ of $L$}]{alexander}{\ensuremath{\Delta_L(t_1,\ldots,t_c)}}
\glsxtrnewsymbol[description={Alexander polynomial of $\tau:X_{\infty}\to X$}]{alexandertau}{\ensuremath{\Delta_{\tau}(t_1,\ldots,t_d)}}
\glsxtrnewsymbol[description={Iwasawa $\mu$-invariant of $0\neq F\in \Lambda$}]{muf}{\ensuremath{\mu(F)}}
\glsxtrnewsymbol[description={Iwasawa $\lambda$-invariant of $0\neq F\in \Lambda$}]{lambdaf}{\ensuremath{\lambda(F)}}
\glsxtrnewsymbol[description={an algebraic closure of $\Q_p$}]{algebraicclosure}{\ensuremath{\overline{\mathbb{Q}}_p}}
\glsxtrnewsymbol[description={$\{\xi\in\overline{\mathbb{Q}}\mid \xi^{p^n}=1\mbox{ for some }n\geq 0\}$ with a fixed embedding $\overline{\mathbb{Q}}\inj \overline{\mathbb{Q}}_p$
}]{W}{\ensuremath{W}}
\glsxtrnewsymbol[description={$F(\zeta_1-1,\ldots,\zeta_d-1)$}]{fzeta}{\ensuremath{F(\zeta-1)}}
\glsxtrnewsymbol[description={a Noetherian integrally closed domain}]{ring}{\ensuremath{R}}
\glsxtrnewsymbol[description={an $R$-module}]{module}{\ensuremath{\mathcal{M}}}
\glsxtrnewsymbol[description={a closed connected orientable 3-manifold}]{closedmanifold}{\ensuremath{M}}
\glsxtrnewsymbol[description={a compact connected orientable 3-manifold, often a link exterior in $M$}]{cptmanifold}{\ensuremath{X}}
\glsxtrnewsymbol[description={the maximal unramified abelian extension of $k$}]{hilbertclassfield}{\ensuremath{k_{\rm ab}^{rm ur}}}
\glsxtrnewsymbol[description={$m$-twisted Whitehead link}]{whiteheadlink}{\ensuremath{W_m}}
\glsxtrnewsymbol[description={the $i$-th homology group with integral coefficient of a space $M$}]{homology}{\ensuremath{H_i(M)=H_i(M;\mathbb{Z})}}
\glsxtrnewsymbol[description={a branched $\mathbb{Z}_p^d$-cover}]{branched}{\ensuremath{(M_{p^n}\to M)_n}}
\glsxtrnewsymbol[description={a knot}]{knot}{\ensuremath{K}}
\glsxtrnewsymbol[description={a link}]{link}{\ensuremath{L}}
\glsxtrnewsymbol[description={the $1$st Betti number of a space $M$}]{betti}{\ensuremath{\beta_1(M)}}
\glsxtrnewsymbol[description={an unbranched $\mathbb{Z}_p^d$-cover}]{unbranched}{\ensuremath{(X_{p^n}\to X)_n}}
\glsxtrnewsymbol[description={the Conway potential function of a link $L$ in $S^3$}]{conway}{\ensuremath{\nabla}}
\glsxtrnewsymbol[description={the $n$-th cyclotomic polynomial}]{cyclotomic}{\ensuremath{\Phi_n(x)}}
\glsxtrnewsymbol[description={a non-zero element of $\Lambda$}]{F}{\ensuremath{F}}
\glsxtrnewsymbol[description={the number of components of a link}]{c}{\ensuremath{c}}

\glsxtrnewsymbol[description={the Fitting ideal of $\mathcal{M}$}]{fitt}{\ensuremath{{\rm Fitt}\, \mathcal{M}}}

\glsxtrnewsymbol[description={the divisorial hull of a fractional ideal $\mathfrak{a}$ of $R$}]{dh}{\ensuremath{{\rm d.h.}\mathfrak{a}}}

\glsxtrnewsymbol[description={the localization of $\mathcal{M}$ at a prime ideal $\mathfrak{p}$ of $R$}]{localization}{\ensuremath{\mathcal{M}_{\mathfrak{p}}}}

\glsxtrnewsymbol[description={the annihilator of an $R$-module $\mathcal{M}$}]{annihilator}{\ensuremath{{\rm Ann}\,\mathcal{M}}}

\glsxtrnewsymbol[description={a maximal ideal of $R$}]{maximalideal}{\ensuremath{\mathfrak{m}}}

\glsxtrnewsymbol[description={$\mathfrak{m}$-completion of $R$}]{rcompletion}{\ensuremath{\widehat{R}}}

\glsxtrnewsymbol[description={the completed module defined by $\varprojlim_n\mathcal{M}\otimes R/\mathfrak{m}^n$}]{mcompletion}{\ensuremath{\widehat{\mathcal{M}}}}

\glsxtrnewsymbol[description={the ring of integers of a number field $k$}]{ringofintegers}{\ensuremath{\mathcal{O}_k}}

\glsxtrnewsymbol[description={$(1+T_j)^{p^n}-1$}]{omegapoly}{\ensuremath{\omega_{p^n}(T_j)}}

\glsxtrnewsymbol[description={$\frac{(1+T_j)^{p^n}-1}{T_j}$}]{nupoly}{\ensuremath{\nu_{p^n}(T_j)}}

\glsxtrnewsymbol[description={$\sum_{\zeta\in W(n)^d} v(F(\zeta-1))$}]{sum}{\ensuremath{\Sigma_n(F)}}

\glsxtrnewsymbol[description={$\{\xi\in W\mid \xi^{p^n}=1\}$}]{Wn}{\ensuremath{W(n)}}

\glsxtrnewsymbol[description={$\mathbb{Z}_p[\zeta_1,\ldots,\zeta_d]$ for $\zeta=(\zeta_1,\ldots,\zeta_d)\in(W\setminus \{1\})^d$}]{zpzeta}{\ensuremath{\mathbb{Z}_p[\zeta]}}

\glsxtrnewsymbol[description={${\rm rank}_{\mathbb{Z}_p}\mathcal{M}_{\zeta}$}]{rzeta}{\ensuremath{r_{\zeta}(\mathcal{M}_{\zeta})}}

\glsxtrnewsymbol[description={$\mathcal{M}\otimes_{\Lambda}\mathbb{Z}_p[\zeta]$}]{mzeta}{\ensuremath{\mathcal{M}_{\zeta}}}

\glsxtrnewsymbol[description={$\{\zeta\in(W\setminus\{1\})^d\mid r_{\zeta}(\mathcal{M}_{\zeta})\geq 1\}$}]{zm}{\ensuremath{Z(\mathcal{M})}}

\glsxtrnewsymbol[description={$Z(\mathcal{M})\cap(W(n)\setminus\{1\})^d$}]{znm}{\ensuremath{Z_n(\mathcal{M})}}

\glsxtrnewsymbol[description={$\{\xi-1\in\overline{\mathbb{Q}}_p\mid \xi\in W\}$}]{mcaW}{\ensuremath{\mathcal{W}}}

\glsxtrnewsymbol[description={the linking number}]{lk}{\ensuremath{{\rm lk}}}

\glsxtrnewsymbol[description={$\mathbb{Z}_p[\![T_1,\ldots, T_d]\!]$}]{Lambda}{\ensuremath{\Lambda_d=\Lambda}}

\glsxtrnewsymbol[description={$\mathbb{Z}[t_1^{\pm 1},\ldots, t_d^{\pm 1}]=\mathbb{Z}[t_1^{\mathbb{Z}},\ldots, t_d^{\mathbb{Z}}]$}]{Lambdaz}{\ensuremath{\Lambda_{\mathbb{Z},d}=\Lambda_{\mathbb{Z}}}}

\glsxtrnewsymbol[description={the resultant of two polynomials $f,g$}]{resultant}{\ensuremath{{\rm Res}(f,g)}}

\glsxtrnewsymbol[description={Euler's totient function}]{euler}{\ensuremath{\bm{\varphi}(n)}}

\glsxtrnewsymbol[description={multiplicative $p$-adic valuation normalized so that $v(p)=1$}]{val}{\ensuremath{v:\overline{\mathbb{Q}}_p \to \mathbb{Q}}}

\glsxtrnewsymbol[description={the discrete valuation associated with a height 1 prime ideal $\mathfrak{p}$ of $\Omega$}]{valp}{\ensuremath{v_{\mathfrak{p}}}}

\glsxtrnewsymbol[description={$\mathbb{F}_p[\![T_1,\ldots,T_d]\!]$}]{Omega}{\ensuremath{\Omega}}

\glsxtrnewsymbol[description={a free $\mathbb{Z}_p$-module of rank $d$}]{Gamma}{\ensuremath{\Gamma}}

\glsxtrnewsymbol[description={the complete group ring of $\Gamma$ over $\mathbb{Z}_p$}]{groupring}{\ensuremath{\mathbb{Z}_p[\![\Gamma]\!]}}

\glsxtrnewsymbol[description={$(t-1)\widetilde{\Delta}_L(t)$}]{AXinf}{\ensuremath{A_{X_{\infty}}(t)}}

\glsxtrnewsymbol[description={the reduced Alexander polynomial $\Delta_L(t,\ldots,t)$}]{Alexandertilde}{\ensuremath{\widetilde{\Delta}_L}(t)}

\glsxtrnewsymbol[description={the $n$-th cyclotomic polynomial}]{cyclopoly}{\ensuremath{\bm{\Phi}_n(x)}}

\glsxtrnewsymbol[description={$\varprojlim H_1(M_{p^n};\mathbb{Z}_p)$}]{HM}{\ensuremath{\mathcal{H}_M}}

\glsxtrnewsymbol[description={the abelianization of a group $G$}]{groupabelianization}{\ensuremath{G^{\rm ab}}}

\glsxtrnewsymbol[description={$\max\{\deg g_i\mid 0\leq i\leq N'-1\}$}]{dml}{\ensuremath{D(M,L)}}

\glsxtrnewsymbol[description={$F\mapsto (F(\zeta-1))_\zeta$, 
where $\zeta$ runs through a complete representative system of 
$(W(n)\setminus\{1\})^d$ modulo the action of ${\rm Gal}(\ol{\Q}_p/\Qp)$ 
}]{varphi}{\ensuremath{\varphi_n:\Lambda/\mathcal{J}_{p^n}\Lambda\to\bigoplus_\zeta \mathbb{Z}_p[\zeta]}}

\glsxtrnewsymbol[description={the natural $\Lambda$-homomorphism}]{Phi}{\ensuremath{\ensuremath{\Phi_n:\mathcal{M}\to\bigoplus_{\zeta}\mathcal{M}}}}

\glsxtrnewsymbol[description={a tubular neighborhood of a knot $K$}]{vk}{\ensuremath{V_K}}

\glsxtrnewsymbol[description={the boundary of $V_K$}]{boundary}{\ensuremath{\partial V_K}}

\numberwithin{equation}{section}

\theoremstyle{definition}
\newtheorem{example}{Example}[section]
\newtheorem{definition}[example]{Definition}

\newtheorem{remark}[example]{Remark}
\newtheorem*{ack}{Acknowledgement} 

\theoremstyle{plain}
\newtheorem{prop}[example]{Proposition}
\newtheorem{cor}[example]{Corollary}
\newtheorem{lem}[example]{Lemma}
\newtheorem{theorem}[example]{Theorem}

\newtheorem{mainresult}{Main Result}

\DeclareMathOperator{\im}{im}

\DeclareMathOperator{\coker}{coker}

\DeclareMathOperator{\Gal}{Gal}
\DeclareMathOperator{\Hom}{Hom}

\DeclareMathOperator{\Spec}{Spec}
\DeclareMathOperator{\tor}{tor}
\DeclareMathOperator{\rank}{rank}
\DeclareMathOperator{\ab}{ab}
\DeclareMathOperator{\Res}{Res}
\DeclareMathOperator{\GL}{GL}
\DeclareMathOperator{\ur}{ur}

\usepackage[pagewise, mathlines]{lineno} %\linenumbers 
\usepackage{time} 

\numberwithin{equation}{section}
\usepackage{etoolbox}          %% <- for \cspreto, \csappto and \patchcmd

\newcommand*\linenomathpatch[1]{%
  \cspreto{#1}{\linenomath}%
  \cspreto{#1*}{\linenomath}%
  \csappto{end#1}{\endlinenomath}%
  \csappto{end#1*}{\endlinenomath}%
}
\newcommand*\linenomathpatchAMS[1]{%
  \cspreto{#1}{\linenomathAMS}%
  \cspreto{#1*}{\linenomathAMS}%
  \csappto{end#1}{\endlinenomath}%
  \csappto{end#1*}{\endlinenomath}%
}
\expandafter\ifx\linenomath\linenomathWithnumbers 
 \let\linenomathAMS\linenomathWithnumbers
\patchcmd\linenomathAMS{\advance\postdisplaypenalty\linenopenalty}{}{}{}
\else
  \let\linenomathAMS\linenomathNonumbers
\fi

\linenomathpatch{equation}
\linenomathpatchAMS{gather}
\linenomathpatchAMS{multline}
\linenomathpatchAMS{align}
\linenomathpatchAMS{alignat}
\linenomathpatchAMS{flalign}
\makeatletter
\patchcmd{\mmeasure@}{\measuring@true}{
  \measuring@true
  \ifnum-\linenopenaltypar>\interdisplaylinepenalty
    \advance\interdisplaylinepenalty-\linenopenalty
  \fi
  }{}{}
\makeatother

\newcommand{\mf}[1]{{\mathfrak{#1}}}

\newcommand{\bb}[1]{{\mathbb{#1}}}
\newcommand{\mca}[1]{{\mathcal{#1}}}

\newcommand{\inj}{\hookrightarrow}
\newcommand{\surj}{\twoheadrightarrow}

\newcommand{\congto}{\overset{\cong}{\to}}

\newcommand{\Z}{\bb{Z}}
\newcommand{\Zp}{\bb{Z}_{p}}
\newcommand{\Q}{\bb{Q}}
\newcommand{\Qp}{\bb{Q}_{p}}
\newcommand{\R}{\bb{R}}

\newcommand{\F}{\bb{F}}
\newcommand{\p}{\mf{p}}

\newcommand{\ol}{\overline}
\newcommand{\ul}{\underline}

\newcommand{\ds}{\displaystyle}

\newcommand{\wt}[1]{{\widetilde{#1}}}
\newcommand{\wh}[1]{{\widehat{#1}}}

\makeatletter
\@namedef{subjclassname@2020}{
\textup{2020} Mathematics Subject Classification}
\makeatother 

\subjclass[2020]{Primary 57M12, 11R23, 20E18; Secondary 57K10, 57M10, 11R29} 

\keywords{Iwasawa theory, knot, link, 3-manifold, Alexander polynomial, arithmetic topology}

\begin{document}

\title{\Large 
The Iwasawa invariants of $\Zp^{\,d}$-covers of links}
\dedicatory{\normalsize To the memory of John Coates}

\author{%\large 
Sohei Tateno$^{\ast 1}$}
\email{inu.kaimashita@gmail.com}
\address{$\,^{\ast 1}$ Graduate School of Mathematics, Nagoya University; Furocho, Chikusa-ku, 464-8602, Nagoya-shi, Aichi, Japan}

\author{%\large 
Jun Ueki$^{\ast 2}$} 
\email{uekijun46@gmail.com}
\address{$\,^{\ast 2}$ Department of Mathematics, Faculty of Science, Ochanomizu University; 2-1-1 Otsuka, Bunkyo-ku, 112-8610, Tokyo, Japan}

%\date{\today \ \now} 

\newenvironment{nouppercase}{%
\let\uppercase\relax%
\renewcommand{\uppercasenonmath}[1]{}}{}

\begin{nouppercase}
\maketitle 
\end{nouppercase}

\begin{abstract} 
Let \glssymbol{p}$p$ be a prime number and let \glssymbol{d}$d\in \Z_{>0}$. 
In this paper, following the analogy between knots and primes, we study the $p$-torsion growth in a compatible system of $(\Z/p^n\Z)^d$-covers of 3-manifolds and establish several analogues of Cuoco--Monsky's multivariable versions of Iwasawa's class number formula. 
Our main goal is to establish the Cuoco--Monsky-type formula for branched covers of links in rational homology 3-spheres. 
In addition, we prove a precise formula over integral homology 3-spheres inspired by Greenberg's conjecture. 
We also derive results on reduced Alexander polynomials and the periodicity of Betti numbers. 
Furthermore, we investigate the twisted Whitehead links in $S^3$ 
and point out that the Iwasawa $\mu$-invariant of a $\Zp^{\,2}$-cover can be an arbitrary non-negative integer. 
We also calculate the Iwasawa $\mu$ and $\lambda$-invariants of the Alexander polynomials of all links in Rolfsen's table.
\end{abstract}

{\small 
\tableofcontents
} 
\section{Introduction} 
Let $p$ be a prime number and let \glssymbol{zp}$\Zp$ denote the ring of $p$-adic integers defined by $\Zp:=\varprojlim \Z/p^n\Z$. 
It is known that there is a deep analogy between low-dimensional topology and number theory (cf.~\cite{Morishita2012, Morishita2024}). 
Among other things,  the analogy between the Alexander--Fox theory for $\Z$-covers and the Iwasawa theory for $\Zp$-extensions was initially pointed out by Mazur in the 1960s \cite{Mazur1963} and has played an important role until the present day. 
In addition, profinite refinements of the Alexander--Fox theory have also been of deep interest in recent studies of geometric topology 
\cite{BoileauFriedl2020AMS, YiLiu2023Invent}. %Reid2018ICM}. 
On the number theory side, 
Iwasawa \cite{Iwasawa1959} initially proved a celebrated theorem called \emph{Iwasawa's class number formula} 
that describes the regularity of the $p$-exponents of the class numbers in a $\Zp$-extension of a number field. 
Afterward, Cuoco--Monsky \cite{CuocoMonsky1981} established its multi-variable analogue for $\Zp^{\,d}$-extensions with $d \in \Z_{>1}$, and Monsky \cite{Monsky1989} gave further refinement (cf.~\Cref{ss.NT}). 
On the topology side, an analogue of Iwasawa's formula for $\Zp$-covers of knots and links (\Cref{thm.d=1}) has been developed by Morishita, the second author, and others \cite{HillmanMateiMorishita2006,KadokamiMizusawa2008,Ueki2}. 
In this paper, we pursue their works to establish several Cuoco--Monsky-type results for $\Zp^{\,d}$-covers of 3-manifolds and shed deeper insight there. 

For an abelian group $G$, \glssymbol{exponential}$e(G)$ denotes the $p$-exponent of the size of the $p$-torsion subgroup of $G$. 
For a $\Zp$-module $\mca{M}$, \glssymbol{zprk}$r(\mca{M})$ denotes the free $\Zp$-rank of $\mca{M}$.  
For a 3-manifold $M$, we write \glssymbol{homology}$H_1(M)=H_1(M;\Z)$. A closed connected orientable 3-manifold \glssymbol{closedmanifold}$M$ is called a rational homology 3-sphere (\glssymbol{qhs3}$\Q$HS$^3$) if $H_i(M;\Q)\cong H_i(S^3;\Q)$ holds for all $i$, or equivalently, if $H_1(M)$ is a finite group. One of our goals is to prove the following. 
\begin{mainresult}[Theorems \ref{thm.branched}, \ref{rem.alpha}]
\label{mainresult} 
Let $M$ be a $\Q$HS\,$^3$ and $L$ a link in $M$. Let  $(M_{p^n}\to M)_n$ be a branched $\Zp^{\,d}$-cover over $(M,L)$, that is, a compatible system of branched $(\Z/p^n\Z)^d$-covers branched along $L$.
Suppose that every $M_{p^n}$ is a $\Q$HS\,$^3$ and that $L$ does not decompose in $(M_{p^n}\to M)_n$. 
Then there exist some $\mu,\lambda\in\Z_{\geq0}$, depending only on $(M_{p^n}\to M)_n$, such that 
\[
e(H_1(M_{p^n}))=\left(\mu p^n+\lambda n+O(1)\right)p^{(d-1)n} 
\]
holds, 
where \glssymbol{landauo}$O$ is the Bachmann--Landau notation with respect to $n$. 

Furthermore, the $O(1)p^{(d-1)n}=O(p^{(d-1)})$ part may be refined to $\mu_1p^{(d-1)n}+O(np^{(d-2)n})$ for some $\mu_1\in\R$. If $d=2$, then $\mu_1\in\Q$. 

Let $(X_{p^n}\to X)_n$ be a system of unbranched $(\Z/p^n\Z)^d$-covers
that are restrictions of $(M_{p^n}\to M)_n$ to the link exteriors. 
Then $\mu$ and $\lambda$ are those of the characteristic element \glssymbol{char}${\rm char}\,\mca{H}$ 
of a $\Lambda$-module defined by $\mca{H}:=\varprojlim_n H_1(X_{p^n};\Zp)$. 
\end{mainresult}
These \glssymbol{iwasawamu}\glssymbol{iwasawalambda}$\mu, \lambda$ are called \emph{the Iwasawa invariants} of the $\Zp$-cover. For a 3-manifold $M$, the link exterior of $L$ in $M$ is defined by \glssymbol{VL}\glssymbol{IntVL}$X:=M\setminus {\rm Int}\,V_L$, where ${\rm Int}\,V_L$ denotes the interior of a tubular neighborhood $V_L$ of $L$. 
The system $(X_{p^n}\to X)_n$ of $(\Z/p^n\Z)^d$-covers corresponds to a homomorphism 
\glssymbol{pi1}$\tau:\pi_1(X)\to \Zp^d$ 
that induces a surjective homomorphism \glssymbol{whpi1}$\wh{\tau}:\wh{\pi}_1(X)\surj \Zp^d$ 
on the pro-$p$ completion, 
and the system $(M_{p^n}\to M)_n$ of branched covers is re-obtained by the Fox-completions. 
We remark that there are many $\Zp^d$-covers that are not derived from $\Z^d$-covers (cf.~\Cref{eg.Z52}). 
If we take topological generators $(t_1,\ldots, t_d)$ of $\Zp^{\,d}$, then we have an isomorphism of topological rings 
so-called the Iwasawa--Serre isomorphism 
\glssymbol{Lambda}\[\Zp[\![\Zp^d]\!]\congto \Lambda:=\Zp[\![T_1,\ldots,T_d]\!];\, t_i\mapsto 1+T_i,\]  
where $\Zp[\![\Zp^d]\!]$ denotes the complete group ring, and $\Zp[\![T_1,\ldots,T_d]\!]$ denotes the formal power series (cf.~\cite{Greenberg1973AJM}, \cite[Theorem 3.3.9]{Sharifi.IT2023OCT}). 
The image $\Zp^{\,d}$ of $\wh{\tau}$ acts on \glssymbol{projectivelimitmodule}$\mca{H}=\varprojlim H_1(X_{p^n};\Zp) \cong \wh{\pi}_1(X)^{\rm ab}$ by the conjugation, and $\mca{H}$ becomes a finitely generated torsion $\Lambda$-module. 
We define the invariants $\mu$ and $\lambda$ of $0\neq F\in \Lambda$ in Subsection \ref{ss.pW}, 
and the characteristic elements of finitely generated $\Lambda$-modules in \Cref{sec.char}. 

To illustrate the proof of \Cref{mainresult}, 
let \glssymbol{ipn}$\mca{I}_{p^n}$ and \glssymbol{jpn}$\mca{J}_{p^n}$ denote the ideals of $\Lambda$ topologically generated by $\{(1+T_i)^{p^n}-1\mid 1\leq i\leq d\}$ 
and $\{((1+T_i)^{p^n}-1)/T_i\mid 1\leq i\leq d\}$ respectively. 
First, we establish \emph{the fundamental two exact sequences} (Theorems \ref{twoexact}, \ref{twoexact2}) 
that are highly versatile, and derive the relationship between 
$\mca{H}$ and $H_1(X_{p^n})$ in a slightly general setting: 
\begin{mainresult}[\Cref{thm.unbr-asymptotic}] \label{mr.HXpn}
Let \glssymbol{cptmanifold}$X$ be a compact connected orientable 3-manifold. 
Let $(X_{p^n}\to X)_n$ be a $\Zp^{\,d}$-cover, that is, a compatible system of $(\Z/p^n\Z)^d$-covers,
and define a $\Lambda$-module by $\mathcal{H}:=\varprojlim H_1(X_{p^n};\Zp)$. Then 
\begin{gather*} 
e(\mathcal{H}/\mathcal{I}_{p^n}\mathcal{H})=e(H_1(X_{p^n}))+O(n),\\ 
r(H_1(X_{p^n},\Zp))=r(\mathcal{H}/\mathcal{I}_{p^n}\mathcal{H})+O(1). 
\end{gather*}
\end{mainresult} 
The two fundamental exact sequences also yield a result on the reduced Alexander polynomials and its $p$-adic variant (\Cref{reduced}), which is a well-known folklore at least for some special cases. We will use it to discuss the vanishing conditions in \Cref{mr.ZHS}. 

Secondly, based on Cuoco--Monsky's result on $\Lambda$-modules and its modification 
(\Cref{22510}, \Cref{Monsky21}), we obtain the following. 

\begin{mainresult}[\Cref{2252}] \label{mr.unbr} 
Let $(X_{p^n}\to X)_n$ be a $\Zp^{\,d}$-cover over a compact connected orientable $3$-manifold $X$. 

{\rm (1)} If $r(\mca{H}/\mathcal{J}_{p^n}\mca{H})=O(p^{(d-2)n})$, 
then there exists $\mu\in\Z_{\geq0}$, depending only on $(X_{p^n}\to X)_n$, such that
\[
e(H_1(X_{p^n}))=\mu p^{dn}+O(np^{(d-1)n}).
\]

{\rm (2)} If $r(\mca{H}/\mathcal{I}_{p^n}\mca{H})=O(p^{(d-2)n})$, 
then there exist $\mu,\lambda\in\Z_{\geq 0}$, depending only on $(X_{p^n}\to X)_n$, such that
\[
e(H_1(X_{p^n}))=\left(\mu p^{(d-1)}+\lambda n+O(1)\right)p^{(d-1)n}.
\]

These $\mu$ and $\lambda$ are those of the characteristic element 
${\rm char}\,\mca{H} \in \Lambda$. 
\end{mainresult}
If the $\Zp^{\,d}$-cover is derived from a $\Z^d$-cover, then ${\rm char}\,\mca{H}$ is the image of the Alexander polynomial (\Cref{zcovertheorem}). 
To determine $\mu$ and $\lambda$, we discuss the basic properties of the characteristic elements and invoke a multi-variable analogue of the $p$-adic Weierstrass preparation theorem developed by Monsky \cite{Monsky1989} (cf.~Section 3). 
One of our main efforts will be made in the analysis of the rank assumptions. 
We prove that the assumptions of \Cref{mainresult} imply the rank assumptions of \Cref{mr.unbr} for the link exteriors (\Cref{decompositionprop}). 

Thirdly, we discuss the gap between the first homology groups of branched and unbranched covers of links, which has been studied as a highly interesting object by Hartley, Mayberry, Murasugi, and Sakuma \cite{HartleyMurasugi1978, MayberryMurasugi1982, Sakuma1980MT}. 
We invoke Hartley--Murasugi's result \cite[Lemma 2.8]{HartleyMurasugi1978} to 
estimate the gaps, completing the proof of \Cref{mainresult}: 
\begin{mainresult}[\Cref{gaptheorem}]\label{MR5} 
Let $M$ be a $\Q$HS\,$^3$ and $L$ a link in $M$. Let $(M_{p^n}\to M)_n$ a branched $\Zp^{\,d}$-cover over $(M,L)$ and let $(X_{p^n}\to X)$ denote the restrictions to the exterior. 
If $L$ does not decompose in $(M_{p^n}\to M)_n$, then 
\[e(H_1(M_{p^n}))=e(H_1(X_{p^n}))+O(n).\] 
\end{mainresult} 

A 3-manifold $M$ is called an integral homology 3-sphere (\glssymbol{zhs3}$\Z$HS$^3$) if $H_i(M;\Z)\cong H_i(S^3;\Z)$ holds for all $i$, or equivalently, $H_1(M;\Z)=0$ holds. 
If we assume that $M$ is a $\Z$HS$^3$, then the results of Mayberry--Murasugi \cite{MayberryMurasugi1982} and Porti \cite{Porti2004} yield the following finer formula, 
which Greenberg conjectured for $\Zp^d$-extensions of number fields \cite[Section 7]{CuocoMonsky1981}. 
Let \glssymbol{W}$W$ denote the set of roots of unity of $p$-power order in an algebraic closure $\ol{\Q}$ of $\Q$. Let \glssymbol{Lambdaz}$\Lambda_{{\Z},d}=\Lambda_{\Z}$ denote the Laurent polynomial ring $\Z[t_1^{\pm 1},\ldots, t_d^{\pm 1}]$. For a $c$-component link $L$ in a $\Q$HS$^3$, let $\Delta_L\in \Lambda_{\Z,c}$ denote the Alexander polynomial of $L$, which will be defined in \Cref{def.Alex} (2).

\begin{mainresult}[Theorems \ref{thm.integral}, \ref{thm.integral.Delta}] \label{mr.ZHS} 
Let $M$ be a $\Z$HS\,$^3$, $L$ a $c$-component link in $M$ with $c\geq d$, and 
suppose that the Alexander polynomial $\Delta_L(t_1,\ldots,t_c)$ does not vanish on $(W\setminus \{1\})^c$. 
Let $(M_{p^n}\to M)_n$ be a branched $\Zp^{\,d}$-cover over $(M,L)$. 
Then, every $M_{p^n}$ is a $\Q$HS\,$^3$, and there exists $f(U,V)\in\Q[U,V]$ with
$\deg_{V}f\leq1$ and the total degree $\deg f\leq d$ such that,  
for every sufficiently large $n$, 
\[
e(H_1(M_{p^n}))=f(p^n, n)
\]
holds. In other words, there exist some $\mu,\lambda\in\Z_{\geq 0}$ and $\mu_1,\ldots,\mu_{d-1},\lambda_1,\ldots,\lambda_{d-1},\nu \in \Q$ such that, for every sufficiently large $n$, 
\[
e(H_1(M_{p^n}))= \mu p^{dn}+\lambda np^{(d-1)n}+\mu_1p^{(d-1)n}+\lambda_1 np^{(d-2)n}+\cdots +\mu_{d-1} p^n+\lambda_{d-1} n+\nu
\] 
holds. 

These $\mu$ and $\lambda$ are those of ${\rm char}\,\mca{H}$ if {\rm (1)} $c=d$ or if {\rm (2)} $\Delta_L$ does not vanish on $W^c\setminus\{(1,1,\ldots,1)\}$. 
\end{mainresult} 

In topology, unbranched covers are generally more challenging and subtle than branched ones, often tied to unresolved questions. For example, Mayberry--Murasugi's work (\Cref{2260}) nearly completes branched cases, while unbranched cases ([30, Section 14]) remain less explored. Our approach first establishes results for unbranched cases, which we then use to derive those for branched cases. 

The Hartley--Murasugi technique (\Cref{lem.HM}, \Cref{gaptheorem}) approximates the gap between the sizes of the first homology groups of branched and unbranched covers, and this approximation requires the non-decomposability of the branch links as an assumption.  
In the $\mathbb{Z}$HS$^3$ case, non-vanishing of the Alexander polynomial on $W^c\setminus\{(1,1,\ldots,1)\}$ implies non-decomposability of $L$ (\Cref{prop.notdecompose}) and, by Porti (\Cref{2260}), ensures that every finite abelian cover branched along $L$ is a $\Q$HS$^3$, both key to \Cref{mainresult}. 

\ 

The relationship between the torsion size of $H_1(M)$ and the 1st Betti number $\beta_1(M)$ of a 3-manifold $M$ is like that of siblings as research objects.
Although the analogue of $\beta_1(M)$ on the number theory side is always trivial, $\beta_1$'s of 3-manifolds deserve to be studied in the pro-$p$ setting. 
In fact, Kionke studies the $p$-adic Betti number in view of the Atiyah conjecture \cite{Kionke2020JLMS}. 
The polynomial periodicity of Betti numbers for $\Z^d$-coverings has been researched by Adams--Sarnak \cite{AdamsSarnak1994}, Hironaka \cite{Hironaka1992Invent}, and Sakuma \cite{Sakuma1995Canada}. 
Combining with our study on the rank growth (\Cref{lemmaforbetti}), we derive the following. 

Let $M$ be a $\Q$HS$^3$ and $L$ a $d$-component link in $M$. 
Let $(M_n\to M)_n$ denote the compatible system of branched $(\Z/n\Z)^d$-covers obtained from the unique $\Z^d$-cover over the exterior. 
Then, by \cite[Corollary 1.4]{AdamsSarnak1994} and {\cite[Theorem 7.5]{Sakuma1995Canada}} (\Cref{22511}), 
the 1st Betti numbers \glssymbol{betti}$\beta_1(M_n)$ are \emph{polynomial periodic}, that is, 
there exists 
$N\in \Z_{\geq 1}$ and $f_0(U),\ldots, f_{N-1}(U)\in \Q(U)$ such that 
$\beta_1(M_n)\equiv f_i(n)$ if $n\equiv i$ mod $N$. 
As a consequence, there exists some $N'$ and $\{g_0,\ldots, g_{N'-1}\} \subset \{f_0,\ldots, f_{N-1}\}$ 
such that $\beta_1(M_{p^n})\equiv g_i(n)$ if $n\equiv i$ mod $N'$. 
Put $D(M,L):={\rm max}\{{\rm deg}\,g_i\mid i=0,1,\ldots, N'-1\}$. 
\begin{mainresult}[\Cref{2284}] Let the setting be as above with $d\geq 2$ and 
Suppose that $\Delta_L(t_1,\ldots, t_d)$ does not vanish on $W^d\setminus\{(1,1,\ldots,1)\}$. 
Then $D(M,L)\leq d-2$ holds. In particular, if $d=2$, then $\beta_1(M_{p^n})$ is periodic. 
\end{mainresult} 

The quest for $\Zp$-extensions with prescribed Iwasawa invariants or modules has been of particular interest (cf.~\cite[Theorem 1]{Iwasawa1973}, \cite[Theorem 2]{Ozaki2004JMSJ}), 
and its topological analogue for $\Zp$-covers has been made in \cite[Proposition 4.15]{KadokamiMizusawa2008}, \cite[Theorem 5.2]{Ueki3}, \cite{UekiYoshizaki-plimits}. 
In this view, we study the twisted Whitehead links $W_m$ ($m\in \Z_{\geq 0}$) with 
\[\Delta_{W_{2m-1}}(x,y)=m(x-1)(y-1)\] 
and verify the following. 
\begin{mainresult}[\Cref{11257}] 
Let $L= W_{2p^\kappa -1}$ in $S\,^3$ with $\kappa\in \Z_{\geq 0}$. Then the $\Zp^{\,2}$-cover satisfies  
\[e(H_1(M_{p^n},\Zp))=\left(\kappa p^n+2n-2\kappa \right)p^n-2n+\kappa.\]
In particular, for any $\kappa\in \Z_{\geq 0}$, there exists a 2-component link $L$ whose $\Zp^{\,2}$-cover satisfies 
$\mu=\kappa$. 
\end{mainresult} 
See \Cref{rem.nu} for the details of further results on prescribed Iwasawa invariants.  
Another explicit example that we will see is ``Magen David'' $L=6_1^2$ in $S^3$ (\Cref{eg.L612}). 
We note that there is no implication between the assumptions of Main Results \ref{mainresult} and \ref{mr.ZHS} (cf.~\Cref{remark.MR1and5}); 
\Cref{eg.L612} is for both, while \Cref{11257} is only for the latter.   
At the end of this paper, we also attach \Cref{table.mulambda} (\Cref{eg.table}) for the Alexander--Iwasawa polynomial $\Delta_L(1+T_1,\ldots)$ and their $\mu$ and $\lambda$ of all links in Rolfsen's table \cite{Rolfsen1990}.\\  

This paper is organized as follows. 
In Section 2, we review the Iwasawa class number formula for $\Zp$-extensions and the Cuoco--Monsky formula for $\Zp^{\,d}$-extensions together with historical background and motivations on the number theory side, 
as well as attach a brief remark on arithmetic topology and previous results for the $d=1$ cases. 
In Section 3, we recollect the results of Cuoco--Monsky on a multi-variable analogue of the $p$-adic Weierstrass preparation theorem that determines the Iwasawa invariants of $F\in \Lambda$ 
and a result of Monsky (\Cref{22521}) that is vital for the proof of \Cref{mr.ZHS}. 
In Section 4, we study some basic properties of modules over a Noetherian integrally closed domain $R$ for use in subsequent sections; 
We verify the relationship amongst the Fitting ideal, the characteristic ideal, and the completion. 
In Section 5, we further recall the results of Cuoco and Monsky on $\Lambda$-modules and develop their modification for later use. We will make much effort in replacing $\mca{I}_{p^n}$ by $\mca{J}_{p^n}$ and in the analysis of the rank assumptions as well. 

In Section 6, we establish purely group theoretic results that are crucial to attain our main results, say, 
the fundamental two exact sequences (\Cref{twoexact}), the trivial action (\Cref{actionprop}), the asymptotic formula (\Cref{asymptotic}), and their versions for topological groups. 
In Section 7, we recall the notion of the Alexander polynomials of links and those of $\Z^d$-covers and give a remark on the completion (\Cref{prop.hatH1Xinfty}). 
We also prove a result on the reduced Alexander polynomials and its $p$-adic variant (\Cref{reduced}). 
In Section 8, we establish estimation formulas for the torsion/rank growth in unbranched $\Z^d$-covers or $\Zp^{\,d}$-covers (Theorems \ref{examplezdcover}, \ref{thm.unbr-asymptotic}). 
We also study a $\Zp^{\,d}$-cover derived from a $\Z^d$-cover $X_\infty\to X$ to find a pseudo-isomorphism of 
$\Lambda$-modules between 
$\mca{H}=\varprojlim H_1(X_{p^n};\Zp)$ and 
$\wh{H}_1(X_\infty):=H_1(X_\infty)\otimes_{\Lambda_\Z}\Lambda$ (\Cref{zcovertheorem}). 
In Section 9, by using Hartley--Murasugi's result, we estimate the gaps between branched and unbranched $(\Z/p^n\Z)^d$-covers (\Cref{gaptheorem}). 

In Section 10, we establish several Iwasawa-type formulas for $\Zp^{\,d}$-covers (Theorems \ref{2252}, \ref{thm.branched}, \ref{thm.integral}, \ref{thm.integral.Delta}) that we exhibited in above. 
In Section 11, we investigate the growth of Betti numbers in $\Zp^{\,d}$-covers and prove \Cref{2284}. 
In Section 12, we calculate the Alexander polynomial of the twisted Whitehead links $W_m$. 
In Section 13, we provide several examples to reinforce our result. 
We investigate $W_{2p^{\kappa}}$ with $\kappa\in \Z_{\geq 0}$, ``Magen David'' $L=6^2_1$, and all links in the Rolfsen table. 
We point out remarks on prescribed Iwasawa invariants (\Cref{rem.nu}) and on the $p$-adic torsion in a $\Zp^{\,2}$-cover as well. 

\begin{ack}
This paper contains materials of the doctoral thesis \cite{Tateno-Phd} of the first author, and he would like to thank Hiroshi Suzuki most warmly for his steady guidance and helpful comments as his advisor. 
Especially, the authors could not have completed the proof of the theorems in \Cref{sec.twoexact} without his support. 
Also, the authors would like to express their sincere gratitude to Makoto Sakuma for giving essential information about the comparison between branched and unbranched covers.   
Furthermore, the authors are grateful to Kazuki Hayashi, Tsuyoshi Itoh, Teruhisa Kadokami, Takenori Kataoka, Barry Mazur, Masanori Morishita, Takefumi Nosaka, Manabu Ozaki, Anh T.~Tran, Hyuga Yoshizaki, the anonymous referee, and the faculty for their informative comments and support. 
The second author has been partially supported by 
%Japan Society for the Promotion of Science
JSPS KAKENHI Grant Numbers JP19K14538 and JP23K12969.
\end{ack}  

\section{Backgrounds} 
\subsection{Number theory} \label{ss.NT} 
Let $k$ be a number field, i.e., a finite extension of the field $\Q$ of rational numbers. The class number of $k$ is the size of the ideal class group $Cl(k)$ of $k$. The notion of class numbers in number theory has been crucial as a research object since the era of Kummer \cite{Kummer1847FLT}. 
Kummer invented the notion of class numbers and successfully proved that the Fermat last theorem holds for a prime number $p$ if $p$ does not divide the class number of the $p$-th cyclotomic field. He also proved that the class numbers of cyclotomic fields are related to the particular values of the Riemann zeta function, which is closely related to the distribution of prime numbers.

That is being said, the regularity of class numbers has been mysterious, and it is basically difficult to control. The Gauss conjecture, which states that there are infinitely many quadratic real fields whose class numbers are one, is still an open problem. We do not even know whether there are infinitely many number fields whose class numbers are one or not.

Under such a background, Iwasawa found the following formula that controls the $p$-exponent of the class numbers in any $\Zp$-tower of number fields. 

\begin{theorem}[{Iwasawa, \cite[Theorem 4]{Iwasawa1959}}]\label{2251}
Let $k_\infty/k$ be a $\Zp$-extension and 
let $k_{p^n}/k$ denote the $\Z/p^n\Z$-subextension for each $n\in \Z_{\geq 0}$. 
Then there exist invariants $\mu,\lambda\in \Z_{\geq0}$ and $\nu\in\Z$, depending only on $k_{\infty}/k$, such that, 
for every sufficiently large $n$, 
\[
e(Cl(k_{p^n}))=\mu p^n+\lambda n+\nu
\]
holds. 
\end{theorem} 
The values $\mu,\lambda,\nu$ are called the Iwasawa invariants of $k_\infty/k$.
This result is known to be the first asymptotic formula that describes the regularity of the variation of the class numbers in certain towers of number fields. 

To prove this formula, class field theory plays an essential role. 
Let $k$ be a number field and let $k_{\ab}^{\ur}$ denote the Hilbert class field, i.e., the maximal unramified abelian extension of $k$. Then the class field theory asserts that there is an isomorphism 
\[ \Gal(k_{\ab}^{\ur}/k)\cong Cl(k)\] 
called Artin's reciprocity, which allows one to regard the class numbers as the sizes of Galois groups. 
Given a $\Zp$-extension, Iwasawa considered the corresponding tower of Hilbert $p$-class fields and obtained the aforementioned formula.\\ 

Now let $d\geq 2$. In \cite{CuocoMonsky1981}, Cuoco and Monsky generalized Iwasawa's formula to $\Zp^{\,d}$-extensions.

\begin{theorem}[Cuoco--Monsky, {\cite[Theorem I]{CuocoMonsky1981}}]
Let $k_\infty /k$ be a $\Zp^{\,d}$-extension of number fields and 
let $k_{p^n}/k$ denote the $(\Z/p^n\Z)^d$-subextensions for each $n\in \Z_{\geq 0}$. 
Then there exist invariants $\mu,\lambda\in \Z_{\geq0}$, depending only on $k_{\infty}/k$, 
such that 
\[
e(Cl(k_{p^n}))=\left(\mu p^n+\lambda n+O(1)\right)p^{(d-1)n}
\]
holds, 
where $O$ is the Bachmann--Landau notation with respect to $n$. 
\end{theorem}
These $\mu, \lambda$ are also called the Iwasawa invariants of $k_\infty/k$. 
Cuoco--Monsky showed this result by developing the structure theory for $\Lambda$-modules, where $\Lambda=\Zp[\![T_1,\ldots,T_d]\!]$. 
Let $l_\infty/k_\infty$ denote the maximal abelian unramified pro-$p$ extension of $k_\infty$. 
Then $l_\infty/k$ is Galois. 
Put $\Gamma:=\Gal(k_\infty/k)$ and $\mathcal{X}:=\Gal(l_\infty/k_\infty)$. 
Then there is a well-defined action
\[
\Gamma\times \mathcal{X}\ni(\sigma,x)\mapsto \tilde{\sigma}x\tilde{\sigma}^{-1}\in \mathcal{X},
\]
where $\tilde{\sigma}$ is a lifting of $\sigma$ to $\Gal(l_\infty/k)$.
Then, as Greenberg proved \cite{Greenberg1973AJM}, 
$\mathcal{X}$ becomes a finitely generated torsion $\Lambda$-module 
via the Iwasawa--Serre isomorphism $\Zp[\![\Gamma]\!]\congto \Lambda; 
\gamma_i\mapsto 1+T_i$. 
Cuoco--Monsky actually proved that  
the $\mu$ and $\lambda$ in above are those of the characteristic element of $\mathcal{X}$. 
Monsky further investigated the structures of $\Lambda$-modules in a series of papers and finally obtained the following. 
\begin{theorem}[Monsky {\cite[Theorem 3.13]{Monsky1989}}] 
The $O(1)p^{(d-1)n}=O(p^{(d-1)n})$ part in above 
may be refined to $\mu_1p^{(d-1)n}+O(np^{(d-2)n})$ for some $\mu_1\in\R$. If $d=2$, then $\mu_1\in\Q$. 
\end{theorem}

According to \cite[Section 7]{CuocoMonsky1981}, 
Greenberg conjectured that there exists a polynomial $f(U,V)$ such that, 
for every sufficiently large $n$, 
\[
e(Cl(k_{p^n}))=f(p^n,n) 
\]
holds. 
Cuoco--Monsky pointed out that the quotients of the $\Lambda$-module $\mathcal{M}=\Zp[\![T_1,T_2]\!]/(p,T_1^5-T_2^5)$ with $p\equiv \pm 2$ mod 5 
does not obey the conjecture. However, the existence of a $\Zp^{\,d}$-extension that does not obey Greenberg's conjecture is not yet known. 

\subsection{Arithmetic topology} 
It is said that Gauss's proof of the quadratic reciprocity law using Gauss sums is based on his insight into the analogy between knots and prime numbers. 
In modern times, years after Mazur's note \cite{Mazur1963}, 
the analogy between low dimensional topology and number theory was
systematically developed in the format of a dictionary by Reznikov and Kapranov \cite{Kapranov1995, Reznikov1997, Reznikov2000}. 
Morishita independently took the analogy seriously with his deep insight and wrote a book, laying the groundwork for the research field \cite{Morishita2002, Morishita2010, Morishita2012}. 
Afterward, M.~Kim gave further perspective and founded arithmetic Chern--Simons theory \cite{MKim2020}. 
Furthermore, Mazur recently announced that by combining the theory of $C^\infty$-category and profinite $B\Gamma$-theory, we may enhance the class field theory and the Iwasawa theory \cite{FengHarrisMazur2023-arXiv}. 

We should remark that profinite refinements of the Alexander--Fox theory are vital in the context of profinite rigidity. 
Yi Liu \cite{YiLiu2023Invent} proved that the hyperbolic volume is almost profinitely rigid by using a certain property of the completed group ring $\Zp[\![t^\wh{\Z}]\!]$ proved by the second author \cite{Ueki5}, where $\wh{\Z}:=\varprojlim \Z/n\Z$. 
Liu works over $\Zp[\![\wh{\Z}^d]\!]$, 
so we may expect further interaction between deep geometric topology and number theory. 
Here we attach a small piece of a basic dictionary of the analogies. 

\begin{longtable}{|c|c|}
\caption{M$^2$KR-dictionary} \label{table.dictionary} 
\endfirsthead 

\hline 
\endhead  

\hline
Number theory& Low dimensional topology \\
\hline
\hline 
The spectrum $\Spec\mathcal{O}_k$ of the ring of integers & A closed connected orientable\\
of a number field $k$ & $3$-manifold $M$\\
\hline 
$\Spec \Z$ of $\Q$ & $S^3$\\
\hline
The ideal class group 
$Cl(k)$&$H_1(M;\Z)$ or \glssymbol{tor}${\rm tor}_\Z H_1(M;\Z)$\\
\hline
Fact: $Cl(k)$ is a finite group&(Assumption: $M$ is a $\Q$HS$^3$)\\
\hline 
Artin reciprocity law &Hurewicz isomorphism \\ 
$\pi_1^{\text{\'{e}t}}(\ol{{\rm Spec}\,\mathcal{O}_k})^{\rm ab} \cong \Gal(k_{\ab}^{\ur}/k) \cong Cl(k)$ & $\pi_1(M)^{\rm ab}\cong {\rm Gal}(M_{\rm ab}\to M)\cong H_1(M;\Z)$ \\ 
\hline \hline 
A non-zero prime ideal $\mathfrak{p}:$ &A knot \glssymbol{knot}$K:S^1\hookrightarrow M$\\ 
$\Spec\mathcal{O}_k/\mathfrak{p}\hookrightarrow\Spec\mathcal{O}_k$ & \\ 
\hline 
A family of primes $S=\{\mathfrak{p}_1,\ldots, \mathfrak{p}_r\}$&A link \glssymbol{link}$L:\sqcup S^1\hookrightarrow M$\\
\hline 
$\pi_1^{\text{\'{e}t}}(\ol{{\rm Spec}\,\mathcal{O}_k}\setminus S)$ & $\pi_1(M\setminus L)$ \\ \hline 

The $\p$-adic number field $k_{\p}$ & A tubular neighborhood \glssymbol{vk}$V_K$ \\ \hline 
The ring of $\p$-adic integers $\mca{O}_{\p}$ & The boundary \glssymbol{boundary}$\partial V_K$ of $V_K$\\ \hline 
Hilbert ramification theory for prime ideals& Hilbert ramification theory for knots\\ 
\hline \hline 
A $\Zp$-extension $k_{\infty}/k$ & A branched $\Zp$-cover $(M_{p^n}\to M)_n$\\
\hline\ 
The $\Z/p^n\Z$-subextension $k_{p^n}/k$ & The $\Z/p^n\Z$-subcover $M_{p^n}\to M$\\ \hline 
Iwasawa theory & Alexander--Fox theory \\ 
& and its $p$-adic refinement \\ \hline \hline 
A $\Zp^{\,d}$-extension $k_{\infty}/k$ & A branched $\Zp^{\,d}$-cover $(M_{p^n}\to M)_n$ \\ 
\hline 
The $(\Z/p^n\Z)^d$-subextension $k_{p^n}$ & The  $(\Z/p^n\Z)^d$-subcover $M_{p^n}\to M$ \\ 
\hline 
Greenberg--Cuoco--Monsky theory & Multivariable Alexander--Fox theory \\ 
& and its $p$-adic refinement \\ \hline 
\end{longtable} 

\subsection{Iwasawa-type formula of $\Zp$-covers of links} Here we recall results for $d=1$ cases in topology. 
The homologies of the cyclic branched covers of links are determined by the Alexander polynomials and 
have been extensively studied as basic invariants. 
Reznikov \cite{Reznikov1997, Reznikov2000} gave further insight there, by focusing on the $p$-torsions of the $p$-cyclic covers after a classical problem in number theory.  

Hillman--Matei--Morishita \cite{HillmanMateiMorishita2006} extended his viewpoint to the system of cyclic $p$-covers and initially proved the Iwasawa-type formula for the case where $M$ is $S^3$ and the $\Zp$-cover is derived from a $\Z$-cover. 
Afterward, Kadokami--Mizusawa \cite{KadokamiMizusawa2008} extended it to the case where $M$ is a $\Q$HS$^3$, which is an analogue of a general number field, and finally, the second author \cite{Ueki2} proved it for the cases where the $\Zp$-cover is not necessarily derived from a $\Z$-cover. 

\begin{theorem} [Iwasawa-type formula; {\cite[Theorem 5.1.7]{HillmanMateiMorishita2006}}, {\cite[Theorem 2.1]{KadokamiMizusawa2008}}, {\cite[Theorem 4.9]{Ueki2}}] \label{thm.Zpcover} \label{thm.d=1}
Let $M$ be a $\Q$HS\,$^3$, $L$ a link in $M$, and let $X=M-{\rm Int}\,V_L$ denote the exterior. 
Let $(X_{p^n}\to X)_n$ be a compatible system of $\Z/p^n\Z$-covers and let $(M_{p^n}\to M)_n$ denote the system of branched covers obtained by the Fox completions. 
Suppose that every $M_{p^n}$ is a $\Q$HS\,$^3$. Then there exist $\mu, \lambda \in \Z_{\geq0}$ and $\nu\in\Z$, depending only on $p$ and $(M_{p^n}\to M)_n$, such that, for every sufficiently large $n$, 
\[
e(H_1(M_{p^n}))=\mu p^n+\lambda n+\nu
\]
holds. 
\end{theorem} These $\mu,\lambda,\nu$ are called \emph{the Iwasawa invariants} of $(M_{p^n}\to M)_n$.
Furthermore, Tange and the second author established a similar formula for knot group representations \cite{TangeUeki2024MathNach}. 
The $\mu$ is a $p$-adic analogue of the Mahler measure, so it may be regarded as an invariant of a dynamical system. 
The $\lambda$ is related to the genus of a knot, as $\lambda$'s of $\Zp$-fields in number theory are analogues of the genera of Riemann surfaces. 
 
In these cases, the $\mu$ and $\lambda$ are determined by the characteristic element ${\rm char}\,\mca{H}$ of the Iwasawa module $\mca{H}=\varprojlim H_1(X_{p^n};\Zp)$ of the restriction to the link exteriors. 
Namely, the $p$-adic Weierstrass preparation theorem yields that \glssymbol{equaluptomultiplicationbyunit}${\rm char}\,\mca{H} \doteq p^\mu (T^\lambda+p(\text{lower terms}))$ in $\Lambda=\Zp[\![T]\!]$ or its constant extension. 

We aim to establish a similar formula for $\Zp^d$-covers in the subsequent sections. 

\section{$p$-adic preparations} 

In this section, we recollect several results of Cuoco--Monsky \cite{CuocoMonsky1981} and Monsky \cite{Monsky1981PS} on 
the ring $\Lambda=\Zp[\![T_1,\ldots,T_d]\!]$ 
that determines the Iwasawa invariants of $0\neq F\in \Lambda$. 

\subsection{$p$-adic Weierstrass preparation theorem} \label{ss.pW} 
In this subsection, based on Cuoco--Monsky's work \cite[Section 1]{CuocoMonsky1981}, we introduce the notion of the Iwasawa invariants of finitely generated torsion modules over Iwasawa algebra $\Lambda$ with multiple variables, 
together with a multivariable analogue of the $p$-adic Weierstrass preparation theorem {\cite[Theorem 7.3]{Washington}} that was essential in the one variable cases in determining Iwasawa invariants. 

Let \glssymbol{Gamma}$\Gamma$ be a free $\Zp$-module of rank $d$ written multiplicatively and fix a basis $\{\gamma_1, \ldots,\gamma_d\}$ of $\Gamma$. Let \glssymbol{groupring}$\Zp[\![\Gamma]\!]$ denote the complete group ring of $\Gamma$ over $\Zp$. 
Then we have so-called the Iwasawa--Serre isomorphism \[\Zp[\![\Gamma]\!]\congto \Lambda=\Zp[\![T_1,\ldots,T_d]\!];\ \gamma_i\mapsto 1+T_i,\] 
and $\Gamma$ may be identified with a multiplicative subset of $\Lambda$ generated by $\{1+T_i\mid 1\leq i\leq d\}$.  
Each element of $\GL_d(\Zp)$ induces an automorphism of $\Gamma$ that prolongs to a ring automorphism of $\Lambda$. We call such automorphisms linear automorphisms of $\Lambda$. 
Let  $\sigma \in \Gamma$ be a $p$-primitive element, that is, $\sigma\in\Gamma\backslash\Gamma^p$, where $\Gamma^p:=\{\gamma\in \Gamma\mid \gamma=\gamma'^p\mbox{ for some }\gamma'\in\Gamma\}$. Then, for an arbitrarily fixed $j$ with $1\leq j\leq d$, there is a linear automorphism mapping $\sigma$ to $1+T_j$. Let \glssymbol{Omega}$\Omega:=\mathbb{F}_p[\![T_1,\ldots,T_d]\!]$. For each $F\in\Lambda$, 
$\overline{F}\in\Omega$ denote the ${\rm mod}\ p$ reduction of $F$. 
Since $(\overline{T}_j)$ is a height 1 prime ideal of $\Omega$, so is $(\overline{\sigma -1})$. 
For each height 1 prime ideal $\mathfrak{p}$ of $\Omega$, let \glssymbol{valp}$v_{\mathfrak{p}}$ denote the associated discrete valuation. 

\begin{definition} \label{def.mulambda}
For $0\neq F\in \Lambda$, 
let \glssymbol{muf}$\mu=\mu(F)\in\Z_{\geq 0}$ denote the maximal integer satisfying $p^\mu\mid F$. 
In addition, we put $F_0:=F/p^\mu\in \Lambda$ and define 
\glssymbol{lambdaf}\[ \lambda=\lambda(F):=\sum v_{\mathfrak{p}}(\overline{F_0}), \]
where $\p$ runs through all primes of the form $(\overline{\sigma -1})$ with $\sigma$ being a $p$-primitive element of $\Gamma$. 
These $\mu,\lambda$ are called \emph{the Iwasawa invariants} of $F$. 
\end{definition} 
Let \glssymbol{algebraicclosure}$\ol{\Q}_p$ be an algebraic closure of $\Qp$ and fix an embedding $\ol{\Q}\inj \ol{\Q}_p$, so that we have \[
W=\{\xi\in\overline{\Q}_p\mid \xi^{p^n}=1\mbox{ for some }n\geq0\}.
\]
Let \glssymbol{val}$v:\overline{\Q}_p\rightarrow\Q$ denote the multiplicative  
$p$-adic valuation normalized so that $v(p)=1$ and with an unusual convention $v(0)=0$. 
For each $\zeta=(\zeta_1,\ldots,\zeta_d)\in W^d$, we write \glssymbol{fzeta}$F(\zeta-1):=F(\zeta_1-1,\ldots,\zeta_d-1)$.   
Since $v(\zeta_j-1)>0$, we have $v(F(\zeta-1))\geq 0$. For $n\geq0$, we define \glssymbol{Wn}$W(n):=\{\xi\in W\mid \xi^{p^n}=1\}$ and 
\glssymbol{sum}\[ \Sigma_n (F):=\sum_{\zeta\in{W(n)^d}}v(F(\zeta-1)).\] 

\begin{example} 
(1) If $\sigma=\gamma_1\gamma_2=(1+T_1)(1+T_2)$, then $\overline{\sigma -1}=\overline{T_1} \overline{T_2}+\overline{T}_1+\overline{T}_2.$ 

(2) 
Suppose $p\neq 2$. If $\sigma=\gamma_1^2=(1+T_1)^2$, then $\overline{\sigma -1}=\overline{T}_1^2+2\overline{T}_1=\overline{T_1(T_1+2)}.$ 
Since $T_1+2\in\Lambda^*$, one has $(\overline{\sigma -1})=(\overline{T}_1)$ as ideals. In general, if $\sigma={\gamma}_1^k$ and $p\nmid k$, then $(\overline{\sigma -1})=(\overline{T}_1)$.
\end{example}
\begin{example}\label{2255} 
(1) 
For arbitrary $m\in\Z_{\geq 0}$, we have
\[
\Sigma_n(p^m)=\sum_{\zeta\in W(n)^d} v(p^m)=\sum_{\zeta}m v(p)=m\sum_{\zeta} 1=m p^{dn}.
\]

(2) 
If we fix a primitive $p^2$-th root of unity $\xi$, then
\begin{align*}
\Sigma_2(T_1)&=\sum_{\zeta\in W(2)^d} v(\zeta_1-1)\\
&=\big(v(\xi -1)+v(\xi^2-1)+\cdots+v(\xi^{p^2}-1)\big)p^{2(d-1)}\\
&=\left(\frac{p-1}{p-1}+\frac{p(p-1)}{p(p-1)}\right)p^{2(d-1)}=2p^{2(d-1)}.
\end{align*}
In general, we have
\[
\Sigma_n(T_1)=np^{n(d-1)}.
\]

(3) 
If $\gamma_1^{e_1}\cdots\gamma_d^{e_d}\in\Gamma\backslash\Gamma^p$, then we find that
\[
\Sigma_n((1+T_1)^{e_1}\cdots(1+T_d)^{e_d}-1)=\sum_{\zeta\in W(n)^d} v(\zeta_1^{e_1}\cdots\zeta_d^{e_d}-1)=\sum_{\zeta\in W(n)^d} v(\zeta_1-1)=\Sigma_n(T_1).
\]
\end{example}
\begin{lem}[{\cite[Lemma 1.4, 1.5, 1.6]{CuocoMonsky1981}}]\label{2256}
{\rm (1)} 
Let $G\in\Lambda$ with $\overline{F}=\overline{G}\neq 0$. Then
\[
\Sigma_n(F)-\Sigma_n(G)=O(p^{(d-1)n}).
\]

{\rm (2)} 
If there exist $F_1, F_2\in\Lambda$ such that $F=F_1F_2$, then
\[
\Sigma_n(F)=\Sigma_n(F_1)+\Sigma_n(F_2)+O(p^{(d-1)n}).
\]

{\rm (3)} 
If $\mu(F)=\lambda(F)=0$, then $\Sigma_n(F)=O(p^{(d-1)n})$.
\end{lem}

One can show \Cref{2256} by using the results of Monsky {\cite{Monsky1981PS}} on $\Lambda$.

\begin{prop} [{\cite[Theorem 1.7]{CuocoMonsky1981}}]\label{2259} 
Let $0\neq F \in \Lambda$. 
Then we have
\[
\Sigma_n(F)=\left(\mu(F) p^n+\lambda(F) n+O(1)\right)p^{(d-1)n}.
\]
\end{prop}

\begin{proof}
Write $F=p^{\mu(F)}F_0$ with $\overline{F}_0\neq 0$. Since $\Omega$ is a UFD, 
there exist $F_1,\ldots, F_k\in\Lambda$ such that $\overline{F}_i$'s are irreducible in $\Omega$ and 
$\overline{F}_0=\overline{F}_1\cdots\overline{F}_k$ holds. 
By \Cref{2256} (1) and (2), we have
\[
\Sigma_n(F)=\Sigma_n(p^{\mu(F)})+\Sigma_n(F_1)+\cdots+\Sigma_n(F_k)+O(p^{(d-1)n}).
\]
By \Cref{2256} (3), if $\lambda(F_j)=0$, then $\Sigma_n(F_j)=O(p^{(d-1)n})$. By \Cref{2255} (1), (2), (3), we have
\begin{align*}
\Sigma_n(F)&=\mu(F) p^{dn}+\lambda(F) n p^{(d-1)n}+O(p^{(d-1)n})\\
&=\left(\mu(F) p^n+\lambda(F) n+O(1)\right)p^{(d-1)n}. \qedhere 
\end{align*}
\end{proof}
\begin{remark} 
In the $d=1$ case, for any $0\neq F\in \Zp[\![T]\!]$, the $p$-adic Weierstrass preparation theorem yields a unique description $F\doteq p^\mu(T^\lambda+(\text{lower terms}))$ that determines the $\mu$ and $\lambda$ of $F$. 
\Cref{2259} ensures that \Cref{def.mulambda} is a natural generalization for $d\geq 2$ cases. 
\end{remark}

\subsection{Results of Monsky}
In this subsection, we briefly review a result of Monsky \cite{Monsky1981PS}, which is vital for our proof of \Cref{mr.ZHS}.
Let $E_d$ denote the $\Zp$-module $\Hom(W^d,W)$.

\begin{definition}
$S\subset W^d$ is said to be \emph{semi-algebraic} if it is a finite union of subsets each of which is defined by finitely many conditions of the following three types
\begin{enumerate}[label=(\alph*)]
\item[i)] $\tau(\zeta)=\varepsilon$,
\item[ii)] $\tau(\zeta)\neq\varepsilon$,
\item[iii)] $\log_p|\langle\tau(\zeta)\rangle|\geq \log_p|\langle\tau'(\zeta)\rangle|+r$,
\end{enumerate}
where $\tau,\tau'\in E_d$, $\varepsilon\in W$, and $r\in\Z$.
\end{definition}

\begin{lem}
$(W\setminus\{1\})^d$ is semi-algebraic. 
\end{lem}

\begin{proof}
For each $1\leq i\leq d$, let $\pi_i$ denote the projection $W^d\surj W$ to the $i$-th component. Then we have
\[
(W\setminus\{1\})^d=\bigcap_{1\leq i\leq d}\{\zeta\in W^d\mid \pi_i(\zeta)\neq 1\}. \qedhere 
\]
\end{proof}

\begin{prop}[{\cite[Theorem 5.6]{Monsky1981PS}}]\label{22521}
Let $S\subset W^d$ be a semi-algebraic set and let $F\in\Lambda$. Then there exists a unique $f(U,V)\in\Q[U,V]$ with $\deg_{V}f\leq1$ and total degree $\deg f\leq d$ such that, for every sufficiently large $n$, 
\[\sum_{\zeta\in S\cap W(n)^d} v(F(\zeta-1))=f(p^n, n)\] 
holds. 
\end{prop}

\section{Characteristic elements} \label{sec.char} 

In this section, we study some basic properties of the Fitting ideals and the characteristic elements of modules 
over an integrally closed Noetherian domain $R$, and their behavior in the completions. 
This will be used to discuss the Alexander polynomials later. 

We intend to state the properties in a slightly general setting, so that they are applicable to, for instance, the Alexander modules of knot group representations over the ring of integers of a number field that is not necessarily a UFD.

\subsection{Fitting ideals and Characteristic ideals} 
Let \glssymbol{ring}$R$ be a Noetherian integrally closed domain and 
let \glssymbol{module}$\mathcal{M},\mathcal{N}$ be finitely generated $R$-modules. 
We say that $\mca{M}$ is \emph{pseudo-isomorphic} to $\mca{N}$ if 
there is an $R$-homomorphism $f:\mathcal{M}\to\mathcal{N}$
whose localization $f_\p:\mathcal{M}_{\mathfrak{p}}\to \mathcal{N}_{\mathfrak{p}}$ at every hight 1 prime $\p$ in $R$ is an isomorphism.  

Since $R$ is Noetherian, we may choose an exact sequence
\[
R^r\to R^s\to \mathcal{M}\to 0.
\]
The ideal of $R$ generated by the $s$-subdeterminants of the presentation matrix of $R^r\to R^s$ is called \emph{the Fitting ideal} of $\mathcal{M}$, and we denote it by \glssymbol{fitt}${\rm Fitt}(\mathcal{M})$. If $r<s$, then we define ${\rm Fitt}(\mathcal{M})=0$. It is known that the definition of Fitting ideals is independent of the choices of exact sequences. Details of the theory of Fitting ideals are in \cite[Chapter 3]{Northcott1976FFR}.

The \emph{divisorial hull} \glssymbol{dh}d.h.$\mathfrak{a}$ of a fractional ideal $\mathfrak{a}$ of $R$ is the intersection of the principal ideals that contain $\mathfrak{a}$. We refer to \cite[Lemma 0.1]{Ueki10} for basic properties of d.h.$\mathfrak{a}$. 
If $R$ is a UFD, then d.h.$\mathfrak{a}$ coincides with the principal ideal generated by the gcd of generators of $\mf{a}$.

\begin{prop}\label{11241}
Let $\mathcal{M}$ be a finitely generated $R$-module. If ${\rm d.h.}{\rm Fitt}\,\mathcal{M} \neq 0$, then $\mathcal{M}$ is a torsion $R$-module.
\end{prop}

\begin{proof} 
By the definition of the divisorial hull, if ${\rm Fitt}\,\mathcal{M}=0$, then ${\rm d.h.}{\rm Fitt}\,\mathcal{M}=0$. 
By the theory of Fitting ideals \cite[Chapter 3, Section 3, Theorem 5]{Northcott1976FFR}, we have \glssymbol{annihilator}${\rm Fitt}\,\mathcal{M}\subset{\rm Ann}_R\mathcal{M}$ (and ${\rm Fitt}\,\mathcal{M}={\rm Ann}_R\mathcal{M}$ holds if $\mathcal{M}$ is generated by one element). 
Thus, the assumption ${\rm d.h.}{\rm Fitt}\,\mathcal{M}\neq 0$ implies ${\rm Fitt}\,\mathcal{M}\neq 0$, and hence ${\rm Ann}_R\mathcal{M}\neq 0$. Therefore, $\mathcal{M}$ is a torsion $R$-module. 
\end{proof}

\begin{definition} \label{def.char}
An $R$-module of the form $\bigoplus_{i=1}^s R/\mathfrak{p}_i^{m_i}$, where $\mathfrak{p}_i$ are height 1 prime ideals in $R$, is called an \emph{elementary $R$-module}.

Let $\mathcal{M}$ be a finitely generated $R$-module. If $\mathcal{M}$ is a torsion $R$-module, then by \cite[Proposition 3.1.6]{Sharifi.IT2023OCT} or \cite[Theorem 2.3.6]{Ochiai2023Iwasawa1en}, there is a unique elementary $R$-module $\mathcal{E}:=\bigoplus_{i=1}^s R/\mathfrak{p}_i^{m_i}$ such that $\mathcal{M}$ is pseudo-isomorphic to $\mathcal{E}$. 
We define the characteristic ideal of $\mathcal{M}$ by \glssymbol{Char}${\rm Char}\,\mathcal{M}=\prod_i \mathfrak{p}_i^{m_i}$. If instead $\mathcal{M}$ is not torsion, then we define ${\rm Char}\,\mathcal{M}=0$. 
If $R$ is a UFD, then every height 1 prime ideal is principal, and \emph{the characteristic element} of $\mathcal{M}$ is defined to be a generator of the principal ideal ${\rm Char}\,\mathcal{M}$ up to multiplication by units.
\end{definition} 

\begin{prop}\label{11242}
Let $\mathcal{M}$ be a finitely generated torsion $R$-module. Then we have
\[
{\rm Char}\,\mathcal{M}= {\rm d.h.}{\rm Fitt}\,\mathcal{M}.
\]
\end{prop}

\begin{proof}
Since $\mathcal{M}$ is pseudo-isomorphic to an elementary $R$-module $\mathcal{E}:=\bigoplus_{i=1}^s R/\mathfrak{p}_i^{m_i}$, for each height 1 prime ideal $\mathfrak{p}$ in $R$, we have
\[
\mathcal{M}_{\mathfrak{p}}\cong (\bigoplus_i R/\mathfrak{p}_i^{m_i})_{\mathfrak{p}}.
\]
This implies
\[
({\rm Fitt}\,\mathcal{M})_{\mathfrak{p}}={\rm Fitt}(\mathcal{M}_{\mathfrak{p}})={\rm Fitt}\big((\bigoplus_i R/\mathfrak{p}_i^{m_i})_{\mathfrak{p}}\big)=\big({\rm Fitt}\bigoplus_i R/\mathfrak{p}_i^{m_i}\big)_{\mathfrak{p}}.
\]
Since ${\rm Fitt}(\bigoplus R/{\mathfrak{p}_i^{m_i}})=\prod_i \mathfrak{p}_i^{m_i}$, we have
\[
({\rm Fitt}\,\mathcal{M})_{\mathfrak{p}}=\mathfrak{p}^{\sum_{\mathfrak{p}_i=\mathfrak{p}} m_i}.
\]
By \cite[Lemma 0.1 (1)]{Ueki10}, %\cite[Lemma 3.2]{Hillman2}, 
we have
\[
{\rm d.h.}{\rm Fitt}\,\mathcal{M}=\bigcap_{\mathfrak{p}}({\rm Fitt}\,\mathcal{M})_{\mathfrak{p}}.
\]
Therefore, we obtain
\[
{\rm d.h.}{\rm Fitt}\,\mathcal{M}={\rm d.h.}\big(\prod_i\mathfrak{p}_i^{m_i}\big)={\rm d.h.}{\rm Char}\,\mathcal{M}={\rm Char}\,\mathcal{M}. \qedhere 
\]
\end{proof}

\subsection{Completions} 
Here we discuss the completions of modules and the characteristic ideals. 

\begin{lem}\label{01162}
Let \glssymbol{maximalideal}$\mathfrak{m}$ be a maximal ideal of $R$ and let \glssymbol{rcompletion}$\widehat{R}$ denote the $\mathfrak{m}$-completion of $R$. Suppose $R$ is a UFD. Let $a_1,\ldots,a_n$ be elements in $R$ and let $g$ be their greatest common divisor in $R$. Then $g$ is their greatest common divisor in $\widehat{R}$ as well.
\end{lem}
\begin{proof}
By the definition of $g$, for each $1\leq i\leq n$, there exists $a'_i\in R$ such that $a_i=g a'_i$. Suppose that there exists a prime element $\widehat{P}$ of $\widehat{R}$ and there exists $a_i''$ for each $1\leq i\leq d$ such that  $a_i'=\widehat{P}a''_i$. Since completions satisfy the going down property, $\widehat{P}\widehat{R}\cap R$ is a height 1 prime ideal of $R$. Since $R$ is a UFD, there exists a prime element $P$ of $R$ such that $PR=P\widehat{R}\cap R$. Therefore, for each $1\leq i\leq n$, $a_i$ can be divided by $P$ in $R$. This contradicts the fact that $g$ is a greatest common divisor in $R$. Therefore, there exists no such a prime element, and so $g$ is a greatest common divisor in $\widehat{R}$ as well.
\end{proof}

\begin{prop} \label{1124} 
Let $\mf{m}$ be a maximal ideal of $R$ such that $R_{\mf m}$ is a UFD, and 
let $\wh{R}$ denote the $\mf{m}$-completion of $R$. 
Let $\mathcal{M}$ be a finitely generated torsion $R$-module and let \glssymbol{mcompletion}$\wh{\mca{M}}$ denote the completed module defined by $\varprojlim_n \mathcal{M}\otimes R/\mf{m}^n$. Then we have 
\[\wh{R}\,{\rm Char}\,\mca{M} ={\rm Char}\,\wh{\mca{M}}.\] 
\end{prop} 

\begin{proof} 
Since $\otimes R/\mf{m}^n$, localizations, and completions in our setting are right exact, a presentation of $\mca{M}$ yields that of $\mca{M}_\mf{m}$ and $\wh{\mca{M}}$, and hence ${\rm Fitt}_R \mca{M}$, ${\rm Fitt}_{R_\mf{m}} \mca{M}_\mf{m}$, and ${\rm Fitt}_{\wh{R}} \wh{\mca{M}}$ are generated by the same generators. 
By \cite[Lemma 0.1 (2)]{Ueki10}, ${\rm d.h.}$ is a local property, so $R_\mf{m}\,{\rm d.h.}{\rm Fitt}_R \mca{M}=({\rm d.h.}{\rm Fitt}_R \mca{M})_\mf{m}={\rm d.h.}{\rm Fitt}_{R_\mf{m}} \mca{M}_\mf{m}$, 
and hence $\wh{R}\,{\rm d.h.}{\rm Fitt}_R \mca{M}=\wh{R}\,{\rm d.h.}{\rm Fitt}_{R_\mf{m}}\mca{M}_\mf{m}$. 
Note that $\wh{R}$ may be seen as the completion of $R_{\mf{m}}$. 
Since $R_\mf{m}$ is a UFD, so is $\wh{R}$, and hence the notion of ${\rm d.h.}$ coincides with the gcd of generators both in $R_\mf{m}$ and in $\wh{R}$. 
By \Cref{01162}, we obtain 
$\wh{R}\,{\rm d.h.}{\rm Fitt}_{R_\mf{m}} \mca{M}_\mf{m}={\rm d.h.}{\rm Fitt}_{\wh{R}} \wh{\mca{M}}$. 
Thus, we have $\wh{R}\,{\rm d.h.}{\rm Fitt}_R \mca{M}={\rm d.h.}{\rm Fitt}_{\wh{R}} \wh{\mca{M}}$, and 
\Cref{11242} yields the assertion. 
\end{proof} 

\begin{remark} If \glssymbol{ringofintegers}$\mca{O}_{k}$ is the ring of integers of a number field $k$, then $R=\mca{O}_{k}[t_1^\Z,\ldots,t_d^\Z]$ is an integrally closed Noetherian closed domain such that the localization $R_{\mf m}$ at the maximal ideal $\mf{m}=(\mf{p},t_1-1,\ldots,t_d-1)$ is a UFD, so \Cref{1124} is applicable. 
\end{remark} 

\begin{prop} \label{prop.hathat}
Let $\mca{M}$ be a finitely generated $\Lambda_{\Z}$-module, 
and let $\wh{M}$ denote the completion with respect to the maximal ideal $\mf{m}=(p,t_1-1,\ldots,t_d-1)\subset \Lambda_\Z$. 
For each $n\in \Z_{>0}$, write $\mca{I}_n=(t_1^n-1,\ldots,t_d^n-1)\subset \Lambda_\Z$. 
Then $\wh{\mca{M}}$ may be naturally identified with the $\Lambda$-module 
$\mca{M}\otimes_{\Lambda_\Z} \Lambda=\varprojlim_n \mca{M}\otimes_{\Lambda_\Z} \Lambda/\mca{I}_{p^n}=\varprojlim_n \mca{M}/\mca{I}_{p^n}\mca{M}\otimes \Zp$. 
\end{prop} 

\begin{proof} 
Recall the embedding $\Lambda_\Z=\Z[t_1^\Z,\ldots, t_d^\Z]=\Z[t_1^{\pm 1},\ldots, t_d^{\pm 1}] \inj \Zp[\![t_1^{\Zp},\ldots, t_d^{\Zp}]\!]\cong \Lambda=\Zp[\![T_1,\ldots, T_d]\!]; t_i\mapsto 1+T_i$ 
obtained via the Iwasawa--Serre isomorphism. 
For each $m\geq n\in \Z_{\geq 0}$, 
$\mca{M}/\mca{I}_{p^n}\mca{M}\otimes \Zp$ is a $\Zp[t_1^{\Z/p^m\Z},\ldots,t_d^{\Z/p^m\Z}]$-module. 
Hence, it is a $\Lambda$-module, and so is the projective limit. 
By \cite[Lemma 2.3.8]{Sharifi.IT2023OCT}, we have a natural identification $\Lambda
=\varprojlim_n \Lambda_\Z/\mf{m}^n$, which yields the assertion. 
\end{proof}

\section{Lemmas on $\Lambda$-modules} 
In this section, we study the sizes of the torsion subgroups and the ranks of quotients of $\Lambda$-modules. We recall some results of Cuoco--Monsky \cite{CuocoMonsky1981} and Monsky \cite{Monsky1989}, as well as modify their results for our use. 
For each $n\geq 1$ and $1\leq j\leq d$, 
define \[\omega_{p^n}(T_j):=(1+T_j)^{p^n}-1,\ \ \nu_{p^n}(T_j):=\frac{(1+T_j)^{p^n}-1}{T_j}\]
and let $\mathcal{I}_{p^n}, \mathcal{J}_{p^n}$ denote the ideals of $\Lambda$ generated by 
\glssymbol{omegapoly}$\{\omega_{p^n}(T_j) \mid 1\leq j\leq d\}$ and \glssymbol{nupoly}$\{\nu_{p^n}(T_j)\mid 1\leq j\leq d\}$ respectively. 

\subsection{Torsion estimates} Here, we recall some results of Cuoco--Monsky \cite{CuocoMonsky1981} and Monsky \cite{Monsky1989} 
on torsion size estimates for quotients of $\Lambda$-modules. 

Let $S$ be a finite set of $\Gamma\setminus\Gamma^p$. Let $\mathcal{M}$ be a finitely generated $\Lambda$-module. Let $r\in\Z_{\geq 0}$. Suppose that, for each $\sigma\in S$, a submodule $\mathcal{M}_{\sigma}$ of $\mathcal{M}$ containing $(\sigma^{p^r}-1)\mathcal{M}$ is given. For $n\geq r$, let $\mathcal{B}_{p^n}:=\mathcal{I}_{p^n}\mathcal{M}+\sum_{\sigma\in S}\frac{\sigma^{p^n}-1}{\sigma^{p^r}-1}\mathcal{M}_{\sigma}$ and let $G_n:=\mathcal{M}/\mathcal{B}_{p^n}$.

\begin{prop}[{\cite[Theorem 4.13]{CuocoMonsky1981}, \cite[Theorem 3.12]{Monsky1989}}]
Let $\mathcal{M}$ be a finitely generated torsion $\Lambda$-module with characteristic element $F={\rm char}\,\mca{M}$. Suppose that $r(G_n)=O(p^{(d-2)n})$. Then there exists $\mu_1\in\R$ such that
\[
e(G_n)=\mu(F) p^{dn}+\lambda(F) np^{(d-1)n}+\mu_1p^{(d-1)n}+O(np^{(d-2)n}).
\]
Moreover, if $d=2$, then $\mu_1\in\Q$.
\end{prop}

In particular, let $\gamma_1,\ldots,\gamma_d\in \Gamma$ be elements corresponding to $1+T_1,\ldots,1+T_d$ via the Iwasawa--Serre isomorphism $\Zp[\![\Gamma]\!]\cong\Zp[\![T_1,\ldots,T_d]\!]$ respectively. Let $S:=\{\gamma_1,\ldots,\gamma_d\}$ and let $\mathcal{M}_{\gamma_i}=\mathcal{M}$ for every $1\leq i\leq d$. Let $r=0$. Then we find that
\[
\mathcal{B}_{p^n}=\mathcal{I}_{p^n}\mathcal{M}+\sum_{1\leq i\leq d}\nu_{p^n}(T_i)\mathcal{M}=\mathcal{J}_{p^n}\mathcal{M}
\]
Thus we obtain the following, where the assertion (2) is the case with $S=\emptyset$.   
\begin{cor} \label{Monsky21} 
Let $\mathcal{M}$ be a finitely generated torsion $\Lambda$-module with characteristic element $F={\rm char}\,\mca{M}$. 

{\rm (1)} 
If $r(\mathcal{M}/\mathcal{J}_{p^n}\mathcal{M})=O(p^{(d-2)n})$, then 
there exists $\mu_1\in\R$ such that
\[
e(\mathcal{M}/\mathcal{J}_{p^n}\mathcal{M})=\mu(F) p^{dn}+\lambda(F) np^{(d-1)n}+\mu_1p^{(d-1)n}+O(np^{(d-2)n}).
\]

{\rm (2)} 
If $r(\mathcal{M}/\mathcal{I}_{p^n}\mathcal{M})=O(p^{(d-2)n})$, then 
there exists $\mu_1\in\R$ such that
\[
e(\mathcal{M}/\mathcal{I}_{p^n}\mathcal{M})=\mu(F) p^{dn}+\lambda(F) np^{(d-1)n}+\mu_1p^{(d-1)n}+O(np^{(d-2)n}).
\] 
\end{cor}

\begin{remark}
The assumption $r(\mathcal{M}/\mathcal{I}_{p^n}\mathcal{M})=O(p^{(d-2)n})$ is equivalent to ``$\Delta(1+T_1,\ldots,1+T_d)$ has no special prime factors in $\Lambda$" in the sense of Cuoco--Monsky \cite[Theorem 3.13]{CuocoMonsky1981}. One of the purposes of this section is to weaken this rank assumption of 
\Cref{Monsky21} (2). \end{remark}

\subsection{Rank estimates} 
Here, we obtain rank estimates of quotients of $\Lambda$-modules, by modifying Cuoco--Monsky's results in \cite{CuocoMonsky1981}. 
Recall that $W(n)$ denotes the set of $p^n$-th roots of unity. 

\begin{prop}\label{CM32}
Let $0\neq F\in\Lambda$. Then 
\[
\sum_{\zeta\in(W(n)\setminus\{1\})^d}v(F(\zeta-1))=\left(\mu(F) p^n+\lambda(F) n+O(1)\right)p^{(d-1)n}.
\]
\end{prop}
\begin{proof}
We have
\[
\sum_{\zeta\in{W(n)^d}}v(F(\zeta-1))=\sum_{\zeta\in(W(n)\setminus\{1\})^d}v(F(\zeta-1))+\sum_{\rm others} v(F(\zeta-1)).
\]
By \Cref{22521}, we have
\[
\sum_{\rm others} v(F(\zeta-1))=O(p^{(d-1)n}).
\]
By \Cref{2259}, we complete the proof.
\end{proof}

\begin{lem}[{\cite[Consequence of Theorem 3.14]{CuocoMonsky1981}}]\label{CM1}
Let $\mathcal{M}$ be a finitely generated $\Lambda$-module.
Then we have
\[
r(\mathcal{M}/\mathcal{I}_{p^n} \mathcal{M})=O(p^{dn}).
\]
Moreover, if $\mathcal{M}$ is torsion over $\Lambda$, then
\[
r(\mathcal{M}/\mathcal{I}_{p^n}\mathcal{M})=O(p^{(d-1)n})
\]
\end{lem}

\begin{prop}\label{CM2}
Let $\mathcal{M}$ be a finitely generated $\Lambda$-module. Then we have
\[
r(\mathcal{M}/\mathcal{J}_{p^n}\mathcal{M})=O(p^{dn})
\]
If $\mathcal{M}$ is torsion over $\Lambda$, then
\[
r(\mathcal{M}/\mathcal{J}_{p^n}\mathcal{M})=O(p^{(d-1)n}).
\]
\end{prop}
\begin{proof}
We prove the assertion by a mathematical induction on $d$. When $d=1$, then, by the one variable Iwasawa theory, we have
\[
r(\mathcal{M}/\nu_{p^n}(T)\mathcal{M})=
\begin{cases}
O(1)&\mbox{if }\mathcal{M}\mbox{ is torsion}\\
O(p^{n})&\mbox{if not}.
\end{cases}
\]
Suppose $d\geq 2$. Since
\begin{align*}
(\mathcal{M}/((1+T)^{p^n}-1)\mathcal{M})\otimes\Q_p 
&=(\mathcal{M}\otimes(\Lambda/((1+T)^{p^n}-1)\Lambda))\otimes\Q_p\\
&=(\mathcal{M}\otimes((\Lambda/T\Lambda)\oplus(\Lambda/\nu_{p^n}(T)\Lambda))\otimes\Q_p\\
&=((\mathcal{M}/T\mathcal{M})\oplus(\mathcal{M}/\nu_{p^n}(T)\mathcal{M}))\otimes\Q_p,
\end{align*}
we have
\begin{align*}
r(\mathcal{M}/\mathcal{I}_{p^n}\mathcal{M})
&=r\big(\bigoplus_f \mathcal{M}/(f_1(T_1),\ldots,f_d(T_d))\big)\\
&=r\big((\mathcal{M}/\mathcal{J}_{p^n}\mathcal{M})\oplus\bigoplus_{\rm others}\mathcal{M}/(f_1(T_1),\ldots,f_d(T_d))\mathcal{M}\big),
\end{align*}
where $f=(f_i)_i$ runs through $\prod_i \{T_i, \nu_{p^n}(T_i)\}$  
and ``others'' stands for $f\neq (\nu_{p^n}(T_i))_i$. 
By induction hypothesis, we must have
\[
r(\mathcal{M}/\mathcal{J}_{p^n}\mathcal{M})=r(\mathcal{M}/\mathcal{I}_{p^n}\mathcal{M})+O(p^{(d-1)n}).
\]
By \Cref{CM1}, we obtain the assertion.
\end{proof}

\subsection{Replacing $\mathcal{I}_{p^n}$ by $\mathcal{J}_{p^n}$} 
In this section, we replace $\mathcal{I}_{p^n}$ by $\mathcal{J}_{p^n}$ in several results in \cite[Sections 2,3]{CuocoMonsky1981}. 
Although most of them are obtained by slight modifications of the original arguments, 
we attach proofs for the convenience of readers.

We define an equivalence relation on $(W(n)\setminus\{1\})^d$ in the following way; 
we say that $\zeta=(\zeta_1,\ldots,\zeta_d)$ and $\zeta'=(\zeta'_1,\ldots,\zeta'_d)$ in $(W(n)\setminus\{1\})^d$ are equivalent if there exists some $\sigma\in {\rm Gal}(\ol{\Q}_p/\Qp)$ satisfying 
$(\sigma(\zeta'_i))_i=(\zeta_i)_i$. 
Suppose that $\zeta\in (W(n)\setminus\{1\})^d$ runs through a complete representative system of the equivalent classes unless otherwise mentioned, and define a $\Lambda$-homomorphism by \glssymbol{varphi}
\[
\varphi_n:\Lambda/\mathcal{J}_{p^n}\Lambda\to\bigoplus_\zeta \Zp[\zeta]; F\mapsto (F(\zeta-1))_\zeta.
\] 
By \cite[Theorem 2.1]{CuocoMonsky1981}, we have that $\coker\varphi_n$ is annihilated by $p^{dn}$. By the definition of $\mathcal{J}_{p^n}$, we have $r(\Lambda/\mathcal{J}_{p^n}\Lambda)=(p^n-1)^d$. 
On the other hand, since the number of elements in the equivalent class of each $\zeta$ is $r(\Zp[\zeta])$, we have 
\[
r(\bigoplus_\zeta \Zp[\zeta])=(p^n-1)^d.
\]
Since both $\Lambda/\mathcal{J}_{p^n}\Lambda$ and $\bigoplus_\zeta \Zp[\zeta]$ are free $\Zp$-modules, we conclude that $\varphi_n$ is an injection with finite cokernel annihilated by $p^{dn}$. 
This argument persists if we replace $\mca{J}_{p^n}$ by any ideal $(f_1,\ldots,f_d)$ with $f_i\in \{\omega_{p^n}(T_i),\nu_{p^n}(T_i)\}$. 

\begin{lem}[{\cite[Lemma 2.4]{CuocoMonsky1981}}]{\label{CM10}}
Let $\mathcal{M}$ be a finitely generated free $\Zp$-module and let $\mathcal{M}'$ be its submodule of finite index. Let $\psi:\mathcal{M}\to \mathcal{M}$ be an endomorphism such that $\psi(\mathcal{M}')\subset \mathcal{M}'$. Suppose that there exists some $a\geq 0$ such that $p^a \mathcal{M}\subset \mathcal{M}'$. Then
\[
|e(\mathcal{M}/\psi(\mathcal{M}))-e(\mathcal{M}'/\psi(\mathcal{M}'))|\leq a(r(\ker \psi))).
\]
\end{lem}

\begin{prop}[{\cite[Theorem 2.5]{CuocoMonsky1981}} with $\mathcal{I}_{p^n}$ replaced by $\mathcal{J}_{p^n}$]\label{CM5}
Let $0\neq F\in\Lambda$ and $\mathcal{M}=\Lambda/(F)$. Then 
\[
e(\mathcal{M}/\mathcal{J}_{p^n} \mathcal{M})=\left(\mu(F) p^n+O(n)\right)p^{(d-1)n}.
\]
Moreover, if $r(\mathcal{M}/\mathcal{I}_{p^n}\mathcal{M})=O(p^{(d-2)n})$, then 
\[
e(\mathcal{M}/\mathcal{J}_{p^n} \mathcal{M})=
\left(\mu(F) p^n+\lambda(F) n+O(1)\right)p^{(d-1)n}.
\]
\end{prop}
\begin{proof}
Consider the commutative diagram
\[
\xymatrix{
\Lambda/\mathcal{J}_{p^n}\ar[r]^{\times F}\ar[d]_{\varphi_n} & \Lambda/\mathcal{J}_{p^n}\ar[d]_{\varphi_n}\\
\bigoplus_\zeta\Zp[\zeta]\ar[r]^{\times F(\zeta-1)}&\bigoplus_\zeta \Zp[\zeta].
}
\]
By \Cref{CM10}, we have
\[
|e(\Lambda/(\mathcal{J}_{p^n},F))-e(\bigoplus_\zeta\mathbb{Z}_p[\zeta]/(F(\zeta-1))|\leq dn(\ker (\times F(\zeta-1))).
\]
Since
\[
0\to\ker(\times F(\zeta-1))\to\bigoplus_\zeta\Zp[\zeta]\to\bigoplus_\zeta\Zp[\zeta]\to\coker(\times F(\zeta-1))\to 0
\]
is exact, we have
\[
r(\ker(\times F(\zeta-1)))=r(\coker(\times F(\zeta-1)).
\]
Since
\[
0\to\Lambda/\mathcal{J}_{p^n}\stackrel{\varphi_n}\to\bigoplus_\zeta \Zp[\zeta]\to\coker\varphi_n\to 0
\]
is exact, 
\[
\Lambda/(\mathcal{J}_{p^n},F)\to\bigoplus_\zeta\Zp[\zeta]/(F(\zeta-1))\to\coker\varphi_n/(\overline{F(\zeta-1)})\to 0
\]
is also exact. Since $\coker\varphi_n$ is finite, we have an exact sequence
\[
\Lambda/(\mathcal{J}_{p^n},F)\otimes\Q_p\to \big(\bigoplus_\zeta \Zp[\zeta]/(F(\zeta-1))\big)\otimes\Q_p\to 0.
\]
This implies
\[r(\Lambda/(\mathcal{J}_{p^n}, F))\geq r\big(\bigoplus_\zeta \Zp[\zeta]/(F(\zeta-1))\big)=r(\coker(\times(F(\zeta-1)).\] 
By \Cref{CM2}, we have
\[
\Big|e(\Lambda/(\mathcal{J}_{p^n}, F))-e\big(\bigoplus_\zeta \Zp[\zeta]/(F(\zeta-1))\big)\Big|\leq O(np^{(d-1)n}).
\]
Moreover, if we assume $r(\mathcal{M}/\mathcal{J}_{p^n} \mathcal{M})=O(p^{(d-2)n})$, then we have
\[
\Big|e(\Lambda/(\mathcal{J}_{p^n}, F))-e\big(\bigoplus_\zeta \Zp[\zeta]/(F(\zeta-1))\big)\Big|\leq O(np^{(d-2)n}).
\]
Hence it suffices to estimate $e(\bigoplus_\zeta \Zp[\zeta]/(F(\zeta-1)))$. Since
\[
\Big|\bigoplus_\zeta\Zp[\zeta]/F(\zeta-1)\Big|=\prod_\zeta N_{\Q_p(\zeta)/\Q_p} F(\zeta-1)=\prod_{\zeta\in (W(n)\setminus\{1\})^d}F(\zeta-1),
\]
where $N_{\Q_p(\zeta)/\Q_p}$ denotes the norm map, by \Cref{CM32}, we obtain
\begin{align*}
e\big(\bigoplus_\zeta\Zp[\zeta]/F(\zeta-1)\big)&=\sum_{\zeta\in (W(n)\setminus\{1\})^d} v(F(\zeta-1))\\
&=(\mu p^n+\lambda n+O(1))p^{(d-1)n}. \qedhere
\end{align*}
\end{proof}

Let $\mathcal{M}$ be a finitely generated $\Lambda$-module. For each $n\geq 0$, let \glssymbol{Phi}$\Phi_n:\mathcal{M}\to\bigoplus_\zeta \mathcal{M}_{\zeta}$ denote the natural $\Lambda$-homomorphism. 

\begin{lem}[{\cite[Lemma 2.7]{CuocoMonsky1981}} with $\mathcal{I}_{p^n}$ replaced by $\mathcal{J}_{p^n}$] \label{CM31}
Both the kernel and cokernel of $\Phi_n:\mathcal{M}/\mathcal{J}_{p^n}\mathcal{M}\to\bigoplus_\zeta \mathcal{M}_{\zeta}$ are annihilated by $p^{dn}$.
\end{lem}

\begin{proof}
Consider an exact sequence of $\Lambda$-modules
\[
\Lambda/\mathcal{J}_{p^n}\Lambda\stackrel{\varphi_n}\to\bigoplus_\zeta \Zp[\zeta]\to\coker{\varphi_n}\to 0.
\]
Since tensoring $\mathcal{M}$ is right exact, we obtain the exact sequence
\[
\mathcal{M}/\mathcal{J}_{p^n} \mathcal{M}\stackrel{\Phi_n}\to\bigoplus_\zeta \mathcal{M}_{\zeta}\to (\coker\varphi_n)\otimes_{\Lambda} \mathcal{M}\to 0.
\]
Since $\coker \varphi_n$ is annhilated by $p^{dn}$ by \cite[Lemma 2.1]{CuocoMonsky1981}, so is $\coker \Phi_n=(\coker\varphi_n)\otimes_{\Lambda} \mathcal{M}$. Hence the statement for the cokernel holds.

To show the statement for the kernel, consider an exact sequence of $\Lambda/\mathcal{J}_{p^n}$-modules
\[
0\to\ker\pi\to \mathcal{F}\stackrel{\pi}\to \mathcal{M}/\mathcal{J}_{p^n}\mathcal{M}\to 0,
\]
where $\mathcal{F}$ is a finitely generated free $\Lambda/\mathcal{J}_{p^n}$-module. For $\zeta\in(W(n)\setminus\{1\})^d$, since tensoring $\Zp[\zeta]$ is right exact, we have a commutative diagram

\[
\xymatrix{
\ker\pi\ar[r]^{\iota}&\mathcal{F}\ar[r]^{\pi}\ar[d]_{\Psi_n}&\mathcal{M}/\mathcal{J}_{p^n}\mathcal{M}\ar[r]\ar[d]_{\Phi_n}&0\\
\bigoplus_\zeta(\ker\pi\otimes_{\Lambda}\Zp[\zeta])\ar[r]^{\iota '}&\bigoplus_\zeta (\mathcal{F}\otimes_{\Lambda}\Zp[\zeta])\ar[r]&\bigoplus_\zeta((\mathcal{M}/\mathcal{J}_{p^n}\mathcal{M})\otimes_{\Lambda}\Zp[\zeta])\ar[r]&0.
}
\]
Since
\begin{align*}
\mathcal{M}/\mathcal{J}_{p^n}\mathcal{M}\otimes_{\Lambda}\Zp[\zeta]&=(\mathcal{M}\otimes_{\Lambda}{\Lambda/\mathcal{J}_{p^n}})\otimes_{\Lambda}(\Lambda/\mathcal{J}_{p^n}\otimes_{\Lambda/\mathcal{J}_{p^n}}\Zp[\zeta])\\
&=\mathcal{M}\otimes_{\Lambda}\Lambda/\mathcal{J}_{p^n}\otimes_{\Lambda/\mathcal{J}_{p^n}}\Zp[\zeta])\\
&=\mathcal{M}\otimes_{\Lambda}\Zp[\zeta]=\mathcal{M}_{\zeta},
\end{align*}
the diagram becomes
\[
\xymatrix{
\ker\pi\ar[r]^{\iota}&\mathcal{F}\ar[r]^{\pi}\ar[d]_{\Psi_n}&\mathcal{M}/\mathcal{J}_{p^n}\mathcal{M}\ar[r]\ar[d]_{\Phi_n}&0\\
\bigoplus_\zeta(\ker\pi)_{\zeta}\ar[r]^{\iota '}&\bigoplus_\zeta \mathcal{F}
_{\zeta}\ar[r]&\bigoplus_\zeta \mathcal{M}_{\zeta}\ar[r]&0.
}
\]

Let $\overline{x}\in\ker\Phi_n$. Then there exists some $x\in \mathcal{F}$ such that $\pi(x)=\overline{x}$. By diagram chasing, there exists $y\in\bigoplus_\zeta(\ker\pi)_{\zeta}$ such that $\Psi_n(x)=\iota'(y)$. Let $\{u_i\}$ be a system of generators of $\Lambda/J$-module $\ker\pi$. Then $\{u_i\otimes 1\}$ generates $(\ker\pi)_\zeta$ over $\Zp[\zeta]$ for each $\zeta$. Hence $y$ is of the form ($\sum_{i} a_{i,\zeta}(u_i\otimes 1))_{\zeta}$ with $a_{i,\zeta}\in\Zp[\zeta]$. This implies
\[
\Psi_n(x)=\iota'(y)=(\sum_i a_{i,\zeta}\iota'(u_i\otimes 1))_{\zeta}.
\]
Since $(a_{i,\zeta})_{\zeta}\in\bigoplus_\zeta \Zp[\zeta]$, by \cite[Lemma 2.1]{CuocoMonsky1981}, there exists $a_i\in\Lambda$ such that $\varphi_n(a_i)=(p^{dn}a_{i,\zeta})_{\zeta}$ for each $i$. Consider an element $p^{dn}x-\sum_i a_iu_i\in \mathcal{F}$. Then
\begin{align*}
\Psi_n(p^{dn}x-\sum_i a_iu_i) 
&=p^{dn}\Psi_n(x)-(\sum_i u_i\otimes p^{dn}a_{i,\zeta})_{\zeta}\\
&=(p^{dn}\sum_i a_{i,\zeta}(u_i\otimes 1))_{\zeta}-(\sum_{i} u_i\otimes p^{dn}a_{i,\zeta})_{\zeta}\\
&=0.
\end{align*}
Since $\varphi_n$ is injective and $\mathcal{F}$ is a free $\Lambda/\mathcal{J}_{p^n}$-module, $\Psi_n$ is also injective. This implies
\[
p^{dn}x=\sum_{i}a_i u_i
\]
in $\mathcal{F}$. Since $\{u_i\}$ is a system of generators of $\ker\pi$, we obtain
\[
p^{dn}\overline{x}=\pi(p^{dn}x)=0.
\]
This completes the proof.
\end{proof}

\begin{prop}[{\cite[Theorem 2.8]{CuocoMonsky1981}} with $\mathcal{I}_{p^n}$ replaced by $\mathcal{J}_{p^n}$]\label{CM34}
Let $\mathcal{M}$ be a finitely generated $\Lambda$-module. Then there exists $c\geq 0$ such that,
for every $n\geq 0$, 
 $p^{dn+c}$ annihilates $\tor_{\Zp} \mathcal{M}/\mathcal{J}_{p^n} \mathcal{M}$. %for every $n\geq 0$.
\end{prop}

\begin{proof}
By \cite[Lemma 2.6]{CuocoMonsky1981}, there exists some $c\geq 0$ such that
\[
p^c\tor_{\Zp} \mathcal{M}_\zeta=0
\]
for every $\zeta\in (W(n)\setminus\{1\})^d$. Let $x\in\tor_{\Zp}(\mathcal{M}/\mathcal{J}_{p^n}\mathcal{M})$. Then we have
\[
\Phi_n(x)\in\bigoplus_\zeta\tor_{\Zp}\mathcal{M}_{\zeta}.
\]
This implies $p^cx\in\ker\Phi_n$. By \Cref{CM31}, we have $p^{dn+c}x=0$. This completes the proof.
\end{proof}

\begin{lem}[{\cite[Lemma 3.1]{CuocoMonsky1981}} with $\mathcal{I}_{p^n}$ replaced by $\mathcal{J}_{p^n}$]\label{CM33}
Let $\mathcal{M}$ be a pseudo-null $\Lambda$-module. Then we have
\[
\dim_{\F_p}\mathcal{M}/(\mathcal{J}_{p^n},p)\mathcal{M}=O(p^{(d-1)n}).
\]
Moreover, if $p$ is not a zero-divisor on $\mathcal{M}$, then we have
\[
\dim_{\F_p}\mathcal{M}/(\mathcal{J}_{p^n},p)\mathcal{M}=O(p^{(d-2)n}).
\]
\end{lem}
\begin{proof}
Put $\overline{\mathcal{M}}=\mathcal{M}/p\mathcal{M}$ and $\Omega=\Lambda/p\Lambda$. Let $R=\Omega/{\rm Ann}_{\Omega}\overline{\mathcal{M}}$. Then, by the proof of \cite[Lemma 3.1]{CuocoMonsky1981}, for every $n\geq 0$, we have
\[
\dim_{\F_p}\overline{\mathcal{M}}/\mathfrak{m}^n\overline{\mathcal{M}}=O(n^{d-1}),
\]
where $\mathfrak{m}$ is a unique maximal ideal of $R$. Since $\nu_{p^n}(T_i)=T_i^{p^n-1}\mod p$, we have
\[
\mathfrak{m}^{d(p^n-1)}(\mathcal{M}/p\mathcal{M})\subset \mathcal{J}_{p^n}(\mathcal{M}/p\mathcal{M}).
\]
Thus we obtain
\[
\dim_{\F_p}\mathcal{M}/(\mathcal{J}_{p^n},p)\mathcal{M}=\dim_{\F_p}\overline{\mathcal{M}}/p\overline{\mathcal{M}}\leq \dim \overline{\mathcal{M}}/\mathfrak{m}^{d(p^n-1)}\overline{\mathcal{M}}=O(p^{n(d-1)}).
\]
Likewise, if $p$ is not a zero-divisor on $\mathcal{M}$, then, by the proof of \cite[Lemma 3.1]{CuocoMonsky1981}, we obtain
\[
\dim_{\F_p}\mathcal{M}/(\mathcal{J}_{p^n},p)\mathcal{M}=O(p^{(d-2)n}). \qedhere%\\[-7mm]
\]
\end{proof}
\begin{prop}[{\cite[Theorem 3.2]{CuocoMonsky1981}} with $\mathcal{I}_{p^n}$ replaced by $\mathcal{J}_{p^n}$]\label{CM35}
Let $\mathcal{M}$ be a pseudo-null $\Lambda$-module. Then we have

{\rm (1)} \ \ $r(\mathcal{M}/\mathcal{J}_{p^n}\mathcal{M})=O(p^{(d-2)n})$, 

{\rm (2)} \ \ $e(\mathcal{M}/\mathcal{J}_{p^n}\mathcal{M})=O(p^{(d-1)n})$.
\end{prop}
\begin{proof}
Let $\mathcal{M}(p):={\rm tor}_{\Zp}\mathcal{M}$ and let $\overline{\mathcal{M}}=\mathcal{M}/\mathcal{M}(p)$. Consider the exact sequence of $\Lambda$-modules
\[
0\to \mathcal{M}(p)\stackrel{\iota}\to \mathcal{M}\to \overline{\mathcal{M}}\to 0.
\]
This induces the exact sequence of finitely generated $\Zp$-modules
\begin{equation}\label{eqCM1}
\mathcal{M}(p)/\mathcal{J}_{p^n}\mathcal{M}(p)\stackrel{\overline{\iota}}\to \mathcal{M}/\mathcal{J}_{p^n}\mathcal{M}\to\overline{\mathcal{M}}/\mathcal{J}_{p^n}\overline{\mathcal{M}}\to 0.
\end{equation}
\noindent 
\ul{Proof of (1).} \ Since $\mathcal{M}(p)$ is $\Zp$-torsion,
\[
0\to(\mathcal{M}/\mathcal{J}_{p^n}\mathcal{M})\otimes_{\Zp}\Q_p\to(\overline{\mathcal{M}}/\mathcal{J}_{p^n}\overline{\mathcal{M}})\otimes_{\Zp}\Q_p\to 0
\]
is also exact. Hence we have
\[
r(\mathcal{M}/\mathcal{J}_{p^n}\mathcal{M})=r(\overline{\mathcal{M}}/\mathcal{J}_{p^n}\overline{\mathcal{M}}).
\]
By the structure theorem of finitely generated modules over principal ideal domains, we have
\[
r(\overline{\mathcal{M}}/\mathcal{J}_{p^n}\overline{\mathcal{M}})\leq \dim_{\F_p}\overline{\mathcal{M}}/(\mathcal{J}_{p^n}, p)\overline{\mathcal{M}}.
\]
By \Cref{CM33}, we have
\[
\dim_{\F_p}\overline{\mathcal{M}}/(\mathcal{J}_{p^n},p)\overline{\mathcal{M}}=O(p^{(d-2)n}).
\]
Putting everything together, we have
\[
r(\mathcal{M}/\mathcal{J}_{p^n}\mathcal{M})\leq O(p^{(d-2)n}).
\]

\noindent
\ul{Proof of (2).} \ By the exact sequence \eqref{eqCM1}, we have
\[
e(\mathcal{M}/\mathcal{J}_{p^n} \mathcal{M})\leq e(\im\overline{\iota})+e(\overline{\mathcal{M}}/\mathcal{J}_{p^n}\overline{\mathcal{M}}).
\]
Since $\mathcal{M}(p)$ is a torsion $\Zp$-module, we must have
\[
e(\im\overline{\iota})\leq e(\mathcal{M}(p)/\mathcal{J}_{p^n} \mathcal{M}(p)).
\]
Let $m$ be a positive integer such that $p^m\mathcal{M}(p)=0$. 
By the structure theorem of finitely generated modules over principal ideal domains and \Cref{CM33} (1), we have
\[
e(\mathcal{M}(p)/\mathcal{J}_{p^n} \mathcal{M}(p))=m \dim_{\F_p}(\mathcal{M}(p)/(\mathcal{J}_{p^n},p)\mathcal{M}(p))=O(p^{(d-1)n}).
\]
By the structure theorem, \Cref{CM34}, and \Cref{CM33} (2), we also have
\[
e(\overline{\mathcal{M}}/\mathcal{J}_{p^n}\overline{\mathcal{M}})\leq(dn+c)\dim_{\F_p}(\overline{\mathcal{M}}/(\mathcal{J}_{p^n},p)\overline{\mathcal{M}})=O(np^{(d-2)n}).
\]
for some $c\geq 0$. Putting everything together, we obtain the assertion.
\end{proof}

\begin{lem}[{\cite[Lemma 3.3]{CuocoMonsky1981}} with $\mathcal{I}_{p^n}$ replaced by $\mathcal{J}_{p^n}$]\label{CM4}
Suppose $\mathcal{M}$ is pseudo-isomorphic to $\mathcal{M}'$ as finitely generated torsion $\Lambda$-modules. Then we have
\[
|r(\mathcal{M}/\mathcal{J}_{p^n}\mathcal{M})-r(\mathcal{M}'/\mathcal{J}_{p^n}\mathcal{M}')|=O(p^{(d-2)n}).
\]
\end{lem}

\begin{proof}
Since $\mathcal{M}$ is pseudo-isomorphic to $\mathcal{M}'$, there is an exact sequence
\[
\mathcal{M}\stackrel{\iota}\to \mathcal{M}'\to \mathcal{M}''\to 0
\]
with $\mathcal{M}''$ pseudo-null. This induces an exact sequence
\begin{equation}\label{CM101}
\mathcal{M}/\mathcal{J}_{p^n}\mathcal{M}\stackrel{\overline{\iota}}\to \mathcal{M}'/\mathcal{J}_{p^n}\mathcal{M}'\to \mathcal{M}''/\mathcal{J}_{p^n}\mathcal{M}''\to 0.
\end{equation}
Since $\mathcal{M}''$ is pseudo-null, by \Cref{CM35} (1), we have
\begin{align*}
r(\mathcal{M}'/\mathcal{J}_{p^n}\mathcal{M}')
&\leq r(\mathcal{M}/\mathcal{J}_{p^n} \mathcal{M})+r(\mathcal{M}''/\mathcal{J}_{p^n}\mathcal{M}'')\\
&=r(\mathcal{M}/\mathcal{J}_{p^n}\mathcal{M})+O(p^{(d-2)n}).
\end{align*}
Reversing the roles of $\mathcal{M}$ and $\mathcal{M}'$, we obtain the assertion.
\end{proof}

\begin{lem}[{\cite[Lemma 3.3]{CuocoMonsky1981}}]\label{pigap}
Suppose $\mathcal{M}$ is pseudo-isomorphic to $\mathcal{M}'$ as finitely generated torsion $\Lambda$-modules. Then we have
\[
|e(\mathcal{M}/\mathcal{I}_{p^n}\mathcal{M})-e(\mathcal{M}'/\mathcal{I}_{p^n}\mathcal{M}')|=O(np^{(d-1)n}).
\]
\end{lem}

\begin{proof}
Since $\mathcal{M}$ is pseudo-isomorphic to $\mathcal{M}'$, there is an exact sequence
\[
\mathcal{M}\stackrel{\iota}\to \mathcal{M}'\to \mathcal{M}''\to 0
\]
with $\mathcal{M}''$ pseudo-null. This induces an exact sequence
\begin{equation}\label{CM101}
\mathcal{M}/\mathcal{I}_{p^n}\mathcal{M}\stackrel{\overline{\iota}}\to \mathcal{M}'/\mathcal{I}_{p^n}\mathcal{M}'\to \mathcal{M}''/\mathcal{I}_{p^n}\mathcal{M}''\to 0.
\end{equation}

Since the torsion part of $\im\overline{\iota}$ is composed of the image of ${\rm tor}_{\Zp}\mathcal{M}/\mathcal{I}_{p^n}\mathcal{M}$ and the image of a subset of the free part of $\mathcal{M}/\mathcal{I}_{p^n}\mathcal{M}$, by \cite[Theorem 2.8]{CuocoMonsky1981}, we must have
\[
e(\im\overline{\iota})\leq e(\mathcal{M}/\mathcal{I}_{p^n}\mathcal{M})+(dn+c)r(\mathcal{M}/\mathcal{I}_{p^n}\mathcal{M}).
\]
for some $c\geq 0$. Since $r(\mathcal{M}/\mathcal{I}_{p^n}\mathcal{M})=O(p^{(d-2)n})$ by \Cref{CM1}, we have
\begin{equation}\label{CM102}
e(\im\overline{\iota})\leq e(\mathcal{M}/\mathcal{I}_{p^n}\mathcal{M})+O(np^{(d-1)n}).
\end{equation}
By the exact sequence \eqref{CM101}, we also have
\[
e(\mathcal{M}'/\mathcal{I}_{p^n} \mathcal{M}')\leq e(\im\overline{\iota})+e(\mathcal{M}''/\mathcal{I}_{p^n}\mathcal{M}'').
\]
Since $\mathcal{M}''$ is pseudo-null, by \cite[Lemma 3.3]{CuocoMonsky1981}, we have
\begin{equation}\label{CM103}
e(\mathcal{M}'/\mathcal{I}_{p^n}\mathcal{M}')\leq e(\im\overline{\iota})+O(p^{(d-1)n}).
\end{equation}
By \eqref{CM102} and \eqref{CM103}, we obtain
\[
e(\mathcal{M}'/\mathcal{I}_{p^n}\mathcal{M}')\leq e(\mathcal{M}/\mathcal{I}_{p^n} \mathcal{M})+O(p^{(d-1)n}).
\]
Reversing the roles of $\mathcal{M}$ and $\mathcal{M}'$, we obtain the assertion.
\end{proof}

\subsection{Weakening the rank assumption}
In this subsection, we weaken the rank assumption of \Cref{Monsky21} (2) and obtain a key proposition, which will be used in the proof of our main results in Section 7.

\begin{prop}\label{CM6}
Let $F$ be a nonzero element of $\Lambda$ and let $\mathcal{M}=\Lambda/F$. Then there exists $\mu\in\Z_{\geq0}$ such that
\[
e(\mathcal{M}/(f_1(T_1),\ldots,f_j(T_j),T_{j+1},\ldots,T_d))=\mu p^{jn}+O(np^{(j-1)n}),
\]
where $f_i$ are either $\omega_{p^n}(T_i)$ or $\nu_{p^n}(T_i)$.
\end{prop}
\begin{proof}
If $F\in(T_{j+1},\ldots,T_d)$, then
\[
e(\mathcal{M}/(f_1(T_1),\ldots,f_j(T_j),T_{j+1},\ldots,T_d))=e(\Zp[\![T_{j+1},\ldots,T_{d}]\!]/(f_1(T_{j+1}),\ldots,f_d(T_{d}))=0.
\]
Hence we may assume $F\notin(T_{j+1},\ldots,T_d)$. Since $F\notin(T_{j+1},\ldots,T_d)$, we have $\overline{F}\neq 0$ in $\Lambda/(T_{j+1},\ldots,T_d)$. Therefore, $\mathcal{M}/(T_{j+1},\ldots,T_d)\mathcal{M}$ is a torsion $\Z[\![T_1,\ldots,T_j]\!]$-module, 
where $\Lambda/(T_{j+1},\ldots,T_d)\cong\Z[\![T_1,\ldots,T_j]\!]$. 
Thus, by \Cref{CM5}, we have 
\[e(\mathcal{M}/(f_1(T_1),\ldots,f_j(T_j),T_{j+1},\ldots,T_d))= \mu p^{jn}+O(np^{(j-1)n})\]
 for some $\mu\in\Z_{\geq 0}$.
\end{proof}

Let $n$ be a positive integer and let $\omega^i, \nu^i$ denote $\omega_{p^n}(T_i), \nu_{p^n}(T_i)$ respectively. Let $f_i$ be either $\omega_{p^n}(T_i)$ or $\nu_{p^n}(T_i)$.

\begin{lem}\label{injectivelem} The map 
\[
\varphi:\Lambda/(\omega^1,f_2,\ldots,f_d)\to\Lambda/(T_1,f_2,\ldots,f_d)\oplus \Lambda/(\nu^1,f_2,\ldots,f_d)
\]
is injective.
\end{lem}
Let $a\in\Lambda$ and suppose that $\varphi(a)$=0. We want to show that $a\in(\omega^1,f_2,\ldots,f_d)$. Since $\varphi(a)=0$, we can write
\[
a=a_1T_1+a_2f_2+\cdots+a_df_d=b_1\nu^1+b_2f_2+\cdots+b_df_d.
\]
Hence we have
\[
b_1p^n\equiv (a_2-b_2)f_2+\cdots+(a_d-b_d)f_d\mod T_1.
\]
We regard this as an equation in $\Lambda_{d-1}:=\Zp[\![T_2,\ldots,T_d]\!]$.
For any $p^n$-th roots of unity $\zeta_2,\ldots,\zeta_d$, we have $b_1(\zeta_2,\ldots,\zeta_d)\, p^n=0$, hence $b_1(\zeta_2,\ldots,\zeta_d)=0.$ 
Since $\Lambda_{d-1}/\mathcal{I}_{p^n}\to \bigoplus_\zeta \Zp[\zeta]$ is injective by the argument at the beginning of this section, we have $b_1(0,T_2,\ldots,T_d)\in(f_2,\ldots,f_d)$. This implies $b_1\in (T_1,f_2,\ldots,f_d)$ in $\Lambda$, and so $a\in(\omega^1,f_2,\ldots,f_d)$. This completes the proof.

\begin{lem}\label{avoidlem}
Let $F$ be a nonzero element of $\Lambda$. Consider the map
\[
\overline{\varphi}_n:\Lambda/(F,\omega^1,f_2,\ldots,f_d)\to\Lambda/(F,T_1,f_2,\ldots,f_d)\oplus\Lambda/(F,\nu^1,f_2,\ldots,f_d)
\]
Then we have $\ker\overline{\varphi}_n, \coker\overline{\varphi}_n$ are finite and $e(\ker\overline{\varphi}_n), e(\coker\overline{\varphi}_n)=O(p^{(d-1)n})$.
\end{lem}
\begin{proof}
Consider the commutative diagram with exact rows 
\[ 
\begin{small}
\xymatrix{
0\ar[r]&F(\Lambda/(\omega^1,f_2,\ldots,f_d))\ar[r]^{\hspace{0mm}\varphi'_n}\ar[d]^{\alpha}
& 
{\begin{array}{@{}c@{}}F(\Lambda/(T_1,f_2,\ldots,f_d))\ \ \\ \ \oplus\Lambda/(\nu^1,f_2,\ldots,f_d)\end{array}} 
\ar[r]\ar[d]^{\beta}
& F/\left(
{{\begin{array}{@{}c@{}}F\Lambda\cap(T_1,f_2,\ldots,f_d)\\ 
+ F\Lambda\cap(\nu^1,f_2,\ldots,f_d)\end{array}}}\right)
\ar[r]\ar[d]^{\gamma}&0\\
0\ar[r]&(\Lambda/(\omega^1,f_2,\ldots,f_d))\ar[r]^{\hspace{0mm}\varphi_n}
&{\begin{array}{@{}c@{}}\Lambda/(T_1,f_2,\ldots,f_d)\ \ \\ \oplus \Lambda/(\nu^1,f_2,\ldots,f_d)\end{array}}\ar[r]
&\Lambda/(T_1,\nu^1,f_2,\ldots,f_d)\ar[r]&0
}\end{small}
\] 
with 
$\coker\alpha=\Lambda/(F,\omega^1,f_2,\ldots,f_d)$, $\coker\beta=\Lambda/(F,T_1,f_2,\ldots,f_d)\oplus\Lambda/(F,\nu^1,f_2,\ldots,f_d)$. 
Note that we have a natural surjection
\[
A:=\frac{\Lambda}{(T_1,\nu^1,f_2,\ldots,f_d)}
\stackrel{\times F}{\surj}
\frac{F\Lambda}{(F\Lambda\cap(T_1,f_2,\ldots,f_d))+(F\Lambda\cap(\nu^1,f_2,\ldots,f_d))}=:B. 
\]
This implies $|B|\leq|A|$, and so
\[
|\ker\gamma|=\frac{|B||\coker\gamma|}{|A|}\leq\frac{|A||\coker\gamma|}{|A|}=|\coker\gamma|.
\]
By the snake lemma, we have $\ker\overline{\varphi}_n=\ker\gamma$ and $\coker\overline{\varphi}_n=\coker\gamma$. Since $e(\coker\gamma)=O(p^{(d-1)n})$ by \Cref{injectivelem}, we have the assertion.
\end{proof}

\begin{lem}\label{cokerlemma}
Let $R$ be an integral domain and let $\varphi:\mathcal{M}\to \mathcal{N}$ be a homomorphism of $R$-modules. Consider the commutative diagram with exact rows 
\[
\xymatrix{
0\ar[r]&{\rm tor}_R(\ker \varphi)\ar[r]\ar[d]&{\rm tor}_R\mathcal{M}\ar[r]^{\varphi'}\ar[d]&{\rm tor}_R\mathcal{N}\ar[r]\ar[d]&\coker \varphi'\ar[r]\ar[d]&0\\
0\ar[r]&\ker \varphi\ar[r]&\mathcal{M}\ar[r]^{\varphi}&\mathcal{N}\ar[r]&\coker\varphi\ar[r]&0,
}
\]
where $\varphi'$ is the restriction of $\varphi$. Suppose that $\ker\varphi={\rm tor}_R(\ker\varphi)$. Then $\coker\varphi'\to\coker\varphi$ is injective.
\end{lem}

\begin{proof}
We divide the given commutative diagrams into two commutative diagrams with exact rows 
\begin{gather*}
\xymatrix{
0\ar[r]&{\rm tor}_R(\ker\varphi)\ar[r]\ar[d]^{\alpha}&{\rm tor}_R \mathcal{M}\ar[r]\ar[d]^{\beta}&\im \varphi'\ar[r]\ar[d]^{\gamma}&0\\
0\ar[r]&\ker\varphi\ar[r]&\mathcal{M}\ar[r]&\im\varphi\ar[r]&0, 
}\\
\xymatrix{
0\ar[r]&\im \varphi'\ar[r]\ar[d]^{\gamma}&{\rm tor}_R \mathcal{N}\ar[r]\ar[d]^{\delta}&\coker\varphi'\ar[r]\ar[d]^{\varepsilon}&0\\
0\ar[r]&\im\varphi\ar[r]&\mathcal{N}\ar[r]&\coker\varphi\ar[r]&0.
}
\end{gather*} 
Applying the snake lemma to the first diagram, by the surjectivity of $\alpha$, we have $\coker\beta=\coker\gamma$, which implies that $\coker \gamma$ is a free $R$-module. Applying the snake lemma to the second diagram, by the injectivity of $\delta$, we may regard $\ker \varepsilon\subset\coker \gamma$. Since $\coker \varphi'$ is torsion over $R$ and $\coker \gamma$ is free, we must have $\ker \varepsilon=0$. Therefore, $\coker\varphi'\to\coker\varphi$ is injective.
\end{proof}

\begin{prop}\label{CM7}
Let $\mathcal{M}$ be a finitely generated $\Lambda$-module. Suppose that $r(\mathcal{M}/\mathcal{J}_{p^n} \mathcal{M})=O(p^{(d-2)n})$. Then we have
\[
|e(\mathcal{M}/\mathcal{I}_{p^n}\mathcal{M})-e(\mathcal{M}/\mathcal{J}_{p^n}\mathcal{M})|=O(np^{(d-1)n})
\]
\end{prop}
\begin{proof}
By \Cref{pigap}, We may assume that there exists a nonzero element $F\in\Lambda$ such that $\mathcal{M}=\Lambda/F\Lambda$. By \Cref{avoidlem} and \Cref{cokerlemma}, We have
\begin{align*}
e(\mathcal{M}/\mathcal{I}_{p^n}\mathcal{M})&
=e\big(\bigoplus_{(f_i)_i \in\prod_i \{T_i, \nu_{p^n}(T_i)\}} \mathcal{M}/(f_1(T_1),\ldots,f_d(T_d))\mathcal{M}\big)+O(np^{(d-1)n})\\
&=e\big(\mathcal{M}/\mathcal{J}_{p^n}\mathcal{M}\oplus\bigoplus_{\rm others} \mathcal{M}/(f_1(T_1),\ldots,f_d(T_d))\mathcal{M}\big)+O(np^{(d-1)n})\\
&= e(\mathcal{M}/\mathcal{J}_{p^n}\mathcal{M})+e\big(\bigoplus_{\rm others} \mathcal{M}/(f_1(T_1),\ldots,f_d(T_d))\mathcal{M}\big)+O(np^{(d-1)n}).
\end{align*}
By \Cref{CM6}, we have
\[
e\big(\bigoplus_{\rm others} \mathcal{M}/(f_1(T_1),\ldots, f_d(T_d))\mathcal{M}\big)=O(p^{(d-1)n}).
\]
This completes the proof. 
\end{proof}
\begin{prop}\label{22510}
Let $\mathcal{M}$ be a finitely generated torsion $\Lambda$-module. Suppose that $r(\mathcal{M}/\mathcal{J}_{p^n}\mathcal{M})=O(p^{(d-2)n})$.
Then there exists $\mu\in\Z_{\geq 0}$ such that
\[
e(\mathcal{M}/\mathcal{I}_{p^n}\mathcal{M})=\mu p^{dn}+O(np^{(d-1)n}).
\]
\end{prop}
\begin{proof}
By \Cref{Monsky21}, There exist $\mu\in\Z_{\geq 0}$ such that
\[
e(\mathcal{M}/\mathcal{J}_{p^n}\mathcal{M})=\mu p^{dn}+O(np^{(d-1)n}).
\]
By \Cref{CM7}, we obtain the assertion.
\end{proof}

\subsection{Sufficient conditions for the rank assumptions}

In this subsection, we prove a sufficient condition for the rank assumption in \Cref{22510}, which is written in words of characteristic elements.

For $\zeta=(\zeta_1,\cdots,\zeta_d)\in (W\setminus\{1\})^d$, define \glssymbol{zpzeta}$\Zp[\zeta]=\Zp[\zeta_1,\ldots,\zeta_d]$. Let $\mathcal{M}$ be a finitely generated torsion $\Lambda$-module. 
For each $\zeta\in (W\setminus\{1\})^d$, put \glssymbol{mzeta}$\mathcal{M}_{\zeta}=\mathcal{M}\otimes_{\Lambda}\Zp[\zeta]$. Then $\mathcal{M}_{\zeta}$ is a finitely generated $\Zp[\zeta]$-module via the ring homomorphism $\Lambda\ni F\mapsto F(\zeta-1)\in \Zp[\zeta]$. Let \glssymbol{rzeta}$r_{\zeta}(\mathcal{M}_\zeta):=\rank_{\Zp[\zeta]} \mathcal{M}_\zeta$. 

\begin{lem}[{\cite[Lemma 3.5 with $\mathcal{I}_{p^n}$ replaced by $\mathcal{J}_{p^n}$]{CuocoMonsky1981}}]\label{CM201}
We have
\[
r(\mathcal{M}/\mathcal{J}_{p^n}\mathcal{M})=\sum_{\zeta\in({W(n)\setminus\{1\}})^d}r_{\zeta}(\mathcal{M}_{\zeta}).
\]
\end{lem}

\begin{proof}
By \Cref{CM31}, we have
\[
\sum r(\mathcal{M}_{\zeta})=r(\mathcal{M}/\mathcal{J}_{p^n}\mathcal{M}),
\]
where the sum runs over a set of conjugacy class representatives.
Hence it suffices to show
\[
\sum_{\zeta\in(W\setminus\{1\})^d} r_{\zeta}(\mathcal{M}_{\zeta})=\sum r(\mathcal{M}_{\zeta}).
\]
Let $\zeta\in(W\setminus\{1\})^d$. Consider $\sum r_{\zeta'}(\mathcal{M}_{\zeta'})$, where $\zeta'$ runs over all conjugate elements of $\zeta$. Since the number of conjugate elements of $\zeta$ is $r(\Zp[\zeta])$, we have
\[
\sum r_{\zeta'}(\mathcal{M}_{\zeta'})=r(\Zp[\zeta])\cdot r_{\zeta}(\mathcal{M}_{\zeta}).
\]
By the definition of ranks, we have
\[
r(\Zp[\zeta])\cdot r_{\zeta}(\mathcal{M}_{\zeta})=r(\mathcal{M}_{\zeta}).
\]
Hence we conclude that
\[
\sum_{\zeta\in(W\setminus\{1\})^d} r_{\zeta}(\mathcal{M}_{\zeta})=\sum r(\mathcal{M}_{\zeta}).
\]
This completes the proof.
\end{proof}

Define
\begin{gather*}
\glssymbol{zm}Z(\mathcal{M}):=\{ \zeta\in (W\setminus\{1\})^d \mid r_{\zeta}(\mathcal{M}_{\zeta})\geq 1\}, \\
\glssymbol{znm}Z_n(\mathcal{M}):=Z(\mathcal{M})\cap (W(n)\setminus\{1\})^d.
\end{gather*} 

\begin{lem}[{\cite[Lemma 3.7 with $\mathcal{I}_{p^n}$ replaced by $\mathcal{J}_{p^n}$]{CuocoMonsky1981}}]\label{CM150}
There exists some $s\geq 0$ such that, for every $n\geq 0$, 
\[
|Z_n(\mathcal{M})|\leq r(\mathcal{M}/\mathcal{J}_{p^n} \mathcal{M})\leq s|Z_n(\mathcal{M})|
\]
holds. 
\end{lem}
\begin{proof}
Suppose $\mathcal{M}$ is generated by $s$ elements over $\Lambda$. Then, by the definition of $\mathcal{M}_{\zeta}$, $\mathcal{M}_{\zeta}$ is also generated by $s$ elements over $\Zp[\zeta]$. This implies $r_{\zeta}(\mathcal{M}_{\zeta})\leq s$. By \Cref{CM201}, we have
\[
r(\mathcal{M}/\mathcal{J}_{p^n}\mathcal{M})\leq s|Z_n(\mathcal{M})|.
\]
On the other hand, by the definition of $Z_n(\mathcal{M})$, we have
\[
|Z_n(\mathcal{M})|\leq r(\mathcal{M}/\mathcal{J}_{p^n}\mathcal{M}). \qedhere
\] 
\end{proof}
Put \glssymbol{mcaW}$\mca{W}=\{\xi-1\in \ol{\Q}_p\mid \xi\in W\}$.
\begin{prop}\label{2285}
Let $\mathcal{M}$ be a finitely generated torsion $\Lambda$-module and $F$ the characteristic element of $\mathcal{M}$. Suppose that $F(T_1,\ldots,T_d)$ does not vanish on $(\mca{W}\setminus\{0\})^d$. 
Then we have
\[
r(\mathcal{M}/\mathcal{J}_{p^n}\mathcal{M})=O(p^{(d-2)n}).
\]
\end{prop}
\begin{proof}
Since $\mathcal{M}$ is torsion over $\Lambda$, so is the fundamental $\Lambda$-module $\mathcal{E}$ corresponding to $\mathcal{M}$. By the definition of a characteristic element, we have $F(T_1,\ldots,T_d)\mathcal{E}=0$. By the definition of the action of $\Zp[\zeta]$ on $\mathcal{E}_{\zeta}$, we have $F(\zeta-1)\mathcal{E}=0$. By assumption, we have $F(\zeta-1)\neq 0$ for any $\zeta\in (W\setminus\{1\})^d$. Therefore, $\mathcal{E}_{\zeta}$ is a torsion $\Zp[\zeta]$-module. Hence $r_{\zeta}(\mathcal{E}_{\zeta})=0$. Therefore, we have $Z(\mathcal{E})=\emptyset$, and so is $Z_n(\mathcal{E})$. By \Cref{CM150}, we have
\[
r(\mathcal{E}/\mathcal{J}_{p^n}\mathcal{E})=0.
\]
Thus, by \Cref{CM4}, we obtain the assertion.
\end{proof}

\begin{remark}\label{Alexanderremark}
If there exists a polynomial $f(t_1,\ldots,t_d)\in\Z[t_1,\ldots,t_d]$ such that $f(1+T_1,\ldots,1+T_d)=F(T_1,\ldots,T_d)$, then the assumption can be replaced by ``$f(t_1, \ldots, t_d)$ does not vanish on $(W\setminus\{1\})^d$.''
\end{remark}

\begin{prop}\label{lemmaforbetti}
Let $\mathcal{M}$ be a finitely generated torsion $\Lambda$-module and $F$ the characteristic element of $\mathcal{M}$. Suppose that $F(T_1,\ldots,T_d)$ does not vanish on $\mca{W}^d \setminus\{(0,0,\ldots,0)\}$. Then we have
\[
r(\mathcal{M}/\mathcal{I}_{p^n}\mathcal{M})=O(p^{(d-2)n}).
\]
\end{prop}
\begin{proof} The proof is almost the same as that of \Cref{2285} except that $Z_n(\mathcal{E})$ can be $1$ if we use the corresponding original results for $\Lambda$-modules of Cuoco--Monsky \cite{CuocoMonsky1981}.
\end{proof}

\section{The fundamental two exact sequences} \label{sec.twoexact} 
In this section, we prove several group theoretic results that will play crucial roles in the proof of our main results. We establish \emph{the fundamental two exact sequences} and \emph{the estimation formulas} for quotients of $\Lambda_\Z$-modules and $\Lambda$-modules as well. We will utilize them several times in the subsequent sections. 

\subsection{Theorems for groups} 
Suppose that a surjective group homomorphism $G\stackrel{\mathcal{P}}{\rightarrow}\Z^d$ with abelian kernel $N$ is given. Let $\{t_1,\ldots,t_d\}$ be a $\Z$-basis of $\Z^d$ and let $x_1,\ldots, x_d$ be elements in the inverse images $\mathcal{P}^{-1}(t_1),\ldots,\mathcal{P}^{-1}(t_d)$ respectively. Define an action of $t_i$ on $y\in N$ by
\[
t_iy=y^{x_i}=x_iyx_i^{-1}.
\]
This does not depend on the choice of $x_i$ since $N$ is abelian and is a normal subgroup of $G$. In this way, $N$ becomes a module over $\Lambda_{\Z}:=\Z[t_1^{\pm 1},\ldots, t_d^{\pm 1}]$.

\begin{theorem}[The fundamental two exact sequences] \label{twoexact}
Let $\mathcal{Z}$ be a subgroups of $\Z^d$ that is isomorphic to $\Z^{d'}$ with $d'\leq d$ and let $H:=\mathcal{P}^{-1}(\mathcal{Z})$, so that there is a commutative diagram 
\[
\xymatrix{
1\ar[r]&N\ar[r]\ar@{=}[d]&G\ar[r]^{\mathcal{P}}\ar@{}[d]|{\bigcup}&\Z^d\ar[r]\ar@{}[d]|{\bigcup}&1\\
1\ar[r]&N\ar[r]&H\ar[r]&\mathcal{Z}\ar[r]&1
}
\]
of exact sequences. 
Let $\{t'_1,\ldots,t'_{d'}\}\subset \Lambda_\Z$ be a $\Z$-basis of $\mathcal{Z}$ and let $x'_1,\ldots,x'_{d'}$ be elements in the inverse images $\mathcal{P}^{-1}(t'_1),\ldots, \mathcal{P}^{-1}(t'_{d'})$. 
Let $I$ denote the ideal of $\Lambda_{\Z}$ generated by $\{t'_1-1,\ldots,t'_{d'}-1\}$. 
Let $C$ denote the $\Lambda_{\Z}$-submodule of $N$ generated by $\{[x'_i,x'_j]\mid 1\leq i< j\leq d'\}$. 
Then we have two exact sequences of $\Lambda_{\Z}$-modules 
\glssymbol{groupabelianization}\begin{gather*}
0\to N/(C+IN)\to H^{\rm ab}\to \Z^{d'}\to 0,\\
0\to (C+IN)/IN\to N/IN\to N/(C+IN)\to 0.
\end{gather*}
\end{theorem}

Since the first exact sequence splits, we have
\[
{\rm tor}_{\Z}(H^{\rm ab})\cong{\rm tor}_{\Z}(N/(C+IN))
\]
as $\Z$-modules.

\begin{proof} 
By the commutative diagram, we have 
\[
H=\langle x'_1,\ldots,x'_{d'},N\rangle.
\]
Since
\[
(t'_i-1)y=y^{x'_i}y^{-1}=x'_iy{x'}_i^{-1}y^{-1}\mbox{ for each }y\in N,
\]
we obtain
\begin{align*}
H^{\rm ab}&=\langle x'_1,\ldots,x'_{d'}, N\rangle^{\rm ab}\\
&=\langle x'_1,\ldots,x'_{d'}, N\rangle/\langle[x'_i,x'_j],IN(1\leq i<j\leq d')\rangle.
\end{align*}
Let $[H,H]$ denote the commutator subgroup of $H$. Since $H$ is a normal subgroup of $G$, so is $[H,H]=\langle[x'_i,x'_j],IN(1\leq i<j\leq d')\rangle$. Also since $[H,H]\subset\ker\mathcal{P}$ and $N$ is abelian, $[H,H]$ is also an abelian group. This implies that $[H,H]$ is a $\Lambda_{\Z}$-module, and so $[H,H]=C+IN$. Therefore,
\[
0\to N/(C+IN)\to H^{\rm ab}\to \Z^{d'}\to 0
\]
is an exact sequence of $\Lambda_{\Z}$-modules. Since $C+IN$ is a $\Lambda_{\Z}$-module,
\[
0\to (C+IN)/IN\to N/IN\to N/(C+IN)\to 0.
\]
is also an exact sequence of $\Lambda_{\Z}$-modules.
\end{proof}

Consider the case where $\mathcal{Z}=n\Z^d$. Then $\{t_1^n,\ldots, t_d^n\}$ is a $\Z$-basis of $\mathcal{Z}$. For each $n\geq 1$, let 
\glssymbol{in}$C_n, H_n, I_n$ denote the corresponding $C, H, I$ respectively.

\begin{prop}\label{actionprop} 
The action of $t_i$ on $(C_n+I_n N)/I_nN$ is trivial.
\end{prop}
\begin{proof}
Put $c_{ij}:=[x_i,x_j]\in N$, $c_{ij}^{(n)}:=[x_i^n, x_j^n]$, and $\mathcal{T}_n(t):=\sum_{k=0}^{n-1}t^k$. Then we have 
\begin{eqnarray}
c_{ij}^{(n)}&=&x_i(x_i^{n-1}x_j^nx_i^{-(n-1)}x_j^{-n})x_i^{-1}(x_ix_j^nx_i^{-1}x_j^{-n})\nonumber\\
&=&[x_i^{n-1}, x_j^n]^{x_i}[x_i,x_j^n]\nonumber\\
&=&([x_i^{n-2}, x_j^n]^{x_i}[x_i,x_j^n])^{x_i}[x_i,x_j^n]\nonumber\\
&=&\ldots\nonumber\\
&=&(t_i^{n-1}+\cdots+t_i+1)[x_i,x_j^n]\nonumber\\
&=&(t_i^{n-1}+\cdots+t_i+1)[x_i,x_j][x_i,x_j^{n-1}]^{x_j}\nonumber\\
&=&\ldots\nonumber\\
&=&(t_i^{n-1}+\cdots+t_i+1)(t_j^{n-1}+\cdots+t_j+1)[x_i,x_j]\nonumber\\
&=&\mathcal{T}_n(t_i)\mathcal{T}_n(t_j)c_{ij}. \label{ceqn}
\end{eqnarray}
Also, we see that $c_{ii}^{(n)}=0$ and $c_{ji}^{(n)}=-c_{ij}^{(n)}$.

Now, we shall show that the action of $t_i$ on $(C_n+I_nN)/I_nN$ is trivial. For $i,j$, we have
\begin{gather*}
(t_i-1)c_{ij}^{(n)}=(t_i^n-1)\mathcal{T}_n(t_j)c_{ij}\in I_n N,\\
(t_j-1)c_{ij}^{(n)}=(t_j^n-1)\mathcal{T}_n(t_i)c_{ij}\in I_nN.
\end{gather*}
On the other hand, for $k\neq i,j$, by the Hall--Witt identity 
\[ [x,[y^{-1},z]]^y\cdot[y,[z^{-1},x]]^z\cdot[z,[x^{-1},y]]^x=1,\] we have
\[
t_j(t_i^{-1}-1)[x_j^{-1},x_k]+t_k(t_j-1)[x_k^{-1},x_i^{-1}]+t_i^{-1}(t_k-1)c_{ij}=0.
\]
By a further calculation, we obtain
\begin{equation} \label{Hall--Witt} 
(t_i-1)c_{jk}+(t_j-1)c_{ki}+(t_k-1)c_{ij}=0.
\end{equation}
By \eqref{ceqn}, we conclude that
\[
(t_k-1)c_{ij}^{(n)}=\mathcal{T}_n(t_i)\mathcal{T}_n(t_j)(t_k-1) c_{ij}= \mathcal{T}_n(t_i)\mathcal{T}_n(t_j)(-(t_i-1)c_{jk}-(t_j-1)c_{ki})\in I_n N.
\]
Therefore, $t_i$ acts trivially on $(C_n+I_nN)/I_nN$. By the definition of $C_n$, we see that $(C_n+I_nN)/I_nN$ is a $\Z$-module generated by at most $\frac{d(d-1)}{2}$ elements.
\end{proof}

\begin{theorem}[The estimation formulas]\label{asymptotic} 
Under the setting as above, if $N$ is finitely generated over $\Lambda_{\Z}$, then we have that 
\begin{itemize}
\item[{\rm i)}] \ $|{\rm tor}_{\Z}(N/I_m N)|\mbox{ divides }A_1m^{\frac{d(d-1)(d-2)}{2}}|{\rm tor}_{\Z}(H_m^{\rm ab})|$,
\item[{\rm ii)}] \ $|{\rm tor}_{\Z}(H_m^{\rm ab})|\mbox{ divides } A_2 m^{d(d-1)}|\tor_{\Z}(N/I_m N)|$,
\end{itemize}
where $A_1$, $A_2$ are constant terms given by 
$A_1:=|{\rm tor}_{\Z}((C_1+I_1N)/I_1N)|$ and $A_2:= |{\rm tor}_{\Z}(H_1^{\rm ab})|$.
In particular, we have
\[
e(H_{p^n}^{\rm ab})=e(N/I_{p^n}N)+O(n)\ \ (n\to\infty).
\]
\end{theorem}

\begin{proof}
Let $P_n$ be the inverse image of ${\rm tor}_{\Z}(N/(C_n+I_nN))$ under $N/I_nN\to N/(C_n+I_nN)$. Then, we have ${\rm tor}_{\Z}(P_n)={\rm tor}_{\Z}(N/I_n N)$, and
\[
0\to(C_n+I_nN)/I_nN\to P_n\to{\rm tor}_{\Z}(N/(C_n+I_nN))\to 0
\]
is an exact sequence of $\Lambda_{\Z}$-modules. For an arbitrary pair of positive integers $n\mid m$, we have a commutative diagram of 
$\Lambda_{\Z}$-modules with exact rows 
\begin{equation}\label{kakanzushiki}
\xymatrix{
0\ar[r]&(C_{m}+I_{m}N)/I_{m}N \ar[r]\ar[d]^{\varphi'_{m,n}} & P_{m} \ar[r]\ar[d]^{\varphi_{m,n}} & {\rm tor}_{\Z}(N/(C_{m}+I_{m}N))\ar[r]\ar[d]^{\varphi''_{m,n}} & 0 \\
0\ar[r]&(C_n+I_nN)/I_nN\ar[r] &P_n\ar[r] & {\rm tor}_{\Z}(N/(C_n+I_nN)) \ar[r] & 0,
}
\end{equation}
where $\varphi'_{m,n}, \varphi_{m,n},\varphi''_{m,n}$ are the maps induced by ${\rm id}_N$. 
Define a $\Lambda_{\Z}$-homomorphism $\psi_{n,m}:(C_n+I_nN)/I_nN\to (C_m+I_mN)/I_mN$ by the action of %$\psi_{n,m}=$
$\mathcal{T}_{m/n}(t_1^n)\cdots \mathcal{T}_{m/n}(t_d^n)$. This map is well-defined, since 
\begin{align*}
\mathcal{T}_{m/n}(t_i^n)\mathcal{T}_n(t_i)
&=((t_i^n)^{m/n-1}+\cdots+t_i^n+1)(t_i^{n-1}+\cdots +t_i+1)\\
&=t_i^{m-1}+\cdots+t_i+1\\
&=\mathcal{T}_m(t_i)
\end{align*}
implies that $t_i^n-1$ maps to $t_i^m-1$ via $\times \mathcal{T}_{m/n}(t_i^n)$.
Since $t_i$ acts trivially on $(C_n+I_nN)/I_nN$, the map multiplying by $\prod_{1\leq i\leq d}\big((t_i^n)^{m/n-1}+ (t_i^n)^{m/n-2}+\cdots+t_i^n+1)\big)$ becomes the map multiplying by $\prod_{1\leq i\leq d}(1+1+\cdots+1+1)$ via $\varphi'_{m,n}\circ\psi_{n,m}$ and $\psi_{n,m}\circ\varphi'_{m,n}$. 
This implies $\varphi'_{m,n}\circ \psi_{n,m}=(\frac{m}{n})^d$ and $\psi_{n,m}\circ \varphi'_{m,n}=\left(\frac{m}{n}\right)^d$.
Since we have
\begin{align*}
c_{ij}^{(m)}&=\mathcal{T}_m(t_i)\mathcal{T}_m(t_j)c_{ij}\\
&=\mathcal{T}_{m/n}(t_i)\mathcal{T}_{m/n}(t_j)\mathcal{T}_n(t_i)\mathcal{T}_n(t_j)c_{ij}\\
&=\mathcal{T}_{m/n}(t_i)\mathcal{T}_{m/n}(t_j)c_{ij}^{(n)}
\end{align*}
and $t_i$ acts trivially on $(C_n+I_nN)/I_nN$, one finds that
\[
\im \varphi'_{m,n}=\left(\frac{m}{n}\right)^2((C_n+I_nN)/I_n N).
\]
Likewise, since
\[
\mathcal{T}_{m/n}(t_1)\cdots \mathcal{T}_{m/n}(t_d) c_{d-1,d}^{(n)}
=\mathcal{T}_{m/n}(t_1)\cdots \mathcal{T}_{m/n}(t_{d-2})c_{d-1,d}^{(m)},
\]
we have
\[
\im\psi_{n,m}=\left(\frac{m}{n}\right)^{d-2}((C_{m}+I_{m}N)/I_{m}N).
\]
Since there are $\Z$-homomorphisms onto subgroups of finite indices each other, the finitely generated $\Z$-modules $(C_n+I_nN)/I_nN$ and $(C_{m}+I_{m}N)/I_{m}N$ have the same free $\Z$-rank, and so $\ker\varphi'_{m,n}$ and $\ker\psi_{n,m}$ are finite.
We shall show that
\[
\psi_{n,m}({\rm tor}_{\Z}((C_n+I_nN)/I_nN))=\left(\frac{m}{n}\right)^{d-2}{\rm tor}_{\Z}((C_{m}+I_{m}N)/I_{m}N).
\]
Indeed, we have
\begin{align*}
\im\psi_{n,m}&=\left(\frac{m}{n}\right)^{d-2}((C_{m}+I_{m}N)/I_{m}N)\\
&=\left(\frac{m}{n}\right)^{d-2}(\Z^r\oplus ({\rm tor}_{\Z}((C_{m}+I_{m}N)/I_{m}N))\\
&=\left(\frac{m}{n}\right)^{d-2}\Z^r\oplus \left(\frac{m}{n}\right)^{d-2}{\rm tor}_{\Z}((C_{m}+I_{m}N)/I_{m}N).
\end{align*}
for some $r\geq 0$. Now, $\psi_{n,m}$ induces a map
\[
\psi_{n,m}^{-1}(\left(\frac{m}{n}\right)^{d-2}{\rm tor}_{\Z}((C_{m}+I_{m}N)/I_{m}N))\to \left(\frac{m}{n}\right)^{d-2}{\rm tor}_{\Z}((C_{m}+I_{m}N)/I_{m}N)),
\]
and we have
\[
{\rm tor}_{\Z}((C_n+I_nN)/I_nN)\subset\psi_{n,m}^{-1}(\left(\frac{m}{n}\right)^{d-2}{\rm tor}_{\Z}((C_{m}+I_{m}N)/I_{m}N)).
\]
Since the kernel of this map is finite, we have that both the image and the kernel of this map are $\Z$-torsion. Therefore, we must have
\[
{\rm tor}_{\Z}((C_n+I_n N)/I_n N)=\psi_{n,m}^{-1}(\left(\frac{m}{n}\right)^{d-2}{\rm tor}_{\Z}((C_{m}+I_{m}N)/I_{m}N)),
\]
i.e.,
\[
\psi_{n,m}({\rm tor}_{\Z}((C_n+I_n N)/I_n N))=\left(\frac{m}{n}\right)^{d-2}{\rm tor}_{\Z}((C_{m}+I_{m}N)/I_{m}N).
\]

Since
\[
0\to{\rm tor}_{\Z}((C_n+I_nN)/I_nN)\to{\rm tor}_{\Z}(P_n)\to{\rm tor}_{\Z}(N/(C_n+I_nN))
\]
is an exact sequence of $\Lambda_{\Z}$-modules, by \Cref{twoexact}, we have
\[
\frac{|{\rm tor}_{\Z}(N/I_n N)|}{|{\rm tor}_{\Z}((C_n+I_nN)/I_nN)|}\mbox{ divides }|\tor_{\Z}(H_n^{\rm ab})|. 
\]
Since
\[
\psi_{n,m}({\rm tor}_{\Z}((C_n+I_n N)/I_n N))=\left(\frac{m}{n}\right)^{d-2}{\rm tor}_{\Z}((C_{m}+I_{m}N)/I_{m}N),
\]
we have
\begin{align*}
&|{\rm tor}_{\Z}(( C_n+I_nN)/I_nN)|\\
&=|\ker \psi_{n,m}|\cdot |\left(\frac{m}{n}\right)^{d-2}{\rm tor}_{\Z}((C_{m}+I_{m}N)/I_{m}N)|\\
&={|\ker \psi_{n,m}|\cdot} \frac{|{\rm tor}_{\Z}((C_{m}+I_{m}N)/I_{m}N)|}{|{\rm tor}_{\Z}((C_{m}+I_{m}N)/I_{m}N)/\left(\frac{m}{n}\right)^{d-2}{\rm tor}_{\Z}((C_{m}+I_{m}N)/I_{m}N)|}.
\end{align*}
Since $(C_{m}+I_{m}N)/I_{m}N$ is generated by at most $\frac{d(d-1)}{2}$ elements over $\Z$, we must have
\[
|{\rm tor}_{\Z}((C_{m}+I_{m}N)/I_{m}N)/\left(\frac{m}{n}\right)^{d-2}{\rm tor}_{\Z}((C_{m}+I_{m}N)/I_{m}N)|\mbox{ divides }\left(\frac{m}{n}\right)^{\frac{d(d-1)(d-2)}{2}}.
\]
Therefore, 
\[
|{\rm tor}_{\Z}((C_{m}+I_{m}N)/I_{m}N)|\mbox{ divides }\left(\frac{m}{n}\right)^{\frac{d(d-1)(d-2)}{2}}|{\rm tor}_{\Z}((C_n+I_nN)/I_nN)|.
\]
By putting $A_1:=|{\rm tor}_{\Z}((C_1+I_1N)/I_1N)|$, we obtain that
\[
|{\rm tor}_{\Z}((C_m+I_mN)/I_m N)|\mbox{ divides } A_1 m^{\frac{d(d-1)(d-2)}{2}},
\]
and so
\[
|{\rm tor}_{\Z} (N/I_m N)|\mbox{ divides }A_1 m^{\frac{d(d-1)(d-2)}{2}}|{\rm tor}_{\Z}(H_m^{\rm ab})|.
\]

On the other hand, applying the snake lemma to the commutative diagram \eqref{kakanzushiki}, we have an exact sequence of $\Lambda_{\Z}$-modules
\[
0\to\ker\varphi'_{m,n}\to\ker\varphi_{m,n}\to\ker\varphi''_{m,n}\to\coker\varphi'_{m,n}\to\coker\varphi_{m,n}\to\coker\varphi''_{m,n}\to 0.
\]
Since $\ker\varphi'_{m,n},\ker\varphi''_{m,n}, \coker\varphi'_{m,n}, \coker\varphi''_{m,n}$ are finite, so are $\ker\varphi_{m,n}, \coker\varphi_{m,n}$.
Hence we have
\[
\frac{|\ker\varphi''_{m,n}|}{|\coker\varphi''_{m,n}|}=\frac{|\ker\varphi_{m,n}|}{|\coker\varphi_{m,n}|}\frac{|\coker\varphi'_{m,n}|}{|\ker\varphi'_{m,n}|}.
\]
By the homomorphism theorem, we have
\[
\frac{|\ker\varphi''_{m,n}|}{|\coker\varphi''_{m,n}|}=\frac{|{\rm tor}_{\Z}(H_m^{\rm ab})|}{|{\rm tor}_{\Z}(H_n^{\rm ab})|}.
\]
Since $\ker\varphi_{m,n}$ is finite, we have
\[
\varphi_{m,n}^{-1}({\rm tor}_{\Z}(P_n))={\rm tor}_{\Z}(P_{m}).
\]
Let $\varphi''':{\rm tor}_{\Z}(P_{m})\to{\rm tor}_{\Z}(P_n)$ be the restriction map of $\varphi_{m,n}$. Then we have $\ker\varphi'''=\ker\varphi_{m,n}$. By \Cref{cokerlemma}, we also have $\coker\varphi'''\subset\coker\varphi_{m,n}$. Therefore, we have
\[
\frac{|\ker\varphi'''|}{|\coker\varphi'''|}=\frac{|{\rm tor}_{\Z}(N/I_m N)|}{|{\rm tor}_{\Z}(N/I_n N)|}.
\]
Accordingly, we obtain
\[
\frac{|{\rm tor}_{\Z}(H_m^{\rm ab})|}{|{\rm tor}_{\Z}(H_n^{\rm ab})|}=\frac{|{\rm tor}_{\Z}(N/I_{m}N)|}{|{\rm tor}_{\Z}(N/I_nN)|}\frac{|\coker\varphi'_{m,n}|}{|\coker\varphi_{m,n}/\coker\varphi'''|\cdot|\ker\varphi'_{m,n}|}.
\]
Put $A_2:=|{\rm tor}_{\Z}(\mathcal{P}^{-1}(H_1)^{\rm ab})|$. Since
\[
|\coker\varphi'_{m,n}|=\Big|(C_n+I_nN)/\big(\left(\frac{m}{n}\right)^2(C_n+I_nN)\big)\Big|\mbox{ divides } \left(\frac{m}{n}\right)^{d(d-1)}=\left(\frac{m}{n}\right)^{2\frac{d(d-1)}{2}},
\]
 we obtain that
\[
|{\rm tor}_{\Z}(H_m^{\rm ab})|\mbox{ divides }A_2m^{d(d-1)}|\tor_{\Z}(N/I_m N)|.
\]
This completes the proof.
\end{proof}

\subsection{Theorems for topological groups}
We may replace 
\begin{itemize}
\item
$\Lambda_{\Z}$ by $\Lambda$,
\item
${\Z^d}$ by $\Zp^{\,d}$,
\item
the term ``generated" for groups by ``topologically generated"
\end{itemize}
in the statement and the proof of \Cref{twoexact}: 

\begin{theorem}[The fundamental two exact sequences]\label{twoexact2}
Suppose that a homomorphism $G\stackrel{\mathcal{P}}{\rightarrow}\Zp^{\,d}$ of topological groups with abelian kernel $N$ is given. 
Let $\mathcal{Z}$ be a subgroups of $\Zp^{\,d}$ that is isomorphic to $\Zp^{\,d'}$ with $d'\leq d$ and let $H:=\mathcal{P}^{-1}(\mathcal{Z})$, so that there is a commutative diagram 
\[
\xymatrix{
1\ar[r]&N\ar[r]\ar@{=}[d]&G\ar[r]^{\mathcal{P}}\ar@{}[d]|{\bigcup}&\Z_{p}^{\,d}\ar[r]\ar@{}[d]|{\bigcup}&1\\
1\ar[r]&N\ar[r]&H\ar[r]&\mathcal{Z}\ar[r]&1
}
\]
of exact sequences. 
Let $\{t_1,\ldots,t_d\}$ be a $\Zp$-basis of $\Zp^{\,d}$ and let $x_1,\ldots, x_d$ be elements in the inverse images $\mathcal{P}^{-1}(t_1),\ldots,\mathcal{P}^{-1}(t_d)$ respectively. 
Let $\{t'_1,\ldots,t'_d\}\in\Lambda$ be a $\Zp$-basis of $\Zp^{\,d}$ and let $x'_1,\ldots,x'_{d'}$ be elements in the inverse images $\mathcal{P}^{-1}(t'_1),\ldots, \mathcal{P}^{-1}(t'_{d'})$. 

Let $I$ denote ideal of $\Lambda$ generated by $\{t'_1-1,\ldots,t'_{d'}-1\}$. Let $C$ denote the $\Lambda$-submodule of $N$ generated by $\{[x'_i,x'_j]\mid 1\leq i< j\leq d\}$. Then
we have two exact sequences of $\Lambda$-modules
\begin{gather*}
0\to N/(C+IN)\to H^{\rm ab}\to \Zp^{d'}\to 0,\\
0\to (C+IN)/IN\to N/IN\to N/(C+IN)\to 0.
\end{gather*} 

In particular, we have an isomorphism of $\Zp$-modules 
\[ {\rm tor}_{\Zp}(H^{\rm ab})\cong{\rm tor}_{\Zp}(N/(C+IN)).\]
\end{theorem} 

We have a version of \Cref{asymptotic} for $\Lambda$ as well. 
Consider the case where $\mathcal{Z}=p^n\Zp^d$. Then $\{t_1^{p^n}, \ldots, t_d^{p^n}\}$ is a $\Zp$-basis of $\mathcal{Z}$. For each $n\geq 1$, let $C_{p^n}, H_{p^n}, \mathcal{I}_{p^n}$ denote the corresponding $C, H, I$ respectively.

\begin{theorem}[The estimation formula] \label{asymptotic2}
Under the setting as above, if $N$ is finitely generated over $\Lambda_{\Z}$, then we have
\[
e(H_{p^n}^{\rm ab})=e(N/\mathcal{I}_{p^n}N)+O(n)\ \ (n\to\infty).
\]
\end{theorem}

\section{Alexander polynomials} 
In this section, we define the Alexander polynomials of $\Z^d$-covers and the characteristic elements of $\Zp^{\,d}$-covers, and point out their relationship. We also prove a result on the reduced Alexander polynomials and establish its $p$-adic variant for the characteristic elements for the Iwasawa modules of general $\Zp^{\,d}$-covers defined by $\mca{H}:=\varprojlim H_1(X_{p^n};\Zp)$. 

\subsection{Alexander polynomials} 

\begin{definition} \label{def.Alex} 
(1) Let $d\in \Z_{>0}$. Let $X$ be a compact connected orientable 3-manifold with a surjective homomorphism $\tau:\pi_1(X)\surj \Z^d$ and let $X_\infty\to X$ the corresponding $\Z^d$-cover. Let $t_1,\ldots, t_d$ be a basis of $\Z^d$. 
Then $H_1(X_{\infty})$ is a finitely generated $\Lambda_{\Z,d}$-module and its characteristic element is called \emph{the Alexander polynomial of $X_\infty\to X$} and is denoted by \glssymbol{alexandertau}$\Delta_\tau(t_1,\ldots,t_d)$. 

(2) If $L=\sqcup_i K_i$ is a $c$-component link in a $\Q$HS$^3$ and $X=M\setminus {\rm Int}\,V_L$, then we consider the maximal free abelian cover $X_\infty\to X$ corresponding to the free abelianization \glssymbol{fab}${\rm fab}:\pi_1(X)\surj \Z^c$ with a basis $t_1,\ldots, t_c$ of $\Z^c$ corresponding to $K_i$'s, and define \emph{the Alexander polynomial of} $L$ by \glssymbol{alexander}$\Delta_L(t_1,\ldots,t_c):=\Delta_{\rm fab}(t_1,\ldots,t_c)$ in $\Lambda_{\Z,c}$. 
\end{definition} 

For more details of the theory of Alexander polynomials, we refer to \cite{Massuyeau2008AP} and \cite[Chapters 3\&4] 
{Hillman2}. 

\begin{prop} \label{prop.hatH1Xinfty} 
Let the setting be as in \Cref{def.Alex} {\rm (1)}. 
Then the characteristic element of the $\Lambda_d$-module $\wh{H}_1(X_\infty)=H_1(X_\infty)\otimes_{\Lambda_\Z}\Lambda=
\varprojlim_n H_1(X_\infty)/I_{p^n}H_1(X_\infty)$ is given by 
\[F_\tau(T_1,\ldots, T_d)=\Delta_\tau(1+T_1,\ldots, 1+T_d).\] 
\end{prop}

\begin{proof} The assertion follows from Propositions \ref{1124} and \ref{prop.hathat}.  
\end{proof}

\subsection{Reduced Alexander polynomials} 
By using \Cref{twoexact}, we derive the following results. 
At least the assertion (1) for the cases $d=1,c$ is a well-known fact and may be proved in various ways. 
We will use this result in the proof of \Cref{thm.integral.Delta}  
to compare the $\mu, \lambda$'s of branched $\Zp^{\,d}$-covers over $\Z$HS$^3$ and those of the reduced Alexander polynomials. 

\begin{theorem} 
\label{reduced}
{\rm (1)} Suppose $d\geq 2$. 
Let $X$ be the exterior of a $c$-component link $L=\sqcup_i K_i$ 
in a $\Z$HS$\,^3$ and let $\alpha_i\in \pi_1(X)$ be a meridian of $K_i$ for each $i=1,\ldots, c$. 

{\rm (i)} 
Let $X_{\infty}\to X$ be the $\Z^d$-cover corresponding to a surjective homomorphism $\tau:\pi_1(X)\surj\Z^d$ and write $\tau(\alpha_i)=(v_{i1},\ldots,v_{id})\in\Z^d$. Then we have 
\[
\Delta_{\tau}(t_1,\ldots,t_d) \doteq \Delta_L(t_1^{v_{11}}\cdots t_d^{v_{1d}},\ldots,t_1^{v_{c1}}\cdots t_d^{v_{cd}}) \ \ \text{in}\ \  \Lambda_{\Z,d}. 
\]

{\rm (ii)} Let $(X_{p^n}\to X)_n$ be the $\Zp^{\,d}$-cover corresponding to a group homomorphism $\tau:\pi_1(X)\to\Zp^d$ that induces a surjective homomorphism $\wh{\tau}:\wh{\pi}_1(X)\surj\Zp^d$, 
and write $\tau(\alpha_i)=(v_{i1},\ldots,v_{id})\in\Zp^d$. Then we have
\[
F_{\tau}(T_1,\ldots,T_d) \doteq \Delta_L((1+T_1)^{v_{11}}\cdots (1+T_d)^{v_{1d}},\ldots,(1+T_1)^{v_{c1}}\cdots (1+T_d)^{v_{cd}}) \ \ \text{in}\ \ \Lambda_d,
\]
where $F_{\tau}$ denotes the characteristic element of the $\Lambda_d$-module $\mathcal{H}=\varprojlim H_1(X_{p^n}, \Zp)$. 

{\rm (2)} The space $X$ in above may be replaced by any compact orientable 3-manifold with tori boundary with a chosen basis $t_1,\ldots, t_c$ of $H_1(X)_{\rm free}\cong \Z^c$, where \glssymbol{c}$c$ is the number of the components of $\partial X$. 
\end{theorem} 

\begin{proof} (1) (i) 
Recall that $\{t_1,\ldots, t_d\}$ is a $\Z$-basis of the the target $\Z^d$ of $\tau$. 
Let ${\rm ab}:\pi_1(X)\to H_1(X)\cong \Z^c$ be the abelianization map and consider the induced map $\ul{\tau}:\Z^c\surj \Z^d$. 
Let $\{s_1,\ldots, s_c\}$ be a $\Z$-basis of $\Z^c$ defined by the meridians of $L$ and  
let $\wt{t}_i$ be an element of $\ul{\tau}^{-1}(t_i)$ for each $1\leq i\leq d$.
Then, by the theory of invariant factors over $\Z$, we may find a basis of the form $\{\wt{t}_1,\ldots,\wt{t}_c\}$ of $\Z^c$ so that $\{\wt{t}_{d+1},\ldots, \wt{t}_c\}$ becomes a basis of $\Z^{c-d}={\rm Ker}(\ul{\tau})$. 
Note that $\Lambda_{\Z,c}=\Z[\Z^c]=\Z[s_1^{\pm 1},\ldots,s_c^{\pm 1}]=\Z[\wt{t}_1^{\pm 1},\ldots,\wt{t}_c^{\pm 1}]$ 
and $\Lambda_{\Z,d}=\Z[\Z^d]=\Z[t_1^{\pm 1},\ldots, t_d^{\pm 1}]=\Lambda_{\Z,c}/I$. 

Let $X_{\rm ab}$ denote the maximal abelian cover of $X$, i.e., the covering space of $X$ corresponding to ${\rm ab}$. Consider the commutative diagram with exact rows  
\[
\xymatrix{
0\ar[r]&H_1(X_{\rm ab})\ar[r]\ar@{=}[d]&\pi_1(X)/[\pi_1(X_{\rm ab}),\pi_1(X_{\rm ab})]\ar[r]^{\hspace{20mm}\rm ab}\ar@{}[d]|{\bigcup}&\Z^c\ar[r]\ar@{}[d]|{\bigcup}&0\\
0\ar[r]&H_1(X_{\rm ab})\ar[r]&\pi_1(X_{\infty})/[\pi_1(X_{\rm ab}),\pi_1(X_{\rm ab})]\ar[r]&\ker(\ul{\tau}:\Z^c\surj\Z^d)\ar[r]&0.
}
\]
By \Cref{twoexact} with
\begin{itemize}
\item
$G:=\pi_1(X)/[\pi_1(X_{\rm ab}),\pi_1(X_{\rm ab})]$,
\item
$\mathcal{P}:G\to \pi_1(X)/\pi_1(X_{\rm ab})\cong \Z^c$, 
so $N=\pi_1(X_{\rm ab})^{\rm ab}=H_1(X_{\ab})$,
\item $\mathcal{Z}:=\ker(\ul{\tau}:\Z^c\to\Z^d)$ ($\cong\Z^{c-d}$), so 
$H^{\rm ab}=({\rm Ker}\tau)^{\rm ab} =H_1(X_{\infty})$,
\item $\{t_1',\ldots, t'_{d'}\}=\{\wt{t}_{d+1},\ldots,\wt{t}_c\}$, 
\end{itemize}
we have two exact sequences of $\Lambda_{\Z,d}$-modules
\begin{gather*} 
0\to H_1(X_{\rm ab})/(C+IH_1(X_{\rm ab}))\to H_1(X_{\infty})\to \Z^{c-d}\to 0,\\ 
0\to (C+IH_1(X_{\rm ab}))/IH_1(X_{\rm ab})\to H_1(X_{\rm ab})/IH_1(X_{\rm ab})\to H_1(X_{\rm ab})/(C+IH_1(X_{\rm ab}))\to 0, 
\end{gather*} 
where $I=(\wt{t}_{d+1}-1,\ldots, \wt{t}_c-1)$ is an ideal of $\Lambda_{\Z,c}$ and $C$ is the corresponding $\Lambda_{\Z,c}$-submodule of $H_1(X_{\rm ab})$ defined in \Cref{twoexact}.
By using equation \eqref{Hall--Witt} derived from the Hall--Witt identity, 
for each $1\leq i\leq d$ and $d+1\leq j,k\leq c$, we have $(\wt{t}_i-1)c_{jk}=-(\wt{t}_j-1)c_{ki}-(\wt{t}_k-1)c_{ij} \in IH_1(X_{\rm ab})$, so 
the actions of $\wt{t}_{1},\ldots, \wt{t}_d$ on $(C+IH_1(X_{\rm ab}))/IH_1(X_{\rm ab})$ are trivial. 
Hence $(C+IH_1(X_{\rm ab}))/IH_1(X_{\rm ab})$ is a $\Z$-module generated by at most $\frac{(c-d)(c-d-1)}{2}$ elements. 
Since $d\geq 2$, we obtain a pseudo-isomorphism of $\Lambda_{\Z,d}$-modules
\[H_1(X_{\rm ab})/IH_1(X_{\rm ab})\to H_1(X_{\infty}),\]
hence the equality of their d.h.Fitt's. 

Consider an exact sequence of $\Lambda_{\Z,c}$-modules
\[ 
\Lambda_{\Z,c}^r\to\Lambda_{\Z,c}^s\to H_1(X_{\rm ab})\to 0. 
\] 
Since $X$ is a compact orientable 3-manifold 
whose boundary is a union of tori and the free $\Z$-rank satisfies ${\rm rank}_{\Z} H_1(X)\geq 2$, 
McMullen's theorem \cite[Theorem 5.1]{McMullen2002ASENS} yields that 
\[
{\rm Fitt}(H_1(X_{\rm ab}))=(s_1-1,\ldots,s_c-1)\Delta_L.
\] 
By taking $\otimes\Lambda_{\Z,d}$, we have an exact sequence of $\Lambda_{\Z,d}$-modules
\[
\Lambda_{\Z,d}^r\to\Lambda_{\Z,d}^s\to H_1(X_{\rm ab})/IH_1(X_{\rm ab})\to 0.
\]
Note that the image of the augmentation ideal $(s_1-1,\ldots,s_c-1) \subset \Lambda_{\Z,c}$ via 
$\Lambda_{\Z,c}\surj \Lambda_{\Z,d}=\Lambda_{\Z,c}/I$ 
is again the augmentation ideal $(t_1-1,\ldots, t_d-1)\subset \Lambda_{\Z,d}$. 
Since $c\geq d\geq 2$, the divisorial hulls of the augmentation ideals are trivial. 

Therefore, we obtain 
\[
(\Delta_{\tau})
={\rm d.h.}{\rm Fitt}H_1(X_\infty)
={\rm d.h.}{\rm Fitt}H_1(X_{\rm ab})/IH_1(X_{\rm ab})
=(\Delta_L)\ {\rm mod}\ I 
\]
in $\Lambda_{\Z,d}$,  
where $\Delta_{\tau}$ is the Alexander polynomial of $X_{\infty}\to X$. This completes the proof of (i).\\ 

(ii) 
We apply a similar argument to arbitrary $\Zp^{\,d}$-covers. 
Let  ${\rm ab}:\wh{\pi}_1(X)\surj \wh{H}_1(X)\cong \Zp^c$ be the abelianization map 
and consider the induced map $\ul{\wh{\tau}}:\Zp^c\surj \Zp^{\,d}$. 
Replacing 
\begin{itemize}
\item $\pi_1(X_\infty)$ by $\ker \wh{\tau}$ 
\item $H_1(X_\infty)=(\ker \tau)^{\rm ab}$ by $\mca{H}=\varprojlim H_1(X_{p^n};\Zp)=(\ker \wh{\tau})^{\rm ab}$
\item $\pi_1$ and $H_1$ by $\wh{\pi}_1$ and $\wh{H}_1$ in other places 
\item $\Lambda_{\Z,\ast}$ by $\Lambda_{\ast}$, 
\item $\Z$ by $\Zp$ 
\item $\tau$ by $\wh{\tau}$ 
\item \Cref{twoexact} by \Cref{twoexact2}
\end{itemize}
in the proof of (1), we obtain a pseudo-isomorphism of $\Lambda_d$-modules 
\[\wh{H}_1(X_{\rm ab})/I\wh{H}_1(X_{\rm ab}) \to \mca{H}.\] 
Indeed, consider the commutative diagram with exact rows  
\[
\xymatrix{
0\ar[r]&\wh{H}_1(X_{\rm ab}) \ar[r]\ar@{=}[d] & \widehat{\pi}_1(X)/[\widehat{\pi}_1(X_{\rm ab}),\widehat{\pi}_1(X_{\rm ab})]\ar[r]^{\hspace{20mm} \rm ab}\ar@{}[d]|{\bigcup} & \Zp^c \ar[r] \ar@{}[d]|{\bigcup} & 0 \\
0\ar[r]& \wh{H}_1(X_{\rm ab})\ar[r] & \ker\widehat{\tau}/[\widehat{\pi}_1(X_{\rm ab}), \widehat{\pi}_1(X_{\rm ab})]\ar[r] & \ker (\ul{\wh{\tau}}:\Zp^c\surj \Zp^{\,d}) \ar[r] & 0.
}
\]
By \Cref{twoexact2} with
\begin{itemize}
\item $G:=\widehat{\pi}_1(X)/[\widehat{\pi}_1(X_{\rm ab}), \widehat{\pi}_1(X_{\rm ab})]$,
\item $\mathcal{P}:G\to \wh{\pi}_1(X)/\widehat{\pi}(X_{\rm ab})\cong \Zp^c$, so 
$N=\wh{\pi}_1(X_{\rm ab})^{\rm ab}=\wh{H}_1(X_{\rm ab}).$ 
\item $\mathcal{Z}:=\ker(\Zp^c\to\Zp^d)$ ($\cong\Zp^{c-d}$),  
so $H^{\rm ab}=(\ker\widehat{\tau})^{\rm ab}=\mathcal{H}$,
\end{itemize}
we have two exact sequences of $\Lambda_d$-modules
\begin{gather*}
0\to \wh{H}_1(X_{\rm ab})/(C+I\wh{H}_1(X_{\rm ab}))\to\mathcal{H}\to \Zp^{c-d}\to 0,\\ 
0\to(C+I\wh{H}_1(X_{\rm ab}))/I\wh{H}_1(X_{\rm ab})\to \wh{H}_1(X_{\rm ab})/I\wh{H}_1(X_{\rm ab})\to \wh{H}_1(X_{\rm ab})/(C+I\wh{H}_1(X_{\rm ab}))\to 0,
\end{gather*} 
where $I$ is the ideal of $\Lambda_c$ and $C$ is the $\Lambda_c$-submodule of $\wh{H}_1(X_{\rm ab})$ defined in \Cref{asymptotic2}.
Likewise, since $d \geq 2$, we obtain the desired pseudo-isomorphism.

Consider an exact sequence of $\Lambda_c$-modules
\[
\Lambda_c^r\to\Lambda_c^s\to \wh{H}_1(X_{\rm ab})\to 0.
\]
Let $\Delta_L$ denote the Alexander polynomial of $L$. Then, as was discussed above, we have
\[
{\rm Fitt}(H_1(X_{\rm ab}))=(s_1-1,\ldots,s_c-1)\Delta_L(s_1,\ldots,s_c).
\]
By the arguments in Section $4$, we have
\[
{\rm Fitt}(\wh{H}_1(X_{\rm ab}))=(S_1,\ldots,S_c)\Delta_L(1+S_1,\ldots,1+S_c).
\]
By taking $\otimes\Lambda_d$, we have an exact sequence of $\Lambda_d$-modules
\[
\Lambda_d^r\to\Lambda_d^s\to \wh{H}_1(X_{\rm ab})/I\wh{H}_1(X_{\rm ab})\to 0.
\]
Note that the image of $(S_1,\ldots,S_c)\subset \Lambda_c$ via $\Lambda_c\surj \Lambda_d=\Lambda_c/I$ is $(T_1,\ldots, T_d)\subset \Lambda_d$. 
Since $c\geq d\geq 2$, their divisorial hulls are trivial. 
Therefore, we obtain 
\[
(F_\tau)
={\rm d.h.Fitt}\mca{H}
={\rm d.h.Fitt}\wh{H}_1(X_{\rm ab})/I\wh{H}_1(X_{\rm ab})
=(\Delta_L){\rm \ mod\ }I
\]
in $\Lambda_d$, where $F_{\tau}$ is the characteristic element of $\mathcal{H}$. Thus we obtain the assertion (ii). 

(2) Since \cite[Theorem 5.1]{McMullen2002ASENS} is applicable to any compact orientable 3-manifold with tori boundary, if we replace $X_{\rm ab}$ by the maximal free abelian cover $X_{\rm fab}$ in the argument above, then we obtain the assertion. 
\end{proof} 

\begin{example} \label{eg.Z52}
Let $L=K_1\sqcup K_2 \sqcup K_3$ be a 3-component link in $S^3$ and consider the $\Z_5^{\,2}$-cover of  the exterior $X=S^3\setminus {\rm Int}\,V_L$ corresponding to the group homomorphism defined by 
$\tau:\pi_1(X)\to \Z_5^{\,2}$ by $\alpha_1\mapsto (1,0)$, $\alpha_2\mapsto (0,1)$, and $\alpha_3\mapsto (\sqrt{-1},0)$. 
Then the characteristic element of the Iwasawa module $\mca{H}=\varprojlim H_1(X_{p^n};\Zp)$ is given by 
\[
\Delta_L(t_1,t_2,t_1^{\sqrt{-1}})
=F_\tau(T_1,T_2)
=\Delta_L(1+T_1,1+T_2,(1+T_1)^{\sqrt{-1}})
\] in $\Lambda_2=\Zp[\![t_1^{\Zp},t_2^{\Zp}]\!]\cong \Zp[\![T_1,T_2]\!];\, t_i\mapsto 1+T_i$. 
We remark that this $\Z_5^{\,2}$-cover is not isomorphic to the one derived from any $\Z^2$-cover. 
Indeed, if so, then $\tau$ must factor through $\Z^2$, but the image ${\rm im}\tau$ is a free $\Z$-module generated by $(1,0)$, $(0,1)$, $(\sqrt{-1},0)$ and is isomorphic to $\Z^3$. 
\end{example}

\begin{remark} A similar assertion to \Cref{reduced} holds for $\wh{\Z}=\varprojlim \Z/n\Z$-cover. 
\end{remark}

\section{Estimates in $\Zp^{\,d}$-covers} 
In this section, based on our group-theoretic results obtained in \Cref{sec.twoexact}, 
we will establish several versions of estimation formulas for the torsion/rank growth in unbranched $\Zp^{\,d}$-covers over compact 3-manifolds. 

\subsection{Unbranched $\Z^d$-covers} 
\begin{theorem} \label{examplezdcover} 
Let $X_\infty\to X$ be a $\Z^d$-cover over a compact connected orientable 3-manifold $X$. Then we have 
\[e(H_1(X_{p^n}))=e(H_1(X_\infty)/I_{p^n}H_1(X_\infty))+O(n)\ (n\to \infty).\]
\end{theorem} 

\begin{proof}
Let $\tau:\pi_1(X)\twoheadrightarrow\Z^d$ be a surjective homomorphism corresponding to the given $\Z^d$-cover. 
Since $\pi_1(X)/\pi_1(X_{\infty})\cong \Z^d$, we may choose a basis $\{t_1,\ldots,t_d\}$ of $\pi_1(X)/\pi_1(X_{\infty})$. Let $\Lambda_{\Z}:=\Z[t^{\pm 1}_1,\ldots,t^{\pm 1}_d]$. Then $H_1(X_{\infty})$ becomes a finitely generated $\Lambda_{\Z}$-module in a natural way. To be precise, $H_1(X_{\infty})$ becomes a module over a group ring of the covering transformation group $\pi_1(X)/\pi_1(X_{\infty})$ over $\Z$. 
Group theoretically, this action coincides with the conjugation action on the kernel of the natural surjective group homomorphism 
\[\mathcal{P}:\pi_1(X)/[\pi_1(X_{\infty}),\pi_1(X_{\infty})]\to\pi_1(X)/\pi_1(X_{\infty})\] 
via the Hurwicz isomorphism
\[N=\ker\mathcal{P}=\pi_1(X_\infty)/[\pi_1(X_{\infty}),\pi_1(X_{\infty})]\cong H_1(X_{\infty}).\]  
Since $\{t_1,\ldots,t_d\}$ is a $\Z$-basis of $\pi_1(X)/\pi_1(X_{\infty})$, we may regard $H_1(X_{\infty})$ as a $\Lambda_{\Z}$-module. Consider the commutative diagram with exact rows 
\[
\xymatrix{
0\ar[r]&H_1(X_{\infty})\ar[r]\ar@{=}[d]&\pi_1(X)/[\pi_1(X_{\infty}), \pi_1(X_{\infty})]\ar[r]^{\hspace{20mm} \mathcal{P}}\ar@{}[d]|{\bigcup}&\Z^d\ar[r]\ar@{}[d]|{\bigcup}&0\\
0\ar[r]&H_1(X_{\infty})\ar[r]&\pi_1(X_n)/[\pi_1(X_{\infty}), \pi_1(X_{\infty})]\ar[r]&n\Z^d\ar[r]&0.
}
\]
Since $H_n^{\rm ab}=H_1(X_n)$, we obtain 
\[
e(H_1(X_{p^n}))=e(H_1(X_{\infty})/I_{p^n}H_1(X_{\infty}))+O(n)\ \ (n\to\infty). \qedhere
\]
\end{proof} 

\subsection{Unbranched $\Zp^{\,d}$-covers} 
\begin{theorem} \label{thm.unbr-asymptotic} 
Let $(X_{p^n}\to X)_n$ be an (unbranched) $\Zp^{\,d}$-cover of a compact connected orientable 3-manifold and put $\mathcal{H}:=\varprojlim H_1(X_{p^n},\Zp)$. Then 
\begin{gather} 
e(\mathcal{H}/\mathcal{I}_{p^n}\mathcal{H})=e(H_1(X_{p^n}))+O(n) \label{examplezpdcover},\\ 
r(H_1(X_{p^n},\Zp))=r(\mathcal{H}/\mathcal{I}_{p^n}\mathcal{H})+O(1) \label{rankequation}. 
\end{gather}
\end{theorem}

\begin{proof}
Let $\tau:\pi_1(X)\to \Zp^{\,d}$ be a homomorphism corresponding to the give $\Zp^{\,d}$-cover, so that it induces a surjective homomorphism $\widehat{\tau}:\widehat{\pi}_1(X)\surj\Zp^d$. 
Let $\widehat{\tau}_n:\widehat{\pi}_1(X)\surj\Z/p^n\Z$ denote the composition of 
$\widehat{\tau}$ and the natural surjection $\Zp^d\surj(\Z/p^n\Z)^d$. Consider the case $\mathcal{P}=\widehat{\tau}:\widehat{\pi}_1(X)/[\ker\tau, \ker\tau]\surj \Zp^d$ and $N=\ker\widehat{\tau}$ in \Cref{twoexact2}.
We have a commutative diagram with exact rows  
\[
\xymatrix{
0\ar[r]&(\ker\widehat{\tau})^{\rm ab}\ar[r]\ar@{=}[d]&\widehat{\pi}_1(X)/[\ker\widehat{\tau}, \ker\widehat{\tau}]\ar[r]^{\hspace{15mm}\mathcal{P}}\ar@{}[d]|{\bigcup}&\Zp^d\ar[r]\ar@{}[d]|{\bigcup}&0\\
0\ar[r]&(\ker\widehat{\tau})^{\rm ab}\ar[r]&\ker\widehat{\tau}_n/[\ker\widehat{\tau},\ker\widehat{\tau}]\ar[r]&p^n\Zp^d\ar[r]&0.
}
\]
Note that, since $(\varprojlim_i G_i)^{{\rm ab}}=\varprojlim_i(G_i^{{\rm ab}})$ for an inverse system $(G_i)_i$ of profinite groups (cf.~\cite[Chapter 4, Section 2, Exercise 6]{Neukirch}), we have
\[
(\ker{\widehat{\tau}})^{\rm ab}=(\varprojlim_n \ker \widehat{\tau}_n)^{\rm ab}
=\varprojlim_n ((\ker \widehat{\tau}_n)^{\rm ab})
\ = \varprojlim_n ((\widehat{\ker \tau_n})^{\rm ab}) 
=\varprojlim_n H_1(X_{p^n},\Zp)=\mathcal{H}.
\]
Hence we have two exact sequences of $\Lambda$-modules
\begin{gather} 
\label{ex1.unbr}
0\to \mathcal{H}/(C_{p^n}+\mathcal{I}_{p^n} \mathcal{H})\to H_1(X_{p^n},\Zp)\to (p^n\Zp)^d \to 0,\\ 
\label{ex2.unbr}
0\to(C_{p^n}+\mathcal{I}_{p^n}\mathcal{H})/\mathcal{I}_{p^n} \mathcal{H}\to \mathcal{H}/\mathcal{I}_{p^n}\mathcal{H}\to \mathcal{H}/(C_{p^n}+\mathcal{I}_{p^n}\mathcal{H})\to 0.
\end{gather}
We have that $\mathcal{H}/\mathcal{I}_{p^n}\mathcal{H}$ is finitely generated over $\Zp$ since other modules are finitely generated over $\Zp$. Since $H_1(X_{p^n},\Zp)$ are profinite, $\mathcal{H}=\varprojlim_nH_1(X_{p^n},\Zp)$ is a compact Hausdorff $\Lambda$-module. By Nakayama's lemma \cite{BalisterHowson1997}, $\mathcal{H}$ is a finitely generated $\Lambda$-module. Hence we obtain
\[
e(\mathcal{H}/\mathcal{I}_{p^n}\mathcal{H})=e(H_1(X_{p^n}))+O(n).
\]
Furthermore, since $(C_{p^n}+\mathcal{I}_{p^n}\mathcal{H})/\mathcal{I}_{p^n} \mathcal{H}$ is generated by at most $\frac{d(d-1)}{2}$ over $\Zp$, by comparing the free $\Zp$-ranks of these two exact sequences, we also have
\[
r(H_1(X_{p^n},\Zp))=r(\mathcal{H}/\mathcal{I}_{p^n}\mathcal{H})+O(1). \qedhere 
\]
\end{proof}

\subsection{$\Zp^{\,d}$-covers derived from $\Z^d$-covers} 
In this subsection, we prove the following.
\begin{theorem}\label{zcovertheorem} 
Let $X_\infty \to X$ be a $\Z^d$-cover over a compact connected orientable 3-manifold and 
let $(X_{p^n}\to X)_n$ denote the derived $\Zp^{\, d}$-cover. Then, 
the $\Lambda$-module $\mathcal{H}=\varprojlim_n H_1(X_{p^n};\Zp)$ is pseudo-isomorphic to 
$\wh{H}_1(X_{\infty}):=H_1(X_\infty)\otimes_{\Lambda_\Z}\Lambda=\varprojlim_n H_1(X_\infty;\Zp)/I_{p^n}H_1(X_\infty;\Zp)$. 
In particular, their characteristic elements in $\Lambda$ coincide with each other.
\end{theorem}

By \Cref{1124} and \Cref{zcovertheorem}, we obtain
\begin{cor}
\label{vanishmentremark} The following equation on ideals of $\Lambda$ holds: 
\[
{\rm Char}\,\mathcal{H}={\rm Char}\,\wh{H}_1(X_{\infty})=\Lambda{\rm Char}\,H_1(X_{\infty}).
\]
\end{cor} 

To prove the theorem, we will utilize the following. 
\begin{lem}[{\cite[Chapter 4, Proposition 2.7]{Neukirch}}]\label{Neukirchprop}
Let $\alpha:\{G'_i,g'_{ij}\}\to\{G_i,g_{ij}\}$ and $\beta:\{G_i,g_{ij}\}\to\{G''_i, g''_{ij}\}$ be morphisms between inverse systems of compact Hausdorff topological groups such that the sequence
\[
G'_i\stackrel{\alpha_i}{\to} G_i\stackrel{\beta_i}{\to}G''_i
\]
is exact for every $i\in I$. Then
\[
\varprojlim_i G'_i\stackrel{\alpha}{\to}\varprojlim_i G_i\stackrel{\beta}{\to}\varprojlim_i G''_i
\]
is again an exact sequence of compact Hausdorff topological groups.
\end{lem}

Note that $H_1(X_{\infty})$ is a finitely generated $\Lambda_{\Z}$-module. 
\begin{lem} \label{22718}
For each $n\in \Z_{>0}$, we have 
\[
\wh{H}_1(X_{\infty})/\mca{I}_{p^n}\wh{H}_1(X_{\infty})\cong H_1(X_\infty,\Zp)/I_{p^n}H_1(X_\infty,\Zp), 
\]
as $\Lambda$-modules. In particular, by \Cref{asymptotic}, we have
\[
e(\wh{H}_1(X_{\infty})/\mca{I}_{p^n}\wh{H}_1(X_{\infty}))=e(H_1(X_{\infty})/I_{p^n}H_1(X_{\infty}))=e(H_1(X_{p^n}))+O(n)\ \ (n\to\infty). 
\]
\end{lem}
\begin{proof} We put \glssymbol{sfHpn}${\mathsf H}_{p^n}:=H_1(X_{\infty})/I_{p^n} H_1(X_{\infty})$ for each $n\in \Z_{>0}$. 
Let $n,N\in \Z_{>0}$ and consider the commutative diagram 
\[ \xymatrix{
0 \ar[r]&I_{p^{n+N}}H_1(X_{\infty})\ar[r]\ar[d] &H_1(X_{\infty})\ar[r] \ar@{=}[d]
&{\mathsf H}_{p^{n+N}}\ar[r]\ar[d]&0\\
0 \ar[r]&I_{p^n}H_1(X_{\infty})\ar[r]&H_1(X_{\infty})\ar[r]&{\mathsf H}_{p^n}\ar[r]&0
} \]
consisting of two exact sequences. By the snake lemma, we have
\[ \ker({\mathsf H}_{p^{n+N}}\to {\mathsf H}_{p^n})\cong \coker(I_{p^n+N} H_1(X_{\infty})\to I_{p^n}H_1(X_{\infty})). \]
Hence we have an exact sequence
\[ 0\to I_{p^n}(H_1(X_{\infty})/I_{p^{n+N}}H_1(X_{\infty}))\to {\mathsf H}_{p^{n+N}}\to {\mathsf H}_{p^n}\to 0. \]
By \Cref{asymptotic}, we have 
\[ H_1(X_{\infty})/I_{p^{n+N}}H_1(X_{\infty})\cong {\mathsf H}_{p^{n+N}},\]
and hence the exact sequence
\[ 0\to I_{p^n}{\mathsf H}_{p^{n+N}}\to {\mathsf H}_{p^{n+N}}\to {\mathsf H}_{p^n}\to 0. \]
This induces the exact sequence of $\Zp[t^{\Z/p^{n+N}\Z}]$-modules
\[ 0\to I_{p^n}{\mathsf H}_{p^{n+N}}\otimes \Zp\to {\mathsf H}_{p^{n+N}}\otimes \Zp \to {\mathsf H}_{p^n}\otimes \Zp \to 0. \]
If we take the inverse limit with respect to $N$, then \Cref{Neukirchprop} yields the exact sequence of $\Lambda$-modules
\[ 0\to \mca{I}_{p^n}\wh{H}_1(X_{\infty})\to\wh{H}_1(X_{\infty})\to {\mathsf H}_{p^n}\otimes \Zp \to 0. \]
This completes the proof.
\end{proof}

\begin{proof}[Proof of \Cref{zcovertheorem}]
We have two exact sequences 
\begin{gather*}
0\to H_1(X_{\infty})/(C_{p^n}+I_{p^n}H_1(X_{\infty}))\to H_1(X_{p^n})\to p^n\Z\to 0,\\
\hspace{-20mm} 
0\to (C_{p^n}+I_{p^n}H_1(X_{\infty}))/I_{p^n}H_1(X_{\infty})\to H_1(X_{\infty})/I_{p^n}H_1(X_{\infty})\\ 
\hspace{60mm} 
\to H_1(X_{\infty})/(C_{p^n}+I_{p^n}H_1(X_{\infty}))\to 0.
\end{gather*} 
Since these terms are finitely generated $\Z$-modules, by taking $\otimes \Zp$, we obtain similar exact sequences of finitely generated $\Zp$-modules. 
By Eq.\eqref{ceqn} in the proof of \Cref{actionprop}, $C_{p^n}$ is generated by 
$c_{ij}^{(p^n)}=\mathcal{T}_{p^n}(t_i)\mathcal{T}_{p^n}(t_j)c_{ij}$, $\mathcal{T}_{p^n}(t):=\sum_{k=0}^{{p^n}-1}t^k$, 
so the family 
\[((C_{p^n}+I_{p^n}H_1(X_{\infty})/I_{p^n}H_1(X_{\infty}))\otimes\Zp)_n\] 
forms an inverse system in a natural way.  
Hence \Cref{Neukirchprop} yields two exact sequences of finitely generated $\Lambda$-modules
\begin{gather*}
0\to \varprojlim_n H_1(X_{\infty})/(C_{p^n}+I_{p^n}H_1(X_{\infty}))\otimes\Zp\to \mathcal{H} \to 0 \to 0,\\
\hspace{-20mm} 0\to\varprojlim_n (C_{p^n}+I_{p^n}H_1(X_{\infty})/I_{p^n}H_1(X_{\infty}))\otimes\Zp\to \wh{H}_1(X_{\infty}) \\ 
\hspace{40mm}
\to\varprojlim_n(H_1(X_{\infty})/(C_{p^n}+I_{p^n}H_1(X_{\infty})))\otimes\Zp \to 0.
\end{gather*} 
If $d=1$, then $C_{p^n}=0$ yields the assertion. Suppose $d\geq 2$. 
Since the finitely generated $\Lambda$-module 
$\mca{C}:=\varprojlim_n (C_{p^n}+I_{p^n}H_1(X_{\infty}))/I_{p^n}H_1(X_{\infty})\otimes\Zp$ 
admits trivial action of $t_1=1+T_1,$ $\ldots$, $t_d=1+T_d$, we find that 
$\mca{C}$ is finitely generated over $\Zp$. Hence, by $d\geq 2$, $\mca{C}$ is a pseudo-null $\Lambda$-module, and $\wh{H}_1(X_{\infty})$ is pseudo-isomorphic to $\mathcal{H}$.
\end{proof}

\section{Comparing branched and unbranched $(\Z/p^n\Z)^d$-covers} \label{sec.br/unbr}
Here, by using Hartley--Murasugi's result \cite[Lemma 2.8]{HartleyMurasugi1978}, we prove the following theorem that compares branched and unbranched $(\Z/p^n\Z)^d$-covers over a link in a $\Q$HS$^3$. 
\begin{theorem}\label{gaptheorem} 
Let $M$ be a $\Q$HS\,$^3$, $L=\sqcup_i K_i$ a $c$-component link in $M$, and put $X=M\setminus {\rm Int}\,V_L$. 
Let $(X_{p^n}\to X)_n$ be a compatible system of $(\Z/p^n\Z)^d$-covers and 
let $(M_{p^n}\to M)_n$ denote that of the branched $(\Z/p^n\Z)^d$-covers obtained by the Fox completions.  
Suppose that the link $L$ does not decompose in $(M_{p^n}\to M)_n$.
Then we have
\[e(H_1(M_{p^n}))=e(H_1(X_{p^n}))+O(n).\]
\end{theorem}

\begin{lem}\label{MVlemma} 
Let $M, L, X$ be as in above. For each $i$, let $V_i$ denote the connected component of a tubular neighborhood $V_L$ of $L$ containing $K_i$. 
Then we have a natural exact sequence
\[
0\to A\to H_1(X)\to H_1(M)\to 0
\]
of abelian groups, 
where $A$ is the subgroup of $H_1(X)$ generated by the meridians.
\end{lem}

\begin{proof}
Since $H_2(M)\simeq H_1(M)_{\rm free}=0$ and $M=X\cup\bigcup_{i=1}^c V_i$, we have the Mayer--Vietoris exact sequence
\[
0\to\bigoplus_{i=1}^cH_1(\partial V_i)\to H_1(X)\oplus\bigoplus_{i=1}^c H_1(V_i)\to H_1(M)\to 0.
\]
Since $H_1(V_i)$ are generated by the longitudes and $H_1(\partial V_i)$ are generated by both meridians and longitudes, this induces our desired exact sequence.
\end{proof}

\begin{lem}[Hartley--Murasugi, {\cite[Lemma 2.8 (ii) $\Leftrightarrow$ (iii)]{HartleyMurasugi1978}}] \label{lem.HM}
Let $M, L, X$ be as in above and let $\alpha_i\in H_1(X)$ denotes the meridian of $K_i$ for each $1\leq i\leq c$. For $a_1,\ldots,a_c\in \Z$, the following conditions are equivalent. 

{\rm (a)} $\sum_i a_i K_i=0$ holds in $H_1(M)$.

{\rm (b)} There exists a group homomorphism $\varphi:H_1(X)\to \Z$ satisfying $(\varphi(\alpha_i))_i=(a_i)_i$. 
\end{lem}

\begin{proof}[Proof of \Cref{gaptheorem}]
By \Cref{MVlemma}, we have a natural exact sequence
\[
0\to A_{p^n} \to H_1(X_{p^n})\to H_1(M_{p^n})\to 0,
\]
where $A_{p^n}$ are $\Z$-submodule of $H_1(X_{p^n})$ that are generated by elements represented by the meridians of the components of the inverse images of $L$. 
Let $\alpha_1,\ldots,\alpha_c\in H_1(X)$ be the meridians of $K_1,\ldots,K_c$ respectively.
Consider the exact sequence
\[
0\to A\to H_1(X)\to H_1(M)\to 0,
\]
where $A$ is the subgroup of $H_1(X)$ generated by $\alpha_1,\ldots,\alpha_c$.
Since $H_1(M)$ is a $\Q$HS$^3$, for each $1\leq i\leq c$,
 there exists $a_i\geq 0$ such that $a_iK_i=0$ in $H_1(M)$. 
 Hence, by Hartley--Murasugi's lemma (\Cref{lem.HM}), 
there exist group homomorphisms $\varphi_i:H_1(X)\to\Z$ such that
\[
\varphi_i(\alpha_j)=
\begin{cases}
a_i&j=i\\
0&j\neq i.
\end{cases}
\]
Consider the commutative diagram with exact rows 
\[
\xymatrix{
0\ar[r]&A_{p^n}\ar[r]\ar[d]&H_1(X_{p^n})\ar[r]\ar[d]^{\psi_{p^n}}&H_1(M_{p^n})\ar[r]\ar[d]&0\\
0\ar[r]&A\ar[r]&H_1(X)\ar[r]&H_1(M)\ar[r]&0.
}
\]
Let $K$ be a knot in $M_{p^n}$ lying above some $K_i$ and let $\alpha$ be a meridian of $K$. Then $(\varphi_i\circ\psi_{p^n}):H_1(X_{p^n})\to \Z$ is a group homomorphism such that 
\[
(\varphi_i\circ\psi_{p^n})(\alpha)=\varphi_i(e_i\alpha_j)=e_i\varphi_i(\alpha_j)=
\begin{cases}
e_ia_i&j=i\\
0&j\neq i,
\end{cases}
\]
where $e_i$ is the branch index of $K_i$ for $M_{p^n}\to M$. 
Again by \Cref{lem.HM}, we have that $e_ia_iK=0$ in $H_1(M_{p^n})$. 
Since $(\pi_1(X):\pi_1(X_{p^n}))=p^{dn}$, by the Hilbert ramification theory for knots, if we put $a={\rm max}\{a_1,\ldots,a_c\}$, then we must have $p^{dn}aK=0$ in $H_1(M_{p^n})$. 
Since $L$ does not decompose in $M_{p^n}\to M$, then $r(H_1(X_{p^n}),\Zp)=O(1)$. 
Therefore, we have
\[
e(H_1(M_{p^n}))=e(H_1(X_{p^n}))+O(n). \qedhere
\] 
\end{proof}

\begin{remark} In the situation above, we may define a $\Lambda$-module \glssymbol{HM}$\mca{H}_M:=\varprojlim H_1(M_{p^n};\Zp)$. 
If $L$ is truly branched in the $\Zp^{\,d}$-cover, then we find that $\mca{H}=\mca{H}_M$. 
In this paper, we derive the Iwasawa-type formula for branched covers from that for unbranched covers. 
We leave a direct estimation of $H_1(M_{p^n};\Zp)$ by $\mca{H}_M$ in a general setting as a further problem. 
\end{remark}

\section{Iwasawa-type formulas} 
In this section, we prove several versions of the Iwasawa-type formulas for $\Zp^{\,d}$-covers, 
which are analogues of Cuoco--Monsky's formulas. 

\subsection{Iwasawa-type formulas for unbranched $\Zp^{\,d}$-covers} 
In this subsection, we prove Iwasawa-type formulas for unbranched $\Zp^{\,d}$-covers.

\begin{theorem}\label{2252} 
Let $(X_{p^n}\to X)_n$ be a $\Zp^{\,d}$-cover of a compact connected orientable $3$-manifold.

{\rm (1)} Suppose that $r(\mca{H}/\mathcal{J}_{p^n}\mca{H})=O(p^{(d-2)n})$. Then we have
\[
e(H_1(X_{p^n}))=\mu(F) p^{dn}+O(np^{(d-1)n}).
\]

{\rm (2)} Suppose that $r(\mca{H}/\mathcal{I}_{p^n}\mca{H})=O(p^{(d-2)n})$. Then we have
\[
e(H_1(X_{p^n}))=\left(\mu(F) p^{n}+\lambda(F)n+O(1)\right)p^{(d-1)n}.
\]
\end{theorem}
\begin{proof} 
The assertion (1) follows from \Cref{Monsky21} (2). The assertion (2) follows from \Cref{22510} 
and Eq.\eqref{examplezpdcover} in \Cref{thm.unbr-asymptotic}. 
\end{proof} 

Recall that $W$ denotes the set of all roots of unity of $p$-power order in $\ol{\Q}_p$ and that $\mca{W}=\{\xi-1\in \ol{\Q}_p\mid \xi\in W\}$.  
If ${\rm char}\,\mathcal{H}$ does not vanish on $\mathcal{W}^d\setminus\{(0,0,\ldots,0)\}$, then, by \Cref{lemmaforbetti}, we have $r(\mca{H}/\mathcal{I}_{p^n}\mca{H})=O(p^{(d-2)n})$. 
If instead ${\rm char}\,\mathcal{H}$ does not vanish on ($\mathcal{W}\setminus\{0\})^d$, then, by \Cref{2285}, we have $r(\mca{H}/\mathcal{J}_{p^n}\mca{H})=O(p^{(d-2)n})$.  

When we consider $\Zp^{\,d}$-covers derived from $\Z^d$-covers, by Remarks \ref{Alexanderremark} and \ref{vanishmentremark}, 
the condition ``${\rm char}\,\mathcal{H}$ does not vanish on $\mca{W}^d\setminus\{(0,0,\ldots,0)\}$ (resp.~$(\mca{W}\setminus\{0\})^d$)'' 
is equivalent to that 
``the Alexander polynomial $\Delta$ does not vanish on $W^d\setminus\{(1,1,\ldots,1)\}$ (resp.~$(W\setminus\{1\})^d$).''

\subsection{Sufficient conditions for the rank assumptions}
In this subsection, we further study sufficient conditions for the rank assumptions of \Cref{2252} in cases of link exteriors. 
As in \Cref{sec.br/unbr}, let $M$ be a $\Q$HS$^3$, let $L=\sqcup_i K_i$ be a $c$-component link, and put $X=M\setminus {\rm Int}\,V_L$. 
Let $(X_{p^n}\to X)_n$ be a $\Zp^{\,d}$-cover and let $(M_{p^n}\to M)_n$ denote the branched $\Zp^{\,d}$-cover obtained by the Fox completions. 

\begin{prop}\label{decompositionprop} 
Suppose that every $M_{p^n}$ is a $\Q$HS\,$^3$ and that $L$ does not decompose in $(M_{p^n}\to M)_n$. 
Then $r(\mathcal{H}/\mathcal{I}_{p^n}\mathcal{H})=O(1)$.
\end{prop}

\begin{proof}
By \Cref{MVlemma}, we have an exact sequence 
\[ 0\to A_{p^n}\otimes\Zp\to H_1(X_{p^n},\Zp)\to H_1(M_{p^n},\Zp)\to 0. \] 
Since $r(H_1(M,\Zp))=0$, we have $r(H_1(X_{p^n},\Zp))=r(A_{p^n}\otimes\Zp)$. Since $L$ does not decompose in $(M_{p^n}\to M)_n$, we have $r(H_1(X_{p^n},\Zp))=O(1)$. By the fundamental two exact sequences
\eqref{ex1.unbr}, \eqref{ex2.unbr} in the proof of \Cref{thm.unbr-asymptotic} together with the estimation by $\frac{d(d-1)}{2}$, 
we obtain $r(\mathcal{H}/\mathcal{I}_{p^n}\mathcal{H})=O(1)$.
\end{proof}

\begin{lem}[Torres condition, \cite{Torres1953, Cimasoni2004BSMM}]\label{Torres}
Suppose $M$ is a $\Z$HS\,$^3$. Then we have
\[
\Delta_L(t_1,\ldots,t_{c-1},1)=
\begin{cases}
\frac{t_1^{l_1}-1}{t_1-1}\Delta_{L'}(t_1)&c=2\\
(t_1^{l_1}\cdots t_{c-1}^{l_{c-1}}-1)\Delta_{L'}(t_1,\ldots,t_{c-1})&c>2,
\end{cases}
\]
where $L'=K_1\sqcup\cdots\sqcup K_{c-1}$ and \glssymbol{lk}$l_i:={\rm lk}(K_i,K_c)$.
\end{lem}

\begin{prop} \label{prop.notdecompose}
Suppose that $M$ is a $\Z$HS\,$^3$. If $\Delta_L(t_1,\ldots,t_c)$ does not vanish on $W^c\setminus \{(1,1,\ldots,1)\}$, then $L$ does not decompose in $(M_{p^n}\to M)_n$.
\end{prop}

\begin{proof}
We prove the statement by contraposition. Suppose that $L$ decomposes. Then, by the Hilbert ramification theory of links, there exist $1\leq i<j\leq c$ such that $p\mid{\rm lk}(K_i,K_j)$ (cf.~\cite[Section 5.1]{Morishita2012}, \cite{Ueki1}). We may assume $i=1$ and $j=c$ without loss of generality. Hence $p\mid l_1$. Let $m\in\Z_{\geq 1}$ such that $p^m\mid l_1$ and $p^{m+1}\nmid l_1$. Let $\xi_{p^m}$ be a $p^m$-th primitive root of unity. Then $\xi_{p^m}^{l_1}=1$, i.e., $\xi_{p^m}^{l_1}-1=0$. By the Torres condition (\Cref{Torres}), this implies
\[
\Delta_L(\xi_{p^m},1,\ldots,1)=0.
\]
Hence $\Delta_L$ vanish on $W^c\setminus \{(1,1,\ldots,1)\}$. This completes the proof.
\end{proof}

\subsection{Iwasawa-type formula for branched $\Zp^{\,d}$-covers}
\begin{theorem}\label{thm.branched} 
Let $(M_{p^n}\to M)_n$ be a branched $\Zp^{\,d}$-cover over $(M,L)$, that is, an inverse system of branched $(\Z/p^n\Z)^d$-covers over $M$ branched along $L$. Suppose that every $M_{p^n}$ is a $\Q$HS\,$^3$ and $L$ does not decompose in ($M_{p^n}\to M)_n$. Then there exist invariants $\mu,\lambda\in\Z_{\geq0}$, depending only on $p$ and $(M_{p^n}\to M)_n$, such that
\begin{equation*}
e(H_1(M_{p^n}))=(\mu p^n+\lambda n+O(1))p^{(d-1)n}.
\end{equation*}
Here, $\mu=\mu(F)$ and $\lambda=\lambda(F)$ for $F={\rm char}\,\mathcal{H}={\rm char}\,\varprojlim H_1(X_{p^n},\Zp)$. If $(M_{p^n}\to M)_n$ is derived from a $\Z^d$-cover, then $\mu=\mu(\Delta(T_1,\ldots,T_d))$ and $\lambda=\lambda(\Delta(T_1,\ldots,T_d))$.
\end{theorem}

\begin{proof}
This follows from \Cref{2252} and \Cref{decompositionprop}.
\end{proof}

\begin{remark} 
In the cases with $d=1$ (cf.~\Cref{thm.d=1}), every $M_{p^n}$ is a $\Q$HS$^3$ if and only if the characteristic element does not vanish outside zero \cite[Theorem 4.17]{Ueki2}. 
In the cases with $d\geq 2$, this condition is related to the rank assumption, as we have discussed so far. 

When $d=1$, if we assume that $L$ does not contain any unbranched component, then the Hilbert ramification theory yields that $L$ finitely decomposes. In the case with $d\geq 2$ stated above, the assumption ``$L$ does not decompose'' may be replaced by that ``$L$ finitely decomposes'' by replacing the base space $M$ by $M_{p^{n_0}}$ for a sufficiently large $n_0$. See also \Cref{rem.L412}. 
\end{remark} 

By \Cref{Monsky21}, \Cref{thm.unbr-asymptotic} \eqref{examplezpdcover}, and \Cref{gaptheorem}, we further have the following, 
as was the case for $\Zp^{\,d}$-extension \cite{Monsky1989}. 

\begin{theorem} \label{rem.alpha}
In \Cref{2252} (2) and in \Cref{thm.branched}, the $O(1)p^{(d-1)n}=O(p^{(d-1)n})$ part may be refined to $\mu_1p^{(d-1)n}+O(np^{(d-2)n})$ for some $\mu_1\in\R$. If $d=2$, then $\mu_1\in\Q$. 
\end{theorem} 

\subsection{A refined Iwasawa-type formula over $\Z$HS$^3$}
In this subsection, we establish a refined version of the Iwasawa-type formula inspired by Greenberg's conjecture for a $\Zp^d$-cover over a $\Z$HS$^3$, based on the results of Mayberry--Murasugi and Porti. 

Let $M$ be a $\Z$HS$^3$, let $L=\sqcup_i K_i$ be a $c$-component link in $M$, and put $X:=M\setminus {\rm Int}\,V_L$. 
Let $G$ be a finite abelian group and $\pi:\pi_1(X)\twoheadrightarrow G$ a surjective group homomorphism. Let $M_{\pi}$ denote the covering of $M$ branched along $L$ corresponding to $\ker\pi$. Let
\[
\widehat{G}:=\{\xi:G\to \mathbb{C}^*\mid \xi\mbox{ is a group homomorphism}\}
\]
denote the Pontryagin dual of $G$. Fix meridians $\alpha_1,\ldots, \alpha_d\in H_1(X)$. For arbitrary $\xi\in\widehat{G}$, let $L_{\xi}:=\bigcup_{\xi(\pi(\alpha_i))\neq 1}K_i$ be a sublink of $L$ and $\Delta_{L_{\xi}}(t_{i_1},\ldots,t_{i_k})$ denote the Alexander polynomial of $L_\xi$. For the trivial representation $G\to\mathbb{C}^*$, one has $L_1=\emptyset$. We put $\Delta_{L_1}:=1$ and 
\[
\widehat{G}^{(1)}:=\{\xi\in\widehat{G}\mid L_{\xi}=K_i\mbox{ for some }1\leq i\leq d\}.
\]
For arbitrary $\xi\in \widehat{G}^{(1)}$, let $i(\xi)$ denote the corresponding $i$. Put
\[
|H_1(M_\pi)|:=
\begin{cases}
\#{H_1(M_\pi})&\mbox{if finite}\\
0&\mbox{if infinite}.
\end{cases}
\]

\begin{prop}[{\cite[Theorem 10.1]{MayberryMurasugi1982}, \cite[Theorem 1.1]{Porti2004}}]\label{2260}
We have
\[
|H_1(M_\pi;\Z)|=\pm\frac{|G|}{\prod_{\xi'\in\widehat{G}^{(1)}}(1-\xi'(\pi(\alpha_{i(\xi')})))}\prod_{\xi\in\widehat{G}}\Delta_{L_\xi}(\xi(\pi(\alpha_{i_1})),\ldots,\xi(\pi(\alpha_{i_k}))).
\]
\end{prop}

\begin{theorem}\label{thm.integral} 
Let $M$ be a $\Z$HS\,$^3$, let $L=\sqcup_iK_i$ be a $c$-component link in $M$ with $c\geq d$, and 
suppose that the Alexander polynomial $\Delta_L(t_1,\ldots,t_c)$ does not vanish on $(W\setminus \{1\})^c$. 
Let $(M_{p^n}\to M)_n$ be a branched $\Zp^{\,d}$-cover over $(M,L)$. 
Then, every $M_{p^n}$ is a $\Q$HS\,$^3$, and there exists $f(U,V)\in\Q[U,V]$ with $\deg_{V}f\leq1$ and total degree $\deg f\leq d$ such that,  for every sufficiently large $n$, 
\[
e(H_1(M_{p^n}))=f(p^n, n) 
\] 
holds, and the coefficients of $U^d$ and $U^{d-1}V$ are in $\Z_{\geq 0}$. 
In other words, there exist some $\mu,\lambda\in\Z_{\geq 0}$ and 
$\mu_1,\ldots,\mu_{d-1},\lambda_1,\ldots,\lambda_{d-1},\nu\in \Q$ 
such that, for every sufficiently large $n$, 
\[e(H_1(M_{p^n}))= 
\mu p^{dn}+\lambda np^{(d-1)n}+\mu_1p^{(d-1)n}+\lambda_1 np^{(d-2)n}+\ldots +\mu_{d-1} p^n+\lambda_{d-1} n+\nu
\]
holds. 
\end{theorem} 

\begin{proof} 
Since $\Delta_L$ does not vanish on $(W\setminus\{1\})^c$, neither does $\Delta_{L'}$ on $(W\setminus\{1\})^{c(L')}$ for any $L'\subset L$ by the Torres condition (\Cref{Torres}). 
Let $\tau:\Z^c\to\Zp^d$ be a homomorphism that induces a surjective homomorphism 
$\Zp^{\,c} \surj \Zp^{\,d}$ corresponding to the given $\Zp^{\,d}$-cover. 
Put $\bm{v}_i:=\tau(\alpha_i)$. Then $\bm{v}_1,\ldots,\bm{v}_c$ generate $\Zp^{\,d}$. 
For each $n\geq 1$, let $\tau_n$ denote the composite map of $\tau$ and a canonical surjection $\Zp^d\surj(\Z/p^n\Z)^d$. If $\bm{v}_i=(v_{i1},\ldots,v_{id})$, put
\[
\zeta^{(v_{i1},\dots,v_{id})}:=\zeta_1^{v_{i1}}\cdots\zeta_d^{v_{id}}.
\]
By \Cref{2260}, we have
\begin{align*}
|H_1(M_{p^n})|&=\pm\frac{|G|}{\prod_{\xi'\in\widehat{G}^{(1)}}(1-\xi'(\tau_n(\alpha_{i(\xi')})))}\prod_{\xi\in\widehat{G}}\Delta_{L_\xi}(\xi(\tau_n(\alpha_{i_1})),\ldots,\xi(\tau_n(\alpha_{i_k})))\\
&=\pm\frac{p^{dn}}{\prod_{i=1}^c \prod_{\substack{\zeta\in W(n)^d\\ \zeta^{\bm{v}_i}\neq 1\\ \zeta^{\bm{v}_j}=1 (j\neq i)}}(1-\zeta^{\bm{v}_i})}\prod_{L'\subset L}\prod_{\substack{\zeta\in W(n)^d\\ \zeta^{\bm{v}_j}\neq 1\, (j \in \{i_1,\ldots i_{c(L')}\})\\ \zeta^{\bm{v}_j}= 1\, (j \notin \{i_1,\ldots i_{c(L')}\}) }}\Delta_{L'}({\zeta}^{\bm{v}_1},\ldots,{\zeta}^{\bm{v}_c})
\end{align*}
Therefore, we have
\begin{align*} 
e(H_1(M_{p^n}))&=v\Big(
\pm\frac{p^{dn}}{\prod_{i=1}^c \prod_{\substack{\zeta\in W(n)^d\\\zeta^{\bm{v}_i}\neq 1\\ \zeta^{\bm{v}_j}=1\, (j\neq i)}}(1-\zeta^{\bm{v}_i})}\prod_{L'\subset L}\hspace{-4mm}\prod_{\substack{\zeta\in W(n)^d\\ \zeta^{\bm{v}_j}\neq 1\, (j \in \{i_1,\ldots i_{c(L')}\})\\ \zeta^{\bm{v}_j}=1\, (j \notin \{i_1,\ldots i_{c(L')}\})}}\hspace{-8mm}\Delta_{L'}({\zeta}^{\bm{v}_{i_1}},\ldots,{\zeta}^{\bm{v}_{i_{c(L')}}})
\Big)\\
&=v(p^{dn})-v\Big({\prod_{i=1}^c \prod_{\substack{\zeta\in W(n)^d\\\zeta^{\bm{v}_i}\neq 1\\ \zeta^{\bm{v}_j}=1\, (j\neq i)}}(1-\zeta^{\bm{v}_i})}\Big) +\sum_{L'\subset L}\hspace{-4mm}\sum_{\substack{\zeta\in W(n)^d\\ \zeta^{\bm{v}_j}\neq 1\, (j \in \{i_1,\ldots i_{c(L')}\}) \\ \zeta^{\bm{v}_j}=1\, (j \notin \{i_1,\ldots i_{c(L')}\})}}\hspace{-8mm}v\Big(\Delta_{L'}({\zeta}^{\bm{v}_{i_1}},\ldots,{\zeta}^{\bm{v}_{i_{c(L')}}})\Big).
\end{align*} 
Since the map $W^d\to W; (\zeta_1,\ldots,\zeta_d)\mapsto \zeta^{\bm{v}_i}$ is an element of ${\rm Hom}(W^d,W)$ for every $\bm{v}_i$, by the definition of semi-algebraic sets, for each $1\leq i\leq c$, we have that 
$\{\zeta\in W^d\mid \zeta^{\bm{v}_i}\neq 1, \zeta^{\bm{v}_j}=1\ (j\neq i)\}$ and $\{\zeta\in W^d\mid \zeta^{\bm{v}_j}\neq 1\, (j \in \{i_1,\ldots i_{c(L')}\}), \zeta^{\bm{v}_j}=1\ (j\notin\{i_1,\ldots, i_{c(L')}\})\}$ are semi-algebraic. 
Hence, by applying \Cref{22521} to the second and the third terms, we obtain 
$g(U,V), h(U,V)\in\Q(U,V)$ with $\deg_Vg,\deg_Vh\leq 1$ and $\deg g, \deg h\leq d$ such that
\[
e(H_1(M_{p^n}))=dn-g(p^n,n)+h(p^n,n)
\]
holds. By putting $f(U,V)=dV-g(U,V)+h(U,V)$, we obtain the assertion. 
\end{proof}

In the proof above, by \cite[Theorem 4.2]{MayberryMurasugi1982}, we find that $0\leq dn-g(p^n,n)\leq dn$, 
so the contribution of $h(p^n,n)$ derived from the term of $\Delta_{L'}$ is essential. 
We may further say the following. 

\begin{theorem} \label{thm.integral.Delta} 
Let the setting be as in \Cref{thm.integral}. 

{\rm (1)} If $c=d$, 
so that $X_\infty \to X$ is the maximal abelian cover, 
then we have $\mu=\mu(\Delta_L)$ and $\lambda=\lambda(\Delta_L)$. 

{\rm (2)} If $\Delta_L(t_1,\ldots,t_c)$ does not vanish on $W^c\setminus\{(1,1,\ldots,1)\}$, then neither does the reduced Alexander polynomial $\Delta$ {\rm (}resp. characteristic element ${\rm char}\,\mathcal{H}${\rm )} on $W^d\setminus\{(1,1,\ldots,1)\}$ {\rm (}resp. $\mca{W}^d\setminus\{(0,0,\ldots,0)\}${\rm )},
and the $\mu$ and $\lambda$ coincide with those of $\Delta$. 
\end{theorem} 

\begin{proof} 
(1) 
We are considering the case when $G=(\Z/p^n\Z)^c$. Note that the equation 
\[
\prod_{\xi\in W(n)\setminus\{1\}}(x-\xi)=x^{p^n-1}+x^{p^n-2}+\cdots+x^2+x+1
\]
of polynomials implies
\[
\prod_{\xi\in W(n)\setminus\{1\}}(1-\xi)=p^n.
\]
Therefore, we have
\[
\frac{|G|}{(\prod_{\xi\in W(n)\setminus\{1\}}(1-\xi))^d}=1.
\]
Hence, by \Cref{2260}, we obtain
\[
|H_1(M_{\pi})|=\pm \prod_{L'} \prod_ {\zeta\in (W(n)\setminus\{1\})^{c(L')}}\Delta_{L'}(\zeta),
\]
where $L'$ runs over the sublinks of $L$ and $c(L')$ is the number of components of $L'$. Hence we obtain
\begin{align*}
e(H_1(M_{p^n}))&=v\big(\pm \prod_{L'} \prod_{\zeta\in (W(n)\setminus\{1\})^{c(L')}}\Delta_{L'}(\zeta)\big)\\
&=\sum_{L'}\sum_{\zeta\in (W(n)\setminus\{1\})^{c(L')}} v(\Delta_{L'}(\zeta)).
\end{align*}
By \Cref{CM32}, we obtain the assertion. 

(2)  The vanishing property follows from \Cref{reduced}. 
By Theorems \ref{thm.branched} and \ref{thm.integral} together with \Cref{decompositionprop}, we obtain the assertion. 
\end{proof} 

\begin{remark} 
Greenberg's conjecture is proved in the affirmative for function fields by Wan \cite{DWan2019JNT}, 
for graphs by DuBose--Valli${\rm \grave{e}}$res \cite{DuboseVallieres2023AC} and Kleine--M\"{u}ller \cite{KleineMuller.2023Jacobian-arXiv}, and for number fields satisfying certain assumptions in the case $d=2$ by Kleine \cite[Theorem 5.4]{Kleine2021ForumMath}. 
The case for $\Zp^{\,d}$-covers over a $\Q$HS$^3$ is still mysterious; It would be interesting if we could construct a non-Greenberg example. 
We should continue comparing these four objects: number fields, function fields, links, and graphs.
\end{remark} 

\begin{remark} \label{remark.MR1and5} 
In the case of the branched $\Zp^{\,c}$-cover of a $c$-component link in a $\Z$HS$^3$ with explicit Alexander polynomial, 
Propositions \ref{prop.notdecompose} and \ref{2260} give a sufficient condition for the assumptions of \Cref{mainresult} (Theorems \ref{thm.branched}, \ref{rem.alpha}), 
which is stronger than the assumption of \Cref{mr.ZHS} (Theorems \ref{thm.integral}, \ref{thm.integral.Delta}). 
In a general setting, there is no implication between the assumptions of Main Results \ref{mainresult} and \ref{mr.ZHS}, 
though we do not have an example that only \Cref{mainresult} applies at this moment. 
The theory of Alexander polynomials of the covers over a general $\Q$HS$^3$ is yet to be developed and will be addressed in our future study.  
\end{remark} 

\section{Polynomial periodicity for Betti numbers}
The relationship between the torsion sizes of the 1st homology group and the 1st Betti numbers of 3-manifolds are like siblings. 
Here, we recall the notion of polynomial periodicity and prove a result on the Betti numbers in $\Zp^{\,d}$-covers. 

\begin{definition} 
A sequence $\{\beta(n)\}_{n\geq1}$ of integers is said to be \emph{polynomial periodic} if and only if there exist some $N\geq1$ and polynomials $f_0(U),\ldots,f_{N-1}(U)\in\Q[U]$ such that, if $n\equiv i\mod N$ with $0\leq i\leq N-1$, then $\beta(n)=f_i(n)$.
\end{definition}

\begin{prop} [{\cite[Corollary 1.4]{AdamsSarnak1994}}, {\cite[Theorem 7.5]{Sakuma1995Canada}}]\label{22511}
Let $M$ be a $\Q$HS\,$^3$ and $L$ a $d$-component link in $M$. 
Let $(M_n\to M)_n$ denote the compatible system of branched $(\Z/n\Z)^d$-cover over $(M,L)$ obtained from the unique $\Z^d$-cover of the exterior $X=M\setminus {\rm Int}\,V_L$. 
Then the Betti number $\beta_1(M_n)$ is polynomial periodic.
\end{prop}
Since $p^n$ mod $N$ is periodic in $n$, in the setting of \Cref{22511}, 
there exists some $N'\in \Z_{>0}$ and $\{g_0,\ldots,g_{N'-1}\}\subset \{f_1,\ldots ,f_{N-1}\}$ such that $\beta_1(M_{p^n})=g_i(p^n)$ if $n\equiv i$ mod $N'$. 
Put \glssymbol{dml}
\[
D(M,L):=\max\{\deg g_i\mid 0\leq i\leq N'-1\}.
\]
\begin{theorem}\label{2284} Let the setting be as above with $d\geq 2$. Suppose that $\Delta_L(t_1,\ldots,t_d)$ does not vanish on $W^d\setminus\{(1,1,\ldots,1)\}$. 
Then $D(M,L)\leq d-2$. 
In particular, if $d=2$, then $\beta_1(M_{p^n})$ is periodic. 
\end{theorem}
\begin{proof}
Since $\Delta_L(t_1,\ldots,t_d)$ does not vanish on $W^d\setminus\{(1,1,\ldots,1)\}$, by \Cref{lemmaforbetti}, we have
\[
\rank_{\Zp}(\wh{H}_1(X_{\infty})/I_{p^n}\wh{H}_1(X_{\infty}))=O(p^{(d-2)n}).
\]
By \eqref{rankequation}, we have 
\[\rank_{\Zp}(\wh{H}_1(X_{\infty})/I_{p^n}\wh{H}_1(X_{\infty}))
=\rank_{\Zp}(H_1(X_{p^n},\Zp))
=\rank_{\Z} H_1(X_{p^n}).\]
Since there is a surjective homomorphism $H_1(X_{p^n})\surj H_1(M_{p^n})$, we obtain \[
\rank_{\Z}(H_1(M_{p^n}))=O(p^{(d-2)n}),
\]
which is equivalent to $D(M,L)\leq d-2$.
\end{proof} 

\begin{remark}
In the case of $\Zp^{\,c}$-cover of a $c$-component link in a $\Z$HS$^3$, 
by Porti's result \Cref{2260}, $\Delta\neq 0$ on $(W\setminus \{1\})^c$ implies that $\beta_1(M_{p^n})=0$ for all $n$, which is one of the assumptions of \Cref{mainresult}. 
\Cref{2284} would play a role when we deal with a general case. 
\end{remark}

\section{Twisted Whitehead links}
In this section, as a preparation to investigate examples, 
we recall the definition of twisted Whitehead links $W_{\kappa}$ ($\kappa\in \Z$) and calculate the Alexander polynomial. 
\begin{definition} 
For each $\kappa\in \Z_{\geq 0}$, the $\kappa$-twisted Whitehead link \glssymbol{whiteheadlink}$W_{\kappa}$ in $S^3$ is defined by the following diagram.
\[\xygraph{
!{0;/r2.0pc/:}
!{\hcap[-3]}
[d]!{\xcapv-}
[uu]!{\xunderh=>}
[ldd]!{\xunderh}
[uu]!{\xcaph@(0)}
[ld]!{\xoverh}
[ldd]!{\xcaph@(0)}
}
\xymatrix@R=2.4pt{
\!\! \overline{\hspace{30mm}{\phantom{.}}} \!\!
\\ \ \\
\cdots\\
\kappa\mbox{ crossings}\\
\cdots \\
\!\! \overline{\hspace{30mm}\phantom{.}} \!\!
}
\xygraph{
!{0;/r2.0pc/:}
!{\xcaph@(0)}
[ld]!{\xoverh}
[ldd]!{\xcaph@(0)}
[uuu]!{\xoverh}
[ldd]!{\xoverh}
[u]!{\xcapv}
[uu]!{\hcap[3]}
}
\] 
\end{definition} 

\begin{prop}\label{prop.Delta_W} 
{\rm (1)} \cite[Section 3]{Kidwell1978Illinois}
If $m\geq0$, then
\[
\Delta_{W_{2m+1}}(x,y)=1+m-mx-my+(1+m)xy.
\]

{\rm (2)} \cite[Chapter 7, I, Exercise 10]{Rolfsen1990}
If $m\geq1$, then
\[
\Delta_{W_{2m}}(x,y)=m(1+xy-x-y).
\]
\end{prop}

In what follows, we present a proof of \Cref{prop.Delta_W} for the convenience of the readers. 
We invoke the notion of \emph{the Conway potential functions} of links, 
which was introduced by Conway \cite{Conway1970CPAA} and 
whose existence was explicitly proved by Hartley \cite[Section 2]{Hartley1983CMH}. 
The Conway potential function \glssymbol{conway}$\nabla\in\Lambda_{\Z}$ of a link $L$ in $S^3$ is defined by 
certain five axioms containing the Skein relations (cf.~\cite{Massuyeau2008AP}) and conveys information the same as the Alexander polynomial. We make use of the following properties.

\begin{lem}[{\cite[(5.5)]{Hartley1983CMH}}]\label{2271}
We have
\[
\nabla(t_1,\ldots,t_d)=(-1)^d\nabla(t_1^{-1},\ldots,t_d^{-1}).
\]
\end{lem}

\begin{lem} [{\cite[(1.1)]{Hartley1983CMH}}]\label{AlexanderConway}
We have
\[
\nabla(t_1,\ldots,t_d)=\Delta(t_1^2,\ldots,t_d^2)t_1^{m_1}\cdots t_d^{m_d},
\]
where $\Delta$ is the Alexander polynomial properly chosen with correct sign and $m_i$ are integers that are uniquely determined by the requirement that $\nabla$ satisfies the \Cref{2271}.
\end{lem}

\begin{lem} [Replacement relations, {\cite[(5.1),(5.2)]{Hartley1983CMH}}]

{\rm (1)} 
Let $L_{00}$ be a link that contains a configuration
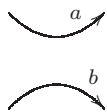
\[
\xygraph{
[u]!{\xunoverh=>|{a}>{b}}}
\]
where $a,b$ are segments from distinct knots. Let $L_{++},L_{--}$ be the links obtained by replacing the configuration by
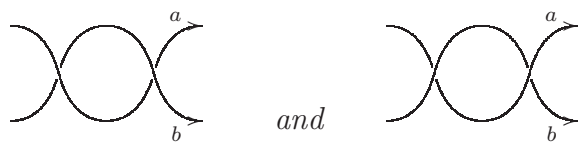
\[
\xygraph{[u]!{\htwist}!{\htwist=>>{a}|{b}}}\qquad\mbox{and} \qquad\xygraph{[u]!{\htwistneg}!{\htwistneg=>|{a}>{b}}}.
\]
Let $\nabla_{++}$, $\nabla_{--}$, and $\nabla_{00}$ denote the Conway potential function of $L_{++}$, $L_{--}$, and $L_{00}$ respectively. Then we have
\[
\nabla_{++}+\nabla_{--}=(t_at_b+t_a^{-1}t_b^{-1})\nabla_{00}.
\]

{\rm (2)} 
Consider the case where one of the arcs of $(1)$ is oppositely oriented. Then we have
\[
\nabla_{++}+\nabla_{--}=(t_at_b^{-1}+t_a^{-1}t_b)\nabla_{00}.
\]
\end{lem}

\begin{proof}[Proof of \Cref{prop.Delta_W}]
We prove 
\begin{gather*}
\nabla_{W_{2m+1}}=-(m+1)(t_at_b+t_a^{-1}t_b^{-1})+m(t_at_b^{-1}+t_a^{-1}t_b),\\ 
\nabla_{W_{2m}}=m(t_at_b-t_at_b^{-1}-t_a^{-1}t_b+t_a^{-1}t_b^{-1})
\end{gather*} 
by induction on $m$. It is known that
\begin{enumerate}
\item[i)]
If $L$ is a split link, then $\nabla_L=0$.
\item[ii)]
If $L:=\hspace{-3mm} \xygraph{
!{0;/r1.0pc/:}
[u]!{\hcap[-2]<}
!{\vover}
!{\vover-}
[uur]!{\hcap[2]<}
}\hspace{-3mm} 
$, then $\nabla_L=1$.
\item[iii)]
If $L:=\hspace{-3mm}\xygraph{
!{0;/r1.0pc/:}
[u]!{\hcap[-2]<}
!{\vover}
!{\vover-}
[uur]!{\hcap[2]>}
}\hspace{-3mm}$, then $\nabla_L=-1$.
\end{enumerate}
Let
\[
L_{++}:=\hspace{-5mm} 
\xygraph{
!{0;/r1.0pc/:}
[u]!{\hcap[-3]}
[d]!{\xcapv-}
[uu]!{\xunderh}
[ldd]!{\xunderh}
[uu]!{\xcaph@(0)}
[ld]!{\xcaph@(0)}
[ld]!{\xcaph@(0)}
[ld]!{\xcaph@(0)}
[uuu]!{\xoverh=>}
[ldd]!{\xoverh=<}
[u]!{\xcapv}
[uu]!{\hcap[3]}
}\hspace{-5mm},\ \ 
L_{--}:=\hspace{-5mm} 
\xygraph{
!{0;/r1.0pc/:}
[u]!{\hcap[-3]}
[d]!{\xcapv-}
[uu]!{\xunderh}
[ldd]!{\xunderh}
[uu]!{\xcaph@(0)}
[ld]!{\xcaph@(0)}
[ld]!{\xcaph@(0)}
[ld]!{\xcaph@(0)}
[uuu]!{\xunderh=>}
[ldd]!{\xunderh=<}
[u]!{\xcapv}
[uu]!{\hcap[3]}
}\hspace{-5mm},\ \  
%\]
%
%\[
L_{00}:=\hspace{-5mm} 
\xygraph{
!{0;/r1.0pc/:}
[u]!{\hcap[-3]}
[d]!{\xcapv-}
[uu]!{\xunderh}
[ldd]!{\xunderh}
[uu]!{\xcaph@(0)}
[ld]!{\xcaph@(0)}
[ld]!{\xcaph@(0)}
[ld]!{\xcaph@(0)}
[uuu]!{\xcaph@(0)=<}
[ld]!{\xcaph@(0)=<}
[ld]!{\xcaph@(0)=>}
[ld]!{\xcaph@(0)=>}
[uu]!{\xcapv}
[uu]!{\hcap[3]}
}\hspace{-5mm} .
\]
Then, by the replacement relation $(1)$, we have
\[
\nabla_{++}+\nabla_{--}=(t_at_b+t_a^{-1}t_b^{-1})\nabla_{00}.
\]
Since $L_{--}=W_1$ and $L_{00}=\xygraph{
!{0;/r1.0pc/:}
[u]!{\hcap[-2]<}
!{\vover}
!{\vover-}
[uur]!{\hcap[2]>}
}$, we obtain
\[
\nabla_{W_{1}}=-(t_at_b+t_a^{-1}t_b^{-1}).
\]
Let
\[
L_{++}:=\hspace{-5mm}
\xygraph{
!{0;/r1.0pc/:}
[u]!{\hcap[-3]}
[d]!{\xcapv-}
[uu]!{\xunderh}
[ldd]!{\xunderh}
[uu]!{\xcaph@(0)}
[ld]!{\xoverh}
[ldd]!{\xcaph@(0)}
[uuu]!{\xunderv=>}
[udd]!{\xunderv=>}
[ruu]!{\xcapv}
[uu]!{\hcap[3]}
}\hspace{-5mm}, \ 
L_{--}:=\hspace{-5mm}
\xygraph{
!{0;/r1.0pc/:}
[u]!{\hcap[-3]}
[d]!{\xcapv-}
[uu]!{\xunderh}
[ldd]!{\xunderh}
[uu]!{\xcaph@(0)}
[ld]!{\xoverh}
[ldd]!{\xcaph@(0)}
[uuu]!{\xoverv=>}
[udd]!{\xoverv=>}
[ruu]!{\xcapv}
[uu]!{\hcap[3]}
}\hspace{-5mm}, \ 
%\]
%
%\[
L_{00}:=\hspace{-5mm}
\xygraph{
!{0;/r1.pc/:}
[u]!{\hcap[-3]}
[d]!{\xcapv-}
[uu]!{\xunderh}
[ldd]!{\xunderh}
[uu]!{\xcaph@(0)}
[ld]!{\xoverh}
[ldd]!{\xcaph@(0)}
[uuu]!{\xcaph@(0)=<}
[ld]!{\xcaph@(0)=>}
[ld]!{\xcaph@(0)=<}
[ld]!{\xcaph@(0)=>}
[uu]!{\xcapv}
[uu]!{\hcap[3]}
}\hspace{-5mm}. 
\]
Then, by the replacement relation $(2)$, we have
\[
\nabla_{++}+\nabla_{--}=(t_at_b^{-1}+t_a^{-1}t_b)\nabla_{00}.
\]
Since $L_{++}=W_1,$ $L_{--}=W_2,$ $L_{00}=\hspace{-3mm}\xygraph{
!{0;/r1.0pc/:}
[u]!{\hcap[-2]<}
!{\vover}
!{\vover-}
[uur]!{\hcap[2]>}
}\hspace{-3mm}$, we have
\begin{gather*}
\nabla_{W_1}+\nabla_{W_2}=-(t_at_b^{-1}+{t_a}^{-1}t_b), \text{\ i.e.,}\\ 
\nabla_{W_2}=t_at_b-t_at_b^{-1}-t_a^{-1}t_b+t_a^{-1}t_b^{-1}.
\end{gather*} 
Iterating these arguments, we obtain
\begin{gather*}
\nabla_{W_{2m+1}}=-\nabla_{W_{2m}}-(t_at_b+t_a^{-1}t_b^{-1}),\\
\nabla_{W_{2m}}=-\nabla_{W_{2m-1}}-(t_at_b^{-1}+t_a^{-1}t_b).
\end{gather*} 
By the induction hypotheses, we have
\begin{align*}
\nabla_{W_{2m+1}}&=-m(t_at_b-t_at_b^{-1}-t_a^{-1}t_b+t_a^{-1}t_b^{-1})-(t_at_b+t_a^{-1}t_b^{-1})\\
&=-(m+1)(t_at_b+t_a^{-1}t_b^{-1})+m(t_at_b^{-1}+t_a^{-1}t_b)),\\ 
\nabla_{W_{2m}}&= m(t_at_b+t_a^{-1}t_b^{-1})-(m-1)(t_at_b^{-1}+t_a^{-1}t_b))-(t_at_b^{-1}+t_a^{-1}t_b)\\
&=m(t_at_b-t_at_b^{-1}-t_a^{-1}t_b+t_a^{-1}t_b^{-1}).
\end{align*} 
By \Cref{AlexanderConway}, we must have
\begin{gather*}
\Delta_{W_{2m+1}}(x,y)=\Delta_{W_{2m+1}}(t_a^2, t_b^2)=(m+1)(t_a^2t_b^2+1)-m(t_a^2+t_b^2),\\
\Delta_{W_{2m}}(x,y)=\Delta_{W_{2m}}(t_a^2, t_b^2)=m(t_a^2t_b^2-t_a^2-t_b^2+1). \qedhere
\end{gather*} 
\end{proof}

\section{Examples}
In this section, we exhibit several concrete examples to reinforce our result. 
In \Cref{ss.Zp2}, by using Porti's result (\Cref{2260}), we obtain explicit formulas of $p$-torsion growth in the $\Zp^{\,2}$-covers of the twisted Whitehead link $W_{2p^\kappa}$ with $\kappa \in \Z_{\geq 0}$ and Magen David $L=6^2_1$. 
In \Cref{ss.TLN}, by using Mizusawa--Kadokami's result \cite{KadokamiMizusawa2008}, we obtain an explicit formula for the TLN-$\Zp$-cover of $W_{2p^\kappa}$, and give a remark on prescribed $\mu$ and $\nu$. 
In \Cref{ss.plim}, we calculate the $p$-adic limit of the torsion size in the $\Zp^{\,2}$-cover of $L=6^2_1$, in a view of \cite{UekiYoshizaki-plimits}. 
In \Cref{ss.Rolfsen}, we demonstrate some calculations of nontrivial $\lambda$'s and place the table of $\mu$ and $\lambda$ of all links in the Rolfsen table. 

\subsection{$\Zp^{\,2}$-covers} \label{ss.Zp2} 
Here, we exhibit two examples of \Cref{thm.integral} with $c=d=2$. 
\begin{example}[Twisted Whitehead links] \label{11257} 
Let $p$ be any prime number and ${\kappa}\in \Z_{\geq 0}$. 
Then branched $\Zp^{\,2}$-cover $(M_{p^n}\to S^3)_n$ over $(S^3,W_{2p^{\kappa}})$ satisfies the following: 
\begin{align*}
e(H_1(M_{p^n}))=\left({\kappa}p^n+2n-2{\kappa}\right)p^n-2n+{\kappa}. 
\end{align*}
Hence we have $\mu=\kappa$, $\lambda=2$, $\mu_1=-2\kappa$, $\lambda_1=-2$, $\nu=\kappa$. 
In particular, for arbitrary ${\kappa}\in \Z_{\geq0}$, there exists a 2-component link $L$ such that 
the $\Zp^{\,2}$-cover satisfies $\mu_L=\kappa$.
\end{example}

\begin{proof} Since $\Delta_{W_{2p^{\kappa}}}(x,y)=p^{\kappa}(1+xy-x-y)$, we have 
\begin{align*}
|H_1(M_{p^n})|&=\prod_{1\neq\zeta_2\in W(n)}\prod_{1\neq\zeta_1\in W(n)}\Delta(\zeta_1,\zeta_2)\\
&=\prod_{\zeta_2}\prod_{\zeta_1} p^{\kappa}(1-\zeta_1)(1-\zeta_2)\\
&=\prod_{\zeta_2}(p^{\kappa})^{p^n-1}(1-\zeta_2)^{p^n-1}(1-\zeta_1)(1-\zeta_1^2)\cdots(1-\zeta_1^{p^n-1})\\
&=\prod_{\zeta_2}(p^{\kappa})^{p^n-1}(1- \zeta_2)^{p^n-1}p^n\\
&=(p^{\kappa})^{(p^n-1)(p^n-1)}(p^n)^{p^n-1}(p^n)^{p^n-1}\\
&=p^{{\kappa}(p^{2n}-2p^n+1)+2np^n-2n}\\
&=p^{({\kappa}p^n+2n-2{\kappa})p^n-2n+{\kappa}}. \qedhere 
\end{align*}
\end{proof} 

\begin{example}[``Magen David'' $L=6_1^2$] \label{eg.L612} 
The link $L=6_1^2$ is defined by the diagram 
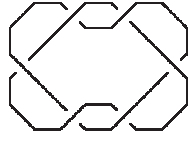
\[
\xygraph{
!{0;/r0.8pc/:}
!{\zbendv@(0)}
!{\xcaph@(0)}
!{\xunderh}
!{\xcaph@(0)}
[ld]!{\xcaph@(0)}
[u]!{\xunderh}
!{\xcaph@(0)}
!{\zbendh@(0)}
[lllllll]!{\xcapv@(0)}
[r]!{\zbendv@(0)}
[rrr]!{\zbendh@(0)}
[ur]!{\xcapv@(0)}
[lllllll]!{\xunderv}
[urrrrrr]!{\xunderv}
[llllll]!{\xcapv@(0)}
[ur]!{\zbendh@(0)}
[rrr]!{\zbendv@(0)}
[r]!{\xcapv@(0)}
[lllllll]!{\zbendh@(0)}
!{\xcaph@(0)}
[u]!{\xunderh}
!{\xcaph@(0)}
[ld]!{\xcaph@(0)}
[u]!{\xunderh}
[d]!{\xcaph@(0)}
!{\zbendv@(0)}
}
\] 

\noindent 
and its Alexander polynomial is $\Delta_L(x,y)=x^2y^2+xy+1$ (cf.~\cite{Rolfsen1990}). 
If $p\neq 3$, then the branched $\Zp^{\,2}$-cover over $(S^3,L)$ satisfies \[|H_1(M_{p^n})|=3^{p^n-1},\]
and its Iwasawa invariants are zero. If $p=3$, then $H_1(M_{p^n})$'s are infinite groups for $n>1$.  
\end{example}

For each $n \in \Z_{\geq 1}$, let $\bm{\Phi}_n(x)\in \Z[x]$ denote the $n$-th cyclotomic polynomial, i.e., \glssymbol{cyclopoly}
\[
\bm{\Phi}_n(x)=\prod_{1\leq \kappa\leq n,\ {\rm gcd}(\kappa,n)=1}(x-e^{2\pi i\frac{\kappa}{n}}) .
\]
We utilize the following result. 

\begin{lem}[Apostol, {\cite[Theorem 1]{Apostol1970pams}}]\label{11252}
If $m>n>1$ and $(m,n)>1$, then we have
\[
\Res(\bm{\Phi}_m(x), \bm{\Phi}_n(x))=
\begin{cases}
l^{\bm{\varphi}(n)}&\frac{m}{n}=l^e\\
1& \mbox{otherwise},
\end{cases}
\]
where $l$ is a prime number, \glssymbol{resultant}${\rm Res}(f,g)$ denotes the resultant of two polynomials $f,g\in \Z[x]$, and \glssymbol{euler}$\bm{\varphi}(n)$ is Euler's totient function.
\end{lem}

\begin{proof}[Proof of \Cref{eg.L612}] 
Note that \Cref{2260} yields the equivalences 
\begin{align*}
p\neq 3 & \iff  \Delta_L(x,y)\neq 0 \text{\ on\ }(W\setminus \{1\})^2\\
& \iff H_1(M_{p^n})<\infty \text{\ for\ }n>1. 
\end{align*}  
Put $r_{p^n}(x):=\Res(y^{p^n}-1,\Delta_L(x,y))$. Then we have
\begin{center}
$\hspace{-20mm} \ds r_{p^n}(x)=\prod_{\xi\in W(n)}(x^2\xi^2+x\xi+1)
=\prod (x\xi-\omega)(x\xi-\omega^2)$\\ 
\hspace{20mm} $\ds =\prod (x-\frac{\omega}{\xi})(x-\frac{\omega}{\xi})
=\prod (x-\xi\omega)(x-\xi\omega^2)
=\prod_{\kappa=0}^{n} \bm{\Phi}_{3p^{\kappa}}(x),$ 
\end{center} 
where $\omega$ is a primitive cube root of unity. By \Cref{11252}, we have
\begin{center}
$\ds \Res(x^{p^n}-1, r_{p^n}(x))
=\prod_{l, \kappa}\Res(\bm{\Phi}_{p^l}(x),\bm{\Phi}_{3\cdot p^{\kappa}}(x))
=\prod_{l=\kappa}\Res(\bm{\Phi}_{p^l}(x),\bm{\Phi}_{3\cdot p^l}(x))$\\ 
$\ds 
=\prod_{l=0}^n 3^{\varphi(p^l)}=3^{1+\sum_{l=1}^n p^{l-1}(p-1)}
=3^{p^n}.$
\end{center}
On the other hand, we have $\Delta_L(1,1)=3$.  
By \Cref{11252}, we also have
\begin{gather*}
\Res(y^{p^n}-1,\Delta_L(1,y))=\Res(y^{p^n}-1,\bm{\Phi}_3(y))=\Res(\bm{\Phi}_1(y),\bm{\Phi}_3(y))=3,\\ 
\Res(x^{p^n}-1,\Delta_L(x,1))=3.
\end{gather*}
Therefore, we have
\[
|H_1(M_{p^n})|=\prod_{1\neq\zeta_1, \zeta_2\in W(n)}\Delta_L(\zeta_1,\zeta_2)=3^{p^n-1}. \qedhere 
\]
\end{proof} 

\begin{remark} \label{rem.L412}
For instance, the $\Zp^{\,2}$-cover of ``Solomon's knot'' $L=4_1^2$ with $\Delta_L(x,y)=xy-1$ does not satisfy the rank assumptions of neither Theorems \ref{2252},  \ref{thm.branched}, nor \ref{thm.integral}. 
In order to obtain the Iwasawa-type formula for the unramified case, we need a finer investigation of the ranks of Iwasawa modules. 
In order to obtain the formula for the branched case, we also need some essential development in the techniques comparing branched and unbranched covers.
\end{remark} 

\subsection{The TLN-covers and prescribed Iwasawa invariants} \label{ss.TLN} 
Here, we observe an example with $(c,d)=(2,1)$ and give a remark on prescribed Iwasawa invariants. 

Let $M$ be a $\Q$HS$^3$ and let $L=\sqcup_i K_i$ be a link in $M$ consisting of null-homologous components. Put $X:=M\setminus {\rm Int}V_L$. 
\emph{The total linking number cover} (TLN-cover) over $(M,L)$ is the $\Z$-cover that corresponds to the surjective homomorphism $\tau:\pi_1(X)\surj \Z$ sending the positive meridian $\alpha_i$ of every $L_i$ to 1. 
Consider the reduced Alexander polynomial \glssymbol{Alexandertilde}$\widetilde{\Delta}_L(t)=\Delta_L(t,\ldots, t)$ and a polynomial $A_{X_\infty}(t)=(t-1)\widetilde {\Delta}_L(t)$. 

\begin{prop}[Kadokami--Mizusawa, {\cite[Theorem 3.3]{KadokamiMizusawa2008}}] \label{prop.KM} 
Let the setting be as above, and let $(M_{p^n}\to M)_n$ denote the branched $\Zp$-cover derived from the TLN-cover. Then  
\[
\frac{|H_1(M_{p^n})|}{\# H_1(M)}=\Big|\prod_{1\neq \xi\in W(n)}A_{X_\infty}(\xi)\Big|.
\]
\end{prop}

\begin{example} \label{11255}
Let $L=W_{2p^{\kappa}}$ in $S^3$ and let $(M_{p^n}\to S^3)_n$ denote the branched $\Zp$-cover derived from the TLN-cover. 
Then we have \glssymbol{AXinf}
\begin{gather*}
A_{X_{\infty}}(t)=(t-1)\widetilde{\Delta}_L(t)
=(t-1)\Delta_L(t,t)
=p^{\kappa}(t-1)^3,\\ 
e(H_1(M_{p^n}))={\kappa}p^n+3n-{\kappa}.
\end{gather*}
In particular, for an arbitrary ${\kappa}\in \Z_{\geq 0}$, there exists a 2-component link $L$ in $S^3$ such that the standard branched $\Zp$-cover satisfies $\mu=\kappa$.
\end{example} 

\begin{proof} 
By \Cref{prop.KM}, we have 
\begin{center}
$\hspace{-10mm} \ds |H_1(M_{p^n})|=\frac{|H_1(M_{p^n})|}{\# H_1(M)}
=\Big|\prod_{1\neq\xi\in W(n)}A_{X_\infty}(\xi)\Big|$\\ 
\hspace{40mm} 
$\ds =\prod_{\xi\neq 1} p^{\kappa}(1-\xi)^3
=(p^{\kappa})^{p^n-1}(p^n)^3
=p^{{\kappa}p^n+3n-{\kappa}}.$ \qedhere
\end{center} \end{proof}

\begin{remark} \label{rem.nu} 
On the number theory side, Iwasawa gave a construction of $\Zp$-extension with arbitrarily large $\mu$ \cite[Theorem 1]{Iwasawa1973}, and Ozaki gave that with arbitrary prescribed $\mu$ and with $\lambda=0$ \cite[Theorem 2]{Ozaki2004JMSJ}. 
On the topology side, the second author pointed out that an analogue of Iwasawa's result may be proved in a parallel way to the original argument \cite[Theorem 5.2]{Ueki3}, while Kadokami--Mizusawa \cite[Proposition 4.15]{KadokamiMizusawa2008} gave another construction of $\Zp$-covers with arbitrarily prescribed $(\mu,\lambda)$ and with $\nu=0$. In addition, we may construct a $\Zp$-cover with arbitrarily large $\nu$ over a base space with small torsion \cite[Proposition 6.16]{UekiYoshizaki-plimits}. 

By \Cref{11257} and \Cref{11255}, both when $d=1,2$, for arbitrary ${\kappa}\in \Z_{\geq 0}$, there exists a 2-component link $L$ in $S^3$ such that the $\Zp^{\,d}$-cover satisfies $\mu_L=\kappa$.

\Cref{11255} indicates that the $\nu$ of a $\Zp$-cover of a link in $S^3$ can be an arbitrary negative integer. 
Also, \Cref{11257} indicates that the $\nu$ of a $\Zp^{\,2}$-cover of a link in $S^3$ can be an arbitrary positive integer. 
\end{remark}

\subsection{$p$-adic torsions} \label{ss.plim} 
Here we recall the context of $p$-adic torsions and present an example in a $\Zp^{\,2}$-cover.
Sinnott, Han, and Kisilevsky proved that in a $\Zp$-extension of a global field, the class numbers $p$-adically converge (\cite[Theorem 4]{Han1991}, \cite[Corollary 1]{Kisilevsky1997PJM}), 
and Ozaki extended their result to a general pro-$p$ extension \cite{Ozaki-padiclimit}. 
Yoshizaki and the second author proved an analogous result for $\Zp$-covers of 3-manifolds: 
\begin{prop}[Ueki--Yoshizaki, a part of {\cite[Theorem 3.1]{UekiYoshizaki-plimits}}] 
Let $(X_{p^n}\to X)_n$ be a compatible system of $\Z/p^n\Z$-covers of a compact $3$-manifold $X$. 
Then the sizes of the non-$p$ torsion subgroups $H_1(X_{p^n})_{\text{non-}p}$ converges in $\Zp$. 
\end{prop} 

If we apply this for the exterior of a link $L$ in $S^3$, then we obtain a similar result for branched $\Z/p^n\Z$-covers $(M_{p^n}\to S^3)_n$ of $L$. We may obtain the result also by combining the explicit formula for the $p$-adic limit of the cyclic resultants \cite[Theorem 5.7]{UekiYoshizaki-plimits} and the results of Mayberry--Murasugi \cite{MayberryMurasugi1982} and Porti \cite{Porti2004}. 

The $p$-adic limit of $|H_1(M_{p^n})_{\text{non-}p}|$ coincides with Kionke's $p$-adic torsion by \cite[Theorem 1.1]{Kionke2020JLMS}. 
Kionke's framework is for arbitrary pro-$p$ covers and the $p$-adic limits of $|H_1(M_{p^n})_{\text{non-}p}|$ in $\Zp^{\,d}$-covers of links also give examples of the $p$-adic torsions. Here we present an example for the case $c=d=2$. 

\begin{example} \label{eg.pAdicTorsion} 
Let $L=6^2_1$ in $S^3$ and let $p\neq 3$. Then, by \Cref{eg.L612}, 
the $\Zp^{\,2}$-cover $(M_{p^n}\to S^3)_n$ satisfies 
\[\lim_{n\to \infty} |H_1(M_{p^n})_{\text{non-}p}|=\lim_{n\to \infty}3^{p^n-1}=\xi/3,\]
where $\xi$ denotes the unique root of unity of order prime to $p$ satisfying $|\xi-3|_p<1$, that is, $\xi \equiv 3$ mod $p$ holds. 
We have $\xi/3\in \Q$ if and only if $p=2$, and in this case we have $\xi/3=1/3$. 
\end{example} 

\begin{proof} We give two proofs. (1) Fermat's small theorem yields $3^{p^n}\equiv 3$ mod $p$, and hence $\xi^{p^n}=\xi$. 
By \cite[Lemma 5.6 (1)]{UekiYoshizaki-plimits}, $\ds \lim_{n\to \infty} (3^{p^n}-\xi^{p^n})=0$ in $\Zp$. Thus we obtain the assertion. 

(2) Let us apply \cite[Theorem 5.7]{UekiYoshizaki-plimits} to verify the consistency. 
If $p\neq 2,3$, then 
\begin{align*}\ds \lim_{n\to \infty} 3^{p^n-1} 
&= \ds \lim_{n\to \infty} \big({\rm sgn}(1-3)\cdot{\rm Res}(t^{p^n}-1, t-3)+1\big)/3 \\
&=\big((-1)(-1)^p(\xi-1)+1\big)/3 \ %&
=\ \xi/3.
\end{align*} 
If $p=2$, then 
\begin{align*}
\ds \lim_{n\to \infty} 3^{p^n-1}
&=\ds \lim_{n\to \infty} \big({\rm sgn}((1-3)(-1-3))\cdot {\rm Res}(t^{p^n}-1, t-3)+1\big)/3\\
&=\big((-1)^2(-1)^2(\xi-1)+1\big)/3 \ =\ \xi/3. \qedhere%\\ & 
\end{align*} 
\end{proof}

\subsection{$\mu$ and $\lambda$-invariants for Rolfsen's table} \label{ss.Rolfsen}

We place below the table of the Alexander--Iwasawa polynomial $\Delta_L(1+T_1,\ldots)$ and their Iwasawa $\mu$ and $\lambda$-invariants in the sense of \Cref{def.mulambda} 
of all links $L$ in Rolfsen table \cite{Rolfsen1990},  which we calculated by using SageMath \cite{SageMath}. 
 
Here we also exhibit examples of detailed calculations of non-trivial $\lambda$'s:  

\begin{example} \label{eg.table} 
(1) 
If $L=6^2_1$ and $p=3$, then we have $\lambda=2$ by 
\begin{align*}
\Delta_L(1+X,1+Y)&=X^2Y^2 + 2X^2Y + 2XY^2 + X^2 + 5XY + Y^2 + 3X + 3Y + 3\\
&\equiv((1+X)(1+Y)-1)^2 \mod 3.
\end{align*}

(2) 
If $L=6^3_3$, then, for any $p$, we have $\lambda=1$ by 
\begin{align*}
\Delta_L(1+X,1+Y,1+Z)&=-XYZ-XY-XZ-YZ-X-Y-Z\\
&=-((1+X)(1+Y)(1+Z)-1).
\end{align*}

(3) 
If $L=8^4_3$, then, for any $p$, we have $\lambda=2$ by 
\begin{align*}
&\hspace{-1cm} \Delta_L(1+X,1+Y,1+Z,1+W)\\
&=WXYZ+WXY+WXZ+WYZ+XYZ+WY+XY+WZ+XZ\\
&=((1+X)(1+W)-1)((1+Y)(1+Z)-1).
\end{align*}

In the following table, for instance, ``1 if $p=2$'' implies ``$0$ if $p\neq 2$,'' if not otherwise mentioned. 
If $\Delta_L\neq 0$ on $(W\setminus \{1\})^c$, then, by Theorems \ref{thm.integral} and \ref{thm.integral.Delta} (1), the $\mu$ and $\lambda$ of $\Delta_L$ coincides with those of the $\Zp^{\,c}$-covers. 
\end{example} 

{\tiny
\begin{longtable}{|c|l|c|c|} 
\caption{\normalsize $\mu$ and $\lambda$-invariants} \label{table.mulambda}
\endfirsthead 
\hline 
link & Alexander--Iwasawa polynomial $\Delta_L(1+X, 1+Y,\ldots)$ & $\mu$ & $\lambda$ \\ \hline \hline 
\endhead 

\hline %表のフッタ 
\endfoot %最後のフッタ
\endlastfoot %このあとは表の中身の記述

\hline 
link & Alexander--Iwasawa polynomial $\Delta_L(1+X, 1+Y, \ldots)$ & $\mu$ & $\lambda$ \\ \hline 
\hline
$0^2_1$ & 0&-&-\\

\hline 

$2^2_1$ & 1&0&0\\ 

\hline 

$4^2_1$&$XY + X + Y + 2$&0&1 if $p=2$\\
\hline

$5^2_1$&$XY$&0&2\\
\hline

$6^2_1$&$X^2Y^2 + 2X^2Y + 2XY^2 + X^2 + 5XY + Y^2 + 3X + 3Y + 3$&0&2 if $p=3$\\
\hline

$6^2_2$&$X^2Y + XY^2 + X^2 + 3XY + Y^2 + 3X + 3Y + 3$&0&0\\
\hline

$6^2_3$&$2XY + X + Y + 2$&0&0\\
\hline

$7^2_1$&$X^2Y^2 + X^2Y + XY^2 + XY - X - Y - 1$&0&0\\
\hline

$7^2_2$&$X^2Y^2 + X^2Y + XY^2 + 3XY + X + Y + 1$&0&0\\
\hline

$7^2_3$&$2XY$&1 if $p=2$&2\\
\hline

\multirow{2}{*}{$7^2_4$}&\multirow{2}{*}{$X^3Y + 2X^2Y + 2XY$}& \multirow{2}{*}{0}&4 if $p=2$\\* 
&&&2 if not\\
\hline

$7^2_5$&$X^3Y + X^3 + X^2Y + 3X^2 + XY + 3X + Y + 2$&0&1 if $p=2$\\
\hline

$7^2_6$&$X^3Y + X^2Y + XY$&0&2\\
\hline

$7^2_7$&$X^3Y + X^3 + 3X^2Y + 3X^2 + 3XY + 3X + Y + 2$&0&1 if $p=2$\\
\hline

$7^2_8$&$XY$&0&2\\
\hline

\multirow{2}{*}{$8^2_1$}&$X^3Y^3 + 3X^3Y^2 + 3X^2Y^3 + 3X^3Y + 9X^2Y^2 + 3XY^3 + 2X^3 + 10X^2Y + 10XY^2$&\multirow{2}{*}{0}&\multirow{2}{*}{0}\\* 
&$+ Y^3 + 7X^2 + 13XY + 4Y^2 + 9X + 6Y + 4$&&\\
\hline

\multirow{2}{*} {$8^2_2$}&$X^3Y + X^2Y^2 + XY^3 + X^3 + 4X^2Y + 4XY^2 + Y^3 + 4X^2 + 7XY $&\multirow{2}{*}{0}&\multirow{2}{*}{0}\\* 
&$+ 4Y^2 + 6X + 6Y + 4$&&\\
\hline

$8^2_3$&$2X^2Y^2 + 3X^2Y + 3XY^2 + X^2 + 7XY + Y^2 + 3X + 3Y + 3$&0&0\\
\hline

\multirow{2}{*}{$8^2_4$}&$X^3Y^2 + X^2Y^3 + 2X^3Y + 4X^2Y^2 + 2XY^3 + X^3 + 7X^2Y + 7XY^2 + Y^3$& \multirow{2}{*}{0}& \multirow{2}{*}{3 if $p=2$}\\* 
&$+ 4X^2 + 10XY + 4Y^2 + 6X + 6Y + 4$&&\\
\hline

$8^2_5$&$X^2Y^2 - X^2 - XY - Y^2 - 3X - 3Y - 3$&0&0\\
\hline

$8^2_6$&$3XY + X + Y + 2$&0&1 if $p=2$\\
\hline

$8^2_7$&$X^2Y^2 - XY - X - Y - 1$&0&0\\
\hline

$8^2_8$&$X^2Y^2 + XY + X + Y + 1$&0&0\\
\hline

$8^2_9$&$-X^3 - 2X^2Y - X^2 + 3X + Y + 2$&0&1 if $p=2$\\
\hline

$8^2_{10}$&$X^3Y$&0&4\\
\hline

$8^2_{11}$&$-X^3Y + X^3 - X^2Y + 3X^2 - XY + 3X + Y + 2$&0&1 if $p=2$\\
\hline

$8^2_{12}$&$X^3Y$&0&4\\
\hline

$8^2_{13}$&$X^3Y - X^2Y - XY$&0&2\\
\hline

$8^2_{14}$&$X^3Y + X^3 - X^2Y + 3X^2 - XY + 3X + Y + 2$&0&1 if $p=2$\\
\hline

$8^2_{15}$&$XY$&0&2\\
\hline

$8^2_{16}$&$-X^3 - X^2 + 2XY + 3X + Y + 2$&0&1 if $p=2$\\
\hline

\multirow{2}{*} {$9^2_1$}&$X^3Y^3 + 2X^3Y^2 + 2X^2Y^3 + X^3Y + 4X^2Y^2 + XY^3 + X^2Y + XY^2 - X^2$& \multirow{2}{*} {0}& \multirow{2}{*} {0}\\* 
&$ - 2XY - Y^2 - 3X - 3Y - 2 $&&\\
\hline

$9^2_2$&$X^3Y + X^2Y^2 + XY^3 + 2X^2Y + 2XY^2 - X^2 - XY - Y^2 - 3X - 3Y -2$& 0& 0\\
\hline

$9^2_3$&$2X^2Y^2 + 2X^2Y + 2XY^2 + 3XY - X - Y - 1$&0&0\\
\hline

$9^2_4$&$X^3Y^2 + X^2Y^3 + X^3Y + 5X^2Y^2 + XY^3 + 5X^2Y + 5XY^2 + 5XY$&0&2\\
\hline

\multirow{2}{*}{$9^2_5$}&\multirow{2}{*}{$X^3Y + 2X^2Y^2 + XY^3 + 4X^2Y + 4XY^2 + 4XY$}&\multirow{2}{*}{0}&4 if $p=2$\\* 
&&&2 if not\\
\hline

\multirow{2}{*}{$9^2_6$}&$-X^3Y^2 - X^2Y^3 - X^3Y - 3X^2Y^2 - XY^3 - 2X^2Y - 2XY^2 + X^2 + XY$& \multirow{2}{*}{0}& \multirow{2}{*}{1 if $p=2$}\\
&$+ Y^2 + 3X + 3Y + 2$&&\\
\hline

\multirow{2}{*} {$9^2_7$}&$-X^3Y^2 - X^2Y^3 - X^3Y - 2X^2Y^2 - XY^3 - X^2Y - XY^2 + X^2 + 2XY $& \multirow{2}{*}{0}&\multirow{2}{*}{1 if $p=2$}\\
&$+ Y^2 +3X + 3Y + 2$&&\\
\hline

$9^2_8$&$2X^2Y + 2XY^2 + 3XY - X - Y - 1$&0&0\\
\hline

$9^2_9$&$X^3Y^2 - X^2Y^3 + X^3Y - 3X^2Y^2 - XY^3 - 3X^2Y - 3XY^2 - 2X^2 - 3XY- 2X$& 0&2\\ 
\hline

$9^2_{10}$&$3XY$&1 if $p=3$&2\\
\hline

$9^2_{11}$&$2X^2Y^2 + X^2Y + XY^2 + XY - X - Y - 1$&0&0\\
\hline

$9^2_{12}$&$X^2Y^2 - X^2Y - XY^2 - XY + X + Y + 1$&0&0\\
\hline

\multirow{2}{*}{$9^2_{13}$}&\multirow{2}{*}{$X^5Y + 4X^4Y + 7X^3Y + 6X^2Y + 3XY$}&\multirow{2}{*}{0}&4 if $p=3$\\* 
&&&2 if not\\
\hline

$9^2_{14}$&$X^5 + 2X^4Y + 5X^4 + 6X^3Y + 10X^3 + 8X^2Y + 10X^2 + 4XY + 5X + Y + 2$&0& 1 if $p=2$\\
\hline

\multirow{2}{*}{ $9^2_{15}$}& \multirow{2}{*}{ $2X^3Y + 3X^2Y + 3XY$}& \multirow{2}{*}{0}&4 if $p=3$\\*
&&&2 if not\\
\hline

$9^2_{16}$&$2X^3 + 4X^2Y + 5X^2 + 3XY + 3X + Y + 2$&0&1 if $p=2$\\
\hline

$9^2_{17}$&$2X^3 + 3X^2Y + 5X^2 + 2XY + 3X + Y + 2$&0&1 if $p=2$\\
\hline

$9^2_{18}$&$2X^3Y + 2X^2Y + 2XY$&1 if $p=2$&2\\
\hline

$9^2_{19}$&$X^4Y + X^3Y^2 + 5X^3Y + 2X^2Y^2 + X^3 + 8X^2Y + 2XY^2 + 2X^2 + 6XY+ 2X + Y + 1$& 0&0\\
\hline 

\multirow{2}{*} {$9^2_{20}$}&$X^4 + 2X^3Y + 2X^2Y^2 + 4X^3 + 7X^2Y + 2XY^2 + 7X^2 + 6XY + Y^2$&\multirow{2}{*}{0}& \multirow{2}{*}{0}\\* 
&$ + 6X + 3Y + 3$&&\\
\hline

\multirow{2}{*} {$9^2_{21}$}&$X^4Y^2 + X^4Y + 3X^3Y^2 + 5X^3Y + 4X^2Y^2 + X^3 + 8X^2Y + 2XY^2 + 2X^2$& \multirow{2}{*}{0}& \multirow{2}{*}{0}\\* 
&$ + 6XY + 2X + Y + 1$&&\\
\hline

\multirow{2}{*}{$9^2_{22}$}&$X^4Y + 2X^3Y^2 + X^4 + 5X^3Y + 4X^2Y^2 + 4X^3 + 10X^2Y + 3XY^2 + 7X^2 $& \multirow{2}{*}{0}& \multirow{2}{*}{0}\\* 
&$+ 8XY + Y^2 + 6X + 3Y + 3$&&\\
\hline

$9^2_{23}$&$X^3Y + 2X^2Y^2 + XY^3 + 3X^2Y + 3XY^2 - X^2 - Y^2 - 3X - 3Y - 2$&0&1 if $p=2$\\
\hline

$9^2_{24}$&$3X^2Y^2 + 3X^2Y + 3XY^2 + X^2 + 7XY + Y^2 + 3X + 3Y + 3$&0&2 if $p=3$\\
\hline

\multirow{2}{*}{$9^2_{25}$}&\multirow{2}{*}{$X^3Y - 2X^2Y - 2XY$}&\multirow{2}{*}{0}&4 if $p=2$\\* 
&&&2 if not\\
\hline

$9^2_{26}$&$X^3Y - X^3 - X^2Y - X^2 + XY + 3X + Y + 2$&0&1 if $p=2$\\
\hline

\multirow{2}{*}{$9^2_{27}$}&\multirow{2}{*}{$2X^3Y + 3X^2Y + 3XY$}& \multirow{2}{*}{0}&4 if $p=3$\\* 
&&&2 if not\\
\hline

$9^2_{28}$&$X^3Y - X^3 - X^2Y - 3X^2 - XY - 3X - Y - 2$&0&1 if $p=2$\\
\hline

$9^2_{29}$&$-2X^4Y - 5X^3Y - 6X^2Y - 3XY + 1$&0&0\\
\hline

$9^2_{30}$&$-X^3Y + 2X^3 + X^2Y + 5X^2 + 3X + Y + 2$&0&0\\
\hline

$9^2_{31}$&$X^5Y + 3X^4Y + 4X^3Y + 2X^2Y + XY$&0&2\\
\hline

$9^2_{32}$&$2X^3Y + X^2Y + XY$&0&2\\
\hline

$9^2_{33}$&$2X^3Y + X^2Y + XY$&0&2\\
\hline

$9^2_{34}$&$X^4Y^2 + X^4Y + 2X^3Y^2 + 3X^3Y + 2X^2Y^2 + 2X^2Y + XY^2 - X^2 - 2X- Y - 1$& 0&0\\
\hline

$9^2_{35}$&$X^4Y^2 + X^4Y + 2X^3Y^2 + 3X^3Y + 2X^2Y^2 + 4X^2Y + XY^2 + X^2 + 4XY+ 2X + Y + 1$&0& 0\\
\hline

$9^2_{36}$&$2X^3Y + 2X^2Y + 2XY$&1 if $p=2$&2\\
\hline

\multirow{2}{*}{$9^2_{37}$}&\multirow{2}{*}{$X^5Y + 3X^4Y + 5X^3Y + 4X^2Y + 2XY$}&0&4 if $p=2$\\* 
&&&2 if not\\
\hline

$9^2_{38}$&$2X^3Y + X^2Y + X^2 + 2XY + 3X + Y + 2$&0&1 if $p=2$\\
\hline

\multirow{2}{*} {$9^2_{39}$}&$-X^4Y - 2X^3Y^2 - X^4 - 4X^3Y - 4X^2Y^2 - 3X^3 - 7X^2Y - 3XY^2 - 4X^2 $& \multirow{2}{*}{0}& \multirow{2}{*}{0}\\* 
&$- 4XY - Y^2 - 2X - Y$&&\\
\hline

\multirow{2}{*}{$9^2_{40}$}&$X^4Y^2 + 2X^4Y + X^3Y^2 + X^4 + 5X^3Y + 4X^3 + 4X^2Y + XY^2 + 7X^2$&\multirow{2}{*}{0}&\multirow{2}{*}{2 if $p=3$}\\* 
&$+ 4XY + Y^2 + 6X + 3Y + 3$&&\\
\hline

\multirow{2}{*}{$9^2_{41}$}& \multirow{2}{*}{$X^3Y^3 + 2X^3Y^2 + X^2Y^3 + X^3Y + 5X^2Y^2 + 3X^2Y + 3XY^2 + 3XY$}&\multirow{2}{*}{0}&4 if $p=3$\\* 
&&&2 if not\\
\hline

$9^2_{42}$&$X^4Y^2 + X^4Y + X^3Y^2 + 2X^3Y - X^2Y - X^2 - 2XY - 2X - Y- 1$&0&0\\ 
\hline 

$9^2_{43}$&$X^5 + 5X^4 + 10X^3 + 10X^2 + 5X + Y + 2$&0&1 if $p=2$\\
\hline

\multirow{2}{*}{$9^2_{44}$}& \multirow{2}{*}{$X^3Y + 2X^2Y + 2XY$}& \multirow{2}{*}{0}&4 if $p=2$\\* 
&&&2 if not\\
\hline

$9^2_{45}$&$2X^3 + 5X^2 - XY + 3X + Y + 2$&0&1 if $p=2$\\
\hline

$9^2_{46}$&$2XY$&1 if $p=2$&2\\
\hline

$9^2_{47}$&$XY$&0&2\\
\hline

$9^2_{48}$&$2X^3 + X^2Y + 5X^2 + 3X + Y + 2$&0&1 if $p=2$\\
\hline

$9^2_{49}$&$X^4 + 4X^3 + X^2Y + 7X^2 + 2XY + Y^2 + 6X + 3Y + 3$&0& 2 if $p=3$\\
\hline

$9^2_{50}$&$-X^2Y - XY + Y^2 + X + Y + 1$&0&0\\
\hline

$9^2_{51}$&$X^4Y + X^4 + 3X^3Y + 4X^3 + 4X^2Y + XY^2 + 7X^2 + 4XY + Y^2 + 6X + 3Y + 3$&0&0\\ 
\hline 

$9^2_{52}$&$X^2Y^2 + X^2Y + XY - Y^2 - X - Y - 1$&0&0\\
\hline

$9^2_{53}$&$X^2Y^2 + X^3 + 2X^2Y + 2XY^2 + Y^3 + 4X^2 + 5XY + 4Y^2 + 6X + 6Y + 4$&0& 2 if $p=2$\\
\hline

$9^2_{54}$&$X^2Y + XY^2 + XY - X - Y - 1$&0&0\\
\hline

$9^2_{55}$&$X^3Y + X^2Y + XY$&0&2\\
\hline

$9^2_{56}$&$X^3Y + X^2Y + XY$&0&2\\
\hline

$9^2_{57}$&$X^3Y + 2X^2Y + X^2 + 3XY + 3X + Y + 2$&0&0\\
\hline

$9^2_{58}$&$X^3Y + X^2Y + X^2 + 2XY + 3X + Y + 2$&0&0\\
\hline

$9^2_{59}$&$X^5Y + X^5 + 5X^4Y + 5X^4 + 9X^3Y + 10X^3 + 8X^2Y + 10X^2 + 4XY+5X + Y + 2$&0&0\\ 
\hline

$9^2_{60}$&$X^3Y + 2X^3 + 2X^2Y + 5X^2 + XY + 3X + Y + 2$&0&0\\
\hline

\multirow{2}{*}{$9^2_{61}$}&$X^3Y^2 + 2X^3Y + 3X^2Y^2 + XY^3 + X^3 + 6X^2Y + 6XY^2 + Y^3 + 4X^2$& \multirow{2}{*}{0}& \multirow{2}{*}{2 if $p=2$}\\* 
&$+ 9XY + 4Y^2 + 6X + 6Y + 4$&&\\
\hline \hline 

$6^3_1$&$XY + XZ + YZ + X + Y + Z$&0&0\\
\hline

$6^3_2$&$-XYZ$&0&3\\
\hline

$6^3_3$&$-XYZ - XY - XZ - YZ - X - Y - Z$&0&1\\
\hline

$7^3_1$&$-XYZ + YZ - X + Y + Z$&0&0\\
\hline

\multirow{2}{*}{ $8^3_1$}&
$-X^2Y^2Z - 2X^2YZ - XY^2Z - X^2Y + XY^2 - X^2Z - 3XYZ - X^2 + Y^2$& \multirow{2}{*}{0}& \multirow{2}{*}{0}\\* 
&$- 2XZ - 2X + 2Y - Z$&&\\
\hline

\multirow{2}{*}{ $8^3_2$}&
$X^2Y^2 + X^2YZ + XY^2Z + 2X^2Y + 2XY^2 + X^2Z + 2XYZ + Y^2Z + X^2$& \multirow{2}{*}{0}& \multirow{2}{*}{0}\\*
&$ + 4XY + Y^2 + 2XZ + 2YZ + 2X + 2Y + Z$&&\\
\hline

$8^3_3$&$XYZ - XY - XZ - YZ - X - Y - Z$&0&1 if $p=2$\\
\hline

$8^3_4$&$X^3Y + 2X^2YZ + X^3 + 3X^2Y + X^2Z + 2XYZ + 4X^2 + 2XY + 2XZ+ 4X$&0&3 if $p=2$\\
\hline

\multirow{2}{*}{$8^3_5$}&
\multirow{2}{*}{$X^2Y^2Z + X^2YZ + XY^2Z + 2XYZ$}& \multirow{2}{*}{0}&4 if $p=2$\\* 
&&&3 if not\\
\hline

$8^3_6$&
$X^2Y^2Z + X^2YZ + XY^2Z + 3XYZ + XZ + YZ + Z$&0&1\\
\hline

\multirow{2}{*}{ $8^3_7$}&
$-X^2Y^2Z - X^2Y^2 - 2X^2YZ - 2XY^2Z - 2X^2Y - 2XY^2 - X^2Z - 4XYZ - Y^2Z$& \multirow{2}{*}{0}& \multirow{2}{*}{1}\\*  
&$ - X^2 - 4XY - Y^2 - 2XZ - 2YZ - 2X - 2Y - Z$&&\\
\hline

$8^3_8$&
$XY^2Z - X^2Y + XY^2 + XYZ + Y^2Z - X^2 + Y^2 + 2YZ - 2X + 2Y + Z$&0& 0\\
\hline

$8^3_9$&$XYZ$&0&3\\
\hline

\multirow{2}{*}{$8^3_{10}$}&$X^3YZ + X^3Y + X^3Z + 3X^2YZ + X^3 + 3X^2Y + 3X^2Z + 2XYZ + 4X^2$& \multirow{2}{*}{0}& 3 if $p=2$\\* 
&$ + 2XY + 2XZ + 4X$&&1 if not\\
\hline

$9^3_1$&
$X^2Y^2Z + X^2YZ + X^2Y - XY^2 - XZ - YZ - Z$&0&0\\
\hline

$9^3_2$&
$X^2Y^2Z + X^2YZ + X^2Y - XY^2 + 2XYZ + XZ + YZ + Z$&0&0\\
\hline

\multirow{2}{*}{ $9^3_3$}&$-X^3YZ - X^2YZ - X^3 + X^2Y + X^2Z - XYZ - X^2 + XY + XZ$& \multirow{2}{*}{0}& \multirow{2}{*}{0}\\* 
&$ + YZ - X + Y + Z$&&\\
\hline

\multirow{2}{*}{ $9^3_4$}&$X^3Y + X^3Z + 2X^2YZ + X^3 + 2X^2Y + 2X^2Z + 2XYZ + X^2 + 2XY$&\multirow{2}{*}{0}& \multirow{2}{*}{0}\\* 
&$ + 2XZ + YZ + X + Y + Z$&&\\
\hline

$9^3_5$&
$X^2YZ + XY^2Z + 2XYZ - Y^2Z + X^2 - Y^2 - 2YZ + 2X - 2Y - Z$& 0&0\\
\hline 

$9^3_6$&
$X^2Y^2Z + X^2YZ + X^2Y - XY^2 + XYZ - Y^2Z + X^2 - Y^2 - 2YZ+ 2X - 2Y - Z$& 0&0\\
\hline 

$9^3_7$&$2XYZ - YZ + X - Y - Z$&0&1 if $p=2$\\
\hline

\multirow{2}{*}{ $9^3_8$}& \multirow{2}{*}{ $X^3Y + 2X^2YZ + 2X^2Y + XYZ$}&\multirow{2}{*}{0}&3 if $p= 2,3$\\*
&&&2 if not\\
\hline

$9^3_9$&$X^3YZ + X^2YZ + XYZ$&0&3\\
\hline

$9^3_{10}$& $X^2Y^2Z$&0&5\\
\hline

$9^3_{11}$&
$X^2Y^2Z - 8XYZ - Y^2Z + X^2 - Y^2 - 8XZ - 10YZ + 2X - 2Y - 9Z$&0&0\\ 
\hline

$9^3_{12}$&$X^3YZ$&0&5\\
\hline

$9^3_{13}$&
$X^2Y^2Z + X^2Y^2 + X^2YZ + XY^2Z + X^2Y + XY^2 - XZ - YZ - Z$&0&0\\
\hline

$9^3_{14}$&
$X^2Y^2Z + X^2Y^2 + X^2YZ + XY^2Z + X^2Y + XY^2 + 2XYZ + XZ + YZ + Z$& 0&0\\ 
\hline 

 $9^3_{15}$&$X^3Y + X^3Z + X^3 + 2X^2Y + 2X^2Z + X^2 + 2XY + 2XZ + YZ + X + Y + Z$& 0&0\\
\hline 

$9^3_{16}$&$-2X^2YZ + X^3 - X^2Y - X^2Z - 2XYZ + X^2 - XY - XZ - YZ+ X - Y - Z$& 0&0\\
\hline 

$9^3_{17}$&$-X^3 + X^2Y + X^2Z - X^2 + XY + XZ + YZ - X + Y  + Z$& 0& 0\\
\hline 

$9^3_{18}$&$-XYZ$&0&3\\
\hline

$9^3_{19}$&$-X^3YZ - X^3Y - 2X^2YZ - 2X^2Y - XYZ$&0&3\\
\hline

\multirow{2}{*}{$9^3_{20}$}&\multirow{2}{*}{$-X^3Z - X^2YZ - 2X^2Z$}& \multirow{2}{*}{0}&4 if $p=2$\\* 
&&&3 if not\\
\hline \hline 

$8^4_1$&$-WXY - WXZ - WYZ - XYZ - WY - XY - WZ - XZ$&0&0\\
\hline

$8^4_2$&$-WXZ - WYZ - WY + XY - WZ - XZ$&0&0\\
\hline

$8^4_3$&$WXYZ + WXY + WXZ + WYZ + XYZ + WY + XY + WZ + XZ$&0&2\\
\hline

$8^4_4$&$WXYZ - WXZ - WYZ - WY + XY - WZ - XZ$&0&0\\
\hline

\end{longtable} 
}

\ \\ 

\printnoidxglossary[type=symbols, style=formel_altlong4colheader, title={Glossary of Symbols}]

\bibliographystyle{amsplain} %amsalpha} %plain}% jabbrv} 
\bibliography{TatenoUeki.bbl} 
%/Users/uekijun/Dropbox/refs1} 

\end{document}